\documentclass[10pt,reqno,a4paper]{amsart}
\usepackage{a4wide}
\usepackage{amsmath}
\usepackage{amssymb}
\usepackage{amsthm}
\usepackage{amstext}
\usepackage{geometry}
\usepackage{fancyhdr}
\usepackage{array}
\usepackage{graphicx}
\usepackage{hyperref}
\usepackage{fullpage}
\usepackage{cite}
\usepackage{bm}
\allowdisplaybreaks
\hypersetup{
colorlinks=true,
linkcolor=black,
urlcolor=red,
citecolor=green,
}

\allowdisplaybreaks[4]

\newtheorem{theorem}{Theorem}[section]

\newtheorem{corollary}[theorem]{Corollary}

\newtheorem{lemma}[theorem]{Lemma}
\newtheorem{proposition}[theorem]{Proposition}
\newtheorem{remark}[theorem]{Remark}

\numberwithin{equation}{section}

\begin{document}

\title{Uniform parameter bounds for weighted exponential equations on simply connected domains}

\author{Hong-Ge Chen}
\address{[Hong-Ge Chen] School of Mathematics and Statistics, Key Laboratory of Nonlinear Analysis Applications (Ministry of Education), Central China Normal University, Wuhan 430079, P.R. China}
\email{hongge\_chen@whu.edu.cn}

\author{Peng Luo}
\address{[Peng Luo] School of Mathematics and Statistics, Key Laboratory of Nonlinear Analysis Applications (Ministry of Education), Central China Normal University, Wuhan 430079, P.R. China}
\email{pluo@ccnu.edu.cn}

\author{Shuangjie Peng}
\address{[Shuangjie Peng] School of Mathematics and Statistics, Key Laboratory of Nonlinear Analysis Applications (Ministry of Education), Central China Normal University, Wuhan 430079, P.R. China}
\email{sjpeng@mail.ccnu.edu.cn}

\author{Fulin Zhong}
\address{[Fulin Zhong] School of Mathematics and Statistics, Central China Normal University, Wuhan 430079, P.R. China}
\email{flzhong@mails.ccnu.edu.cn}


\date{\today}

\begin{abstract}
We consider positive solutions of
\begin{equation*}
\begin{cases}
-\Delta u=p\lambda h(x)u^{p-1}\mathrm{e}^{u^p}&\text{in }\Omega,\\
u>0&\text{in }\Omega,\\
u=0&\text{on }\partial\Omega,
\end{cases}
\end{equation*}
where $p\in[1,2]$, $\lambda>0$, $\Omega\subset\mathbb{R}^2$ is a smooth bounded simply connected domain and $h\in C^2(\overline\Omega)$ is positive on $\overline\Omega$. The cases $p=1$ and $p=2$ correspond to the Liouville and Moser-Trudinger equations, respectively. We prove that the normalized parameter
\begin{equation*}
\beta=\frac{\lambda p^2}{2}\left(\int_\Omega h(\mathrm{e}^{u^p}-1)\mathrm{d}x\right)^{\frac{2-p}{p}}\left(\int_\Omega hu^p\mathrm{e}^{u^p}\mathrm{d}x\right)^{\frac{2(p-1)}{p}}
\end{equation*}
is bounded by a constant depending only on $\Omega$ and $h$, uniformly over all $p\in[1,2]$, $\lambda>0$ and positive solutions. At $p=2$, this excludes positive constrained critical points of the weighted Moser-Trudinger functional at sufficiently large Dirichlet energies. For $p=2$ and $h\equiv1$, we further prove that the energy converges uniformly to $4\pi$ as $\lambda\to0$ when $\Omega$ is convex. We also develop a boundary exclusion argument for positive weights that rules out boundary blow-up for $1<p<2$ without  a priori energy bound.

\vspace{0.25cm}


\vspace{0.25cm}

\end{abstract}

\maketitle

\tableofcontents

\section{Introduction} \label{sec-introduction}

We consider the family of equations
\begin{equation} \label{1}
\begin{cases}
-\Delta u=p\lambda h(x)u^{p-1}\mathrm{e}^{u^p}&\text{in }\Omega,\\
u>0&\text{in }\Omega,\\
u=0&\text{on }\partial\Omega,
\end{cases}
\end{equation}
where $p\in[1,2]$, $\lambda>0$, $\Omega\subset\mathbb{R}^2$ is a smooth bounded simply connected domain and  $h\in C^2(\overline\Omega)$ is positive on $\overline\Omega$. Following \cite{MalchiodiMartinazziThizy2024}, we use the normalized parameter
\begin{equation} \label{beta-definition}
\beta=\frac{\lambda p^2}{2}\left(\int_\Omega h(\mathrm{e}^{u^p}-1)\mathrm{d}x\right)^{\frac{2-p}{p}}\left(\int_\Omega hu^p\mathrm{e}^{u^p}\mathrm{d}x\right)^{\frac{2(p-1)}{p}}.
\end{equation}
Throughout the paper, a classical solution belongs to $C^2(\Omega)\cap C^1(\overline\Omega)$.

\vskip 0.1cm

For $p=1$, equation \eqref{1} reduces to the weighted Liouville equation
\begin{equation*}
\begin{cases}
-\Delta u=\lambda h(x)\mathrm{e}^u&\text{in }\Omega,\\
u>0&\text{in }\Omega,\\
u=0&\text{on }\partial\Omega.
\end{cases}
\end{equation*}
For a priori estimates and blow-up analysis, see \cite{BrezisMerle1991,BrezisLiShafrir1993,LiShafrir1994,Li1999,ChenLin1998,ChenLin2002}. For related results on mean field equations, one can refer to  \cite{CagliotiLionsMarchioroPulvirenti1992,CagliotiLionsMarchioroPulvirenti1995,DingJostLiWang1999,Lin2000,ChenLin2003,BartolucciLin2014}.

\vskip 0.1cm

For $p=2$, equation \eqref{1} becomes the weighted Moser-Trudinger equation
\begin{equation*}
\begin{cases}
-\Delta u=2\lambda h(x)u\mathrm{e}^{u^2}&\text{in }\Omega,\\
u>0&\text{in }\Omega,\\
u=0&\text{on }\partial\Omega.
\end{cases}
\end{equation*}
For $v\in H_0^1(\Omega)$, we denote its Dirichlet energy by
\begin{equation*}
E(v)=\int_\Omega|\nabla v|^2\mathrm{d}x.
\end{equation*}
Testing the equation with $u$ yields
\begin{equation*}
\beta=2\lambda\int_\Omega hu^2\mathrm{e}^{u^2}\mathrm{d}x=E(u).
\end{equation*}
Thus the normalized parameter is precisely the Dirichlet energy in this case. The corresponding variational problem is to find positive critical points of $J$ constrained to $\mathcal{M}_\beta$, where
\begin{equation*}
J(v)=\int_\Omega h(\mathrm{e}^{v^2}-1)\mathrm{d}x,\quad
\mathcal{M}_\beta=\left\{v\in H_0^1(\Omega)\mid E(v)=\beta\right\}.
\end{equation*}
The sharp Moser-Trudinger inequality \cite{Trudinger1967,Moser1971} states that
\begin{equation*}
\sup_{v\in\mathcal{M}_\beta}J(v)<+\infty
\quad\text{if and only if}\quad
\beta\leq4\pi.
\end{equation*}

We first recall the case $h\equiv1$. For $0<\beta<4\pi$, the direct variational method yields maximizers. At $\beta=4\pi$, compactness may fail. Critical extremals were obtained on the disk by Carleson and Chang \cite{CarlesonChang1986}, and on more general bounded planar domains by Struwe \cite{Struwe1988} and Flucher \cite{Flucher1992}. Lions \cite{Lions1985} developed the concentration-compactness framework for the limiting embedding. For further compactness and blow-up results, see \cite{AdimurthiStruwe2000,AdimurthiDruet2004}. For energies above $4\pi$, Struwe \cite{Struwe1988} proved, under an additional assumption on $\Omega$, that there exists $\varepsilon_0>0$ such that a positive local maximizer of $J|_{\mathcal{M}_\beta}$ exists for every $\beta\in(4\pi,4\pi+\varepsilon_0)$, and for almost every $\beta$ in this interval, a second positive critical point of mountain-pass type exists. Later, Lamm, Robert and Struwe \cite{LammRobertStruwe2009} proved that, for any smooth bounded planar domain $\Omega$, there exists $\varepsilon_0>0$ such that $J|_{\mathcal{M}_\beta}$ has two distinct positive critical points for every $\beta\in(4\pi,4\pi+\varepsilon_0)$.

\vskip 0.1cm

For $p=2$ and $h\equiv1$, Malchiodi and Martinazzi \cite{MalchiodiMartinazzi2014} proved that $J|_{\mathcal{M}_\beta}$ has no positive critical point on the disk for sufficiently large $\beta$. On non-simply connected domains, Del Pino, Musso and Ruf \cite{delPinoMussoRuf2012} constructed positive constrained critical points with two peaks at energies slightly above $8\pi$. On annuli, they also constructed such critical points with $k$ peaks at energies slightly above $4\pi k$ for every $k\geq1$. For $p=2$ and general positive weights $h\in C^2(\overline\Omega)$, Malchiodi, Martinazzi and Thizy \cite{MalchiodiMartinazziThizy2024} proved that $J|_{\mathcal{M}_\beta}$ has a positive critical point for every $\beta>0$ whenever $\Omega$ is a smooth bounded non-contractible planar domain.

\vskip 0.1cm

For the weighted equation at $p=2$, Druet and Thizy \cite{DruetThizy2020} established energy quantization and characterized the location of concentration points for sequences of positive solutions with bounded Dirichlet energy.
For $h\equiv1$ and every fixed $p\in[1,2)$, Malchiodi, Martinazzi and Thizy explain after \cite[Theorem 1.2]{MalchiodiMartinazziThizy2024} that \eqref{1} has no positive solution for sufficiently large $\beta$ on a smooth bounded simply connected domain. They also asked whether the Dirichlet energies of positive solutions are bounded from above at $p=2$ \cite[Open problem]{MalchiodiMartinazziThizy2024}. Recently, Hu, Wu and Zou \cite{HuWuZou2026} proved this bound for $p=2$ and $h\equiv1$. The following theorem treats general positive weights and establishes a bound uniform in $p\in[1,2]$.

\begin{theorem} \label{thm-simply-connected}
Let $\Omega\subset\mathbb{R}^2$ be a smooth bounded simply connected domain and let $h\in C^2(\overline\Omega)$ be positive on $\overline\Omega$. There exists a constant $C_{\Omega,h}>0$ such that every positive solution of \eqref{1}, with $p\in[1,2]$ and $\lambda>0$, satisfies
\begin{equation*}
\beta\leq C_{\Omega,h}.
\end{equation*}
The constant depends only on $\Omega$, $\|h\|_{C^2(\overline\Omega)}$ and $(\min_{\overline\Omega}h)^{-1}$.
\end{theorem}

\begin{remark}
The assumption in Theorem \ref{thm-simply-connected} that $\Omega$ is simply connected cannot be omitted, as shown in Remark \ref{rem-annulus}.
\end{remark}

For $p=2$, Theorem \ref{thm-simply-connected} yields the following variational consequence. When $h\equiv1$, the corollary extends the large energy nonexistence result of \cite[Theorem 1]{MalchiodiMartinazzi2014} from the disk to all smooth bounded simply connected planar domains.

\begin{corollary} \label{cor-bound}
Let $\Omega\subset\mathbb{R}^2$ be a smooth bounded simply connected domain and let $h\in C^2(\overline\Omega)$ be positive on $\overline\Omega$. There exists a constant $\beta_*>4\pi$, depending only on $\Omega$ and $h$, such that $J|_{\mathcal{M}_\beta}$ has no positive critical point whenever $\beta>\beta_*$.
\end{corollary}

We also consider the limit $\lambda\to0$ on convex domains for $h\equiv1$ and $p=2$. On the disk, Malchiodi and Martinazzi \cite[Theorem 2]{MalchiodiMartinazzi2014} proved that the energy tends to $4\pi$ as the maximum of a positive solution tends to infinity. Mancini and Martinazzi \cite{ManciniMartinazzi2017} obtained upper and lower bounds for the energy of positive radial critical points. For related semilinear equations, Grossi and Takahashi \cite{GrossiTakahashi2010} proved the nonexistence of multi-bubble solutions on convex domains. We prove the following uniform limit.

\begin{theorem} \label{thm-convex}
Let $\Omega\subset\mathbb{R}^2$ be a smooth bounded convex domain. For every $\varepsilon>0$ there exists a constant $\lambda_\varepsilon>0$, depending only on $\Omega$ and $\varepsilon$, such that every positive solution of \eqref{1} with $h\equiv1$, $p=2$ and $0<\lambda<\lambda_\varepsilon$ satisfies
\begin{equation*}
|E(u)-4\pi|<\varepsilon.
\end{equation*}
\end{theorem}

For a positive solution $u$ of \eqref{1}, set
\begin{equation*}
\begin{aligned}
&f=p\lambda h(x)u^{p-1}\mathrm{e}^{u^p},\quad F(s)=\lambda(\mathrm{e}^{s^p}-1),\quad
M=\int_\Omega f\mathrm{d}x,\quad m=\max_{x\in\overline{\Omega}}u(x),\quad L=\log\frac{1}{\lambda}.
\end{aligned}
\end{equation*}
Thus $hF'(u)=f$. We call $f$ the source and $M$ the total source mass. For a sequence of solutions $u_n$ with exponents $p_n$ and parameters $\lambda_n$, we use the corresponding notation $f_n$, $F_n$, $M_n$, $m_n$, $L_n$ and $\beta_n$.

\vskip 0.1cm

The proof of Theorem \ref{thm-simply-connected} uses blow-up analysis first at a global maximum of the solution and then at the maximum on a superlevel component with the largest source mass. We argue by contradiction and assume that $\beta_n\to\infty$. The proof consists of the following four steps.

\smallskip
\noindent\textbf{Step 1.} A conformal Pohozaev argument following \cite{Battaglia2019} yields $M\leq C$ for all $\lambda>0$ and
\begin{equation*}
M\leq CL^{-\frac{p-1}{p}}
\end{equation*}
for sufficiently small $\lambda>0$, uniformly in $p\in[1,2]$. Under the assumption $\beta_n\to\infty$, the integral inequality for superlevel components implies $M_nm_n^{p_n-1}\to\infty$. Using the source estimates, we also obtain $m_n\to\infty$ and $\frac{L_n}{m_n^{p_n}}\to0$. We perform the first blow-up analysis at a global maximum, using $u_n\leq m_n$, and prove that $p_n\to2$. The radial comparison of \cite{DeMarchisMalchiodiMartinazziThizy2022} then yields $\lambda_n\to0$ and a circle mean estimate at a fixed radius. For $s>0$, set
\begin{equation*}
\Theta_n(s)=\lambda_ns^{2(p_n-1)}\mathrm{e}^{\frac{2-p_n}{2}s^{p_n}}.
\end{equation*}
The Green representation yields a lower bound $cM_n$ for the circle mean at the same fixed radius. Combining this lower bound with the upper bound from radial comparison, we obtain
\begin{equation*}
\Theta_n(m_n)\leq C\mathrm{e}^{-cM_nm_n^{p_n-1}}\to0.
\end{equation*}

\smallskip
\noindent\textbf{Step 2.} We choose a high superlevel component $D_n$ with maximal source mass. The source and integral estimates imply
\begin{equation*}
q_n=\int_{D_n}f_n\mathrm{d}x\geq cM_n.
\end{equation*}
The components have small diameters but need not be separated. We use the maximality of $q_n$ and a truncated radial vector field to bound the high source outside $D_n$ with a distance weight and establish
\begin{equation*}
\frac{2\lambda_n}{q_n^2}\int_{D_n}h\mathrm{e}^{u_n^{p_n}}\mathrm{d}x\to\frac{1}{4\pi}.
\end{equation*}
The exponential integrability of normalized Green potentials, combined with this limit, yields
\begin{equation*}
M_n^{\frac{p_n}{p_n-1}}L_n\to0.
\end{equation*}

\smallskip
\noindent\textbf{Step 3.} The stronger source bound allows us to estimate the Green potential generated by $f_n$ outside $D_n$ and its oscillation near $D_n$. Combining these estimates with the integral estimate on $D_n$, we prove that
\begin{equation*}
\frac{p_n}{2}\int_{D_n}u_n^{p_n-1}f_n\mathrm{d}x\to4\pi.
\end{equation*}
The same limit holds on a suitable shrinking neighborhood of $D_n$. Choose $x_{1,n}\in D_n$ with $\gamma_n=u_n(x_{1,n})=\max_{\overline{D_n}}u_n$. The oscillation estimate provides a local upper bound near $x_{1,n}$, while the stronger source bound implies $\frac{L_n}{\gamma_n^{p_n}}\to0$. Blow-up analysis at $x_{1,n}$ and the integral estimates on $D_n$ then yield
\begin{equation*}
\frac{p_n}{2}\gamma_n^{p_n-1}q_n\to4\pi.
\end{equation*}
The local profile and the integral limit on the shrinking neighborhood exclude a second bubble through a point selection argument. The resulting pointwise estimate and the Green representation yield the gradient bound needed for radial comparison. We also verify the required radius condition.

\smallskip
\noindent\textbf{Step 4.} Within the range of radial comparison, we choose radii $0<r_n<R_n$ such that $D_n\subset B_{r_n}(x_{1,n})$ and
\begin{equation*}
\log\frac{R_n}{r_n}\geq cM_n^{-\frac{p_n}{p_n-1}},\quad
0<\log\frac{1}{R_n}\leq CM_n^{-\frac{p_n}{p_n-1}}.
\end{equation*}
The circle mean identity and radial comparison imply
\begin{equation*}
\frac{p_n}{2}\gamma_n^{p_n-1}q_n\leq4\pi+CM_n^{\frac{p_n}{p_n-1}}.
\end{equation*}
Comparing the radial estimate with the Green representation on $\partial B_{R_n}(x_{1,n})$ and using these bounds, we obtain $-\log\Theta_n(\gamma_n)\leq C$. Since $\gamma_n\leq m_n$ and $\Theta_n$ is nondecreasing, we obtain the contradiction
\begin{equation*}
0<c\leq\Theta_n(\gamma_n)\leq\Theta_n(m_n)\to0.
\end{equation*}

For convex domains with $h\equiv1$ and $p=2$, the uniform energy bound allows us to apply \cite{DruetThizy2020}. The nonexistence result in \cite{GrossiTakahashi2010} applies to the location equations and excludes multiple concentration points. Since $m\to\infty$ as $\lambda\to0$, there exists exactly one concentration point. Energy quantization then implies $E(u)\to4\pi$ and proves Theorem \ref{thm-convex}.

\vskip 0.1cm

The paper is organized as follows. Section \ref{sec-source} bounds the total source mass and records the Green function estimates. Section \ref{sec-cores} establishes estimates for superlevel components and the maximum of the solution. Section \ref{sec-maximal} selects a component with maximal source mass and improves the source bound. Section \ref{sec-background} estimates the Green potential generated by the source outside this component and studies blow-up at the component maximum. Section \ref{sec-final} completes the proof of Theorem \ref{thm-simply-connected}. Section \ref{app-remaining} contains the proofs of Theorem \ref{thm-convex} and Corollary \ref{cor-bound}. Appendix \ref{app-boundary} proves the boundary exclusion used in Section \ref{sec-cores}.

\section{Source and Green estimates} \label{sec-source}

In this section, we use a conformal Pohozaev identity to bound the total source mass. For $p>1$, we also estimate the Green potential generated by the source where $u$ is below a suitable level.

\vskip 0.1cm

Fix a Riemann map $\Phi\colon B_1(0) \to\Omega$. Since $\Omega$ is smooth,  we know from
\cite{Bell2016} that the map $\Phi$ extends to a $C^\infty$ diffeomorphism from $\overline{B_1(0)}$ onto $\overline\Omega$, with $\Phi'\neq0$ in $\overline{B_1(0)}$. Consequently, there exist constants $c,C>0$ such that
\begin{equation*}
0<c\leq|\Phi'|\leq C,\quad |\Phi''|\leq C \quad\text{in }\overline{B_1(0)}.
\end{equation*}
Set $v=u\circ\Phi$ and $a=|\Phi'|^2$, and write $\Phi=(\Phi_1,\Phi_2)$. The Cauchy-Riemann equations are
\begin{equation*}
\partial_{y_1}\Phi_1=\partial_{y_2}\Phi_2,\quad \partial_{y_2}\Phi_1=-\partial_{y_1}\Phi_2.
\end{equation*}
The chain rule yields
\begin{equation*}
\Delta_yv=a(\Delta_xu)\circ\Phi, \quad
\mathrm{d}x=a(y)\mathrm{d}y.
\end{equation*}

\begin{proposition} \label{pro:bound}
There exists a constant $C>0$, depending only on $\Omega$ and $\|\nabla\log h\|_{L^\infty(\Omega)}$, such that every positive solution of \eqref{1} with $p\in[1,2]$ satisfies
\begin{equation} \label{2}
M^2\leq C\int_\Omega hF(u)\mathrm{d}x.
\end{equation}
Moreover, $M\leq C$ for all $\lambda>0$. There exists $\lambda_0\in(0,1)$, depending only on $\Omega$ and $\|h\|_{L^\infty(\Omega)}$, such that
\begin{equation} \label{3}
M\leq CL^{-\frac{p-1}{p}} \quad \text{for any } 0<\lambda<\lambda_0.
\end{equation}
\end{proposition}

\begin{proof}
Set $v=u\circ\Phi$ and $b=|\Phi'|^2(h\circ\Phi)$. Then $v$ satisfies
\begin{equation} \label{3-1}
\begin{cases}
-\Delta v=bF'(v)&\text{in } B_1(0) ,\\[1mm]
v=0&\text{on }\partial B_1(0) ,
\end{cases}
\end{equation}
and
\begin{equation*}
M=\int_{B_1(0)}bF'(v)\mathrm{d}y=-\int_{\partial B_1(0)}\partial_\nu v\mathrm{d}s,\quad \int_\Omega hF(u)\mathrm{d}x=\int_{B_1(0)}bF(v)\mathrm{d}y.
\end{equation*}
On $\partial B_1(0) $, $y=\nu$, $\nabla v=(\partial_\nu v)\nu$ and $F(v)=0$. Multiplying \eqref{3-1} by $y\cdot\nabla v$ and integrating by parts, we obtain
\begin{equation*}
\begin{aligned}
\int_{B_1(0)}(-\Delta v)(y\cdot\nabla v)\mathrm{d}y &=-\int_{\partial B_1(0)}(\partial_\nu v)^2\mathrm{d}s +\int_{B_1(0)}\nabla v\cdot\nabla(y\cdot\nabla v)\mathrm{d}y\\
&=-\int_{\partial B_1(0)}(\partial_\nu v)^2\mathrm{d}s +\int_{B_1(0)}|\nabla v|^2\mathrm{d}y +\frac{1}{2}\int_{B_1(0)}y\cdot\nabla|\nabla v|^2\mathrm{d}y\\
&=-\int_{\partial B_1(0)}(\partial_\nu v)^2\mathrm{d}s +\frac{1}{2}\int_{\partial B_1(0)}(y\cdot\nu)|\nabla v|^2\mathrm{d}s =-\frac{1}{2}\int_{\partial B_1(0)}(\partial_\nu v)^2\mathrm{d}s,
\end{aligned}
\end{equation*}
where we use that
\begin{equation*}
\operatorname{div}\left(y|\nabla v|^2\right)=2|\nabla v|^2+y\cdot\nabla|\nabla v|^2.
\end{equation*}
Since $\nabla F(v)=F'(v)\nabla v$ and $\operatorname{div}(by)=2b+y\cdot\nabla b$, integration by parts gives
\begin{equation*}
\begin{aligned}
\int_{B_1(0)}bF'(v)(y\cdot\nabla v)\mathrm{d}y&=\int_{\partial B_1(0)}bF(v)(y\cdot\nu)\mathrm{d}s-\int_{B_1(0)}(2b+y\cdot\nabla b)F(v)\mathrm{d}y\\
&=-\int_{B_1(0)}(2b+y\cdot\nabla b)F(v)\mathrm{d}y.
\end{aligned}
\end{equation*}
Thus we obtain the Pohozaev identity (see also \cite{Battaglia2019})
\begin{equation*}
\int_{\partial B_1(0)}(\partial_\nu v)^2\mathrm{d}s =2\int_{B_1(0)}(2b+y\cdot\nabla b)F(v)\mathrm{d}y.
\end{equation*}
Since $b=|\Phi'|^2(h\circ\Phi)$,
\begin{equation*}
\begin{aligned}
\|\nabla\log b\|_{L^\infty(B_1(0))}
&\leq C_\Omega\left(1+\|\nabla\log h\|_{L^\infty(\Omega)}\right),
\end{aligned}
\end{equation*}
and
\begin{equation*}
\begin{aligned}
2b+y\cdot\nabla b&\leq\left(2+\|\nabla\log b\|_{L^\infty(B_1(0))}\right)b\leq Cb.
\end{aligned}
\end{equation*}
Using $F(v)\geq0$ and the Cauchy-Schwarz inequality on $\partial B_1(0) $,
\begin{equation*}
\begin{aligned}
M^2 &\leq2\pi\int_{\partial B_1(0)}(\partial_\nu v)^2\mathrm{d}s =4\pi\int_{B_1(0)}(2b+y\cdot\nabla b)F(v)\mathrm{d}y\\
&\leq C\int_{B_1(0)}bF(v)\mathrm{d}y =C\int_\Omega hF(u)\mathrm{d}x.
\end{aligned}
\end{equation*}
This proves \eqref{2}.

\vskip 0.1cm

For $s>0$, using $1-\mathrm{e}^{-s^p} \leq \min\{s^p,1\}$,
\begin{equation*}
\frac{F(s)}{F'(s)}=\frac{1-\mathrm{e}^{-s^p}}{ps^{p-1}}\leq\begin{cases}
\frac{s}{p}\leq1,&0<s\leq1,\\
\frac{1}{ps^{p-1}}\leq1,&s\geq1.
\end{cases}
\end{equation*}
Thus
\begin{equation*}
\int_\Omega hF(u)\mathrm{d}x\leq \int_\Omega hF'(u)\mathrm{d}x=M,
\end{equation*}
and \eqref{2} gives $M^2\leq CM$. Since $M>0$, we obtain $M\leq C$.

\vskip 0.1cm
For every $T>0$, we have
\begin{equation*}
2hF(u)\leq2\lambda h\mathrm{e}^{T^p} \quad\text{in }\{u\leq T\},
\end{equation*}
and
\begin{equation*}
2hF(u)=2\lambda h(\mathrm{e}^{u^p}-1)\leq\frac{2f}{pu^{p-1}}\leq\frac{2f}{pT^{p-1}} \quad\text{in }\{u>T\}.
\end{equation*}
Consequently,
\begin{equation*}
\begin{aligned}
2\int_\Omega hF(u)\mathrm{d}x &\leq 2\lambda\int_{\{u\leq T\}}h\mathrm{e}^{T^p}\mathrm{d}x +\int_{\{u>T\}}\frac{2f}{pu^{p-1}}\mathrm{d}x \leq2\lambda\|h\|_{L^\infty(\Omega)}|\Omega|\mathrm{e}^{T^p}+\frac{2M}{pT^{p-1}}.
\end{aligned}
\end{equation*}
Using \eqref{2} and Young's inequality, we obtain
\begin{equation*}
M^2\leq C\lambda\|h\|_{L^\infty(\Omega)}\mathrm{e}^{T^p}+\frac{CM}{T^{p-1}}\leq\frac{1}{2}M^2+C\lambda\|h\|_{L^\infty(\Omega)}\mathrm{e}^{T^p}+\frac{C}{T^{2(p-1)}}.
\end{equation*}
Take $T=(\frac{L}{2})^{\frac{1}{p}}$. Choose $\lambda_0\in(0,1)$, depending only on $\|h\|_{L^\infty(\Omega)}$, so that
\begin{equation*}
L\geq2,\quad \|h\|_{L^\infty(\Omega)}\lambda^{\frac{1}{2}}\leq L^{-1}
\quad \text{for } 0<\lambda<\lambda_0.
\end{equation*}
Since $0\leq\frac{2(p-1)}{p}\leq1$, we conclude that
\begin{equation*}
M^2\leq C\left(\|h\|_{L^\infty(\Omega)}\lambda^{\frac{1}{2}}+L^{-\frac{2(p-1)}{p}}\right)\leq CL^{-\frac{2(p-1)}{p}}.
\end{equation*}
\end{proof}

\begin{remark} \label{rem-annulus}
Already for $h\equiv1$ and $p=2$, estimate \eqref{2} fails on the annulus $\mathcal{A}=\{x\in\mathbb{R}^2\mid 1<|x|<2\}$. Introduce $t=\log|x|\in(0,\log2)$. For $\beta>0$, consider
\begin{equation*}
\max_{\substack{v\in H_0^1(0,\log2)\\
2\pi\int_0^{\log2}|v'|^2\mathrm{d}t\leq\beta}}
2\pi\int_0^{\log2}\mathrm{e}^{2t}(\mathrm{e}^{v^2}-1)\mathrm{d}t.
\end{equation*}
The compact Sobolev embedding theorem provides a nonnegative maximizer $v_\beta\not\equiv0$ satisfying
\begin{equation*}
2\pi\int_0^{\log2}|v_\beta'|^2\mathrm{d}t=\beta.
\end{equation*}
There exists a constant $\lambda_\beta>0$ such that
\begin{equation*}
-v_\beta''=2\lambda_\beta\mathrm{e}^{2t}v_\beta\mathrm{e}^{v_\beta^2} \quad\text{in }(0,\log2),\quad v_\beta(0)=v_\beta(\log2)=0.
\end{equation*}
Multiplication by $v_\beta$ and integration give
\begin{equation*}
\lambda_\beta=\frac{\beta}{4\pi\displaystyle\int_0^{\log2}\mathrm{e}^{2t}v_\beta^2
\mathrm{e}^{v_\beta^2}\mathrm{d}t}>0.
\end{equation*}
Thus $v_\beta$ is concave and positive in $(0,\log2)$. Set $u_\beta(x)=v_\beta(\log|x|)$. Then
\begin{equation*}
-\Delta u_\beta(x)=-|x|^{-2}v_\beta''(\log|x|)=2\lambda_\beta u_\beta(x)\mathrm{e}^{u_\beta(x)^2} \quad\text{in }\mathcal{A},
\end{equation*}
and
\begin{equation*}
u_\beta=0\quad\text{on }\partial\mathcal{A},\quad E(u_\beta)=2\pi\int_0^{\log2}|v_\beta'|^2\mathrm{d}t=\beta.
\end{equation*}
Set $f_\beta=2\lambda_\beta u_\beta\mathrm{e}^{u_\beta^2}$ and $M_\beta=\displaystyle\int_{\mathcal{A}}f_\beta\mathrm{d}x$. For every $t\in(0,\log2)$, the Cauchy-Schwarz inequality gives
\begin{equation*}
v_\beta(t)^2\leq t\displaystyle\int_0^t|v_\beta'|^2\mathrm{d}s,\quad v_\beta(t)^2\leq(\log2-t)\displaystyle\int_t^{\log2}|v_\beta'|^2\mathrm{d}s.
\end{equation*}
Therefore, it holds
\begin{equation*}
\frac{\beta}{2\pi}\geq\frac{v_\beta(t)^2}{t}+\frac{v_\beta(t)^2}{\log2-t}\geq\frac{4v_\beta(t)^2}{\log2}.
\end{equation*}
Using $\beta=\displaystyle\int_{\mathcal{A}}u_\beta f_\beta\mathrm{d}x$ and the bound on $\|v_\beta\|_{L^\infty(0,\log2)}$, we obtain
\begin{equation*}
\beta\leq\|v_\beta\|_{L^\infty(0,\log2)}M_\beta,\quad M_\beta\geq\sqrt{\frac{8\pi\beta}{\log2}}.
\end{equation*}
For $s>0$, $\frac{1-\mathrm{e}^{-s^2}}{s}\leq1$. Hence
\begin{equation*}
\lambda_\beta\int_{\mathcal{A}}(\mathrm{e}^{u_\beta^2}-1)\mathrm{d}x
=\frac{1}{2}\int_{\mathcal{A}}\frac{1-\mathrm{e}^{-u_\beta^2}}{u_\beta}f_\beta\mathrm{d}x
\leq\frac{M_\beta}{2}.
\end{equation*}
Consequently,
\begin{equation*}
\frac{M_\beta^2}{\lambda_\beta\int_{\mathcal{A}}(\mathrm{e}^{u_\beta^2}-1)\mathrm{d}x}
\geq2M_\beta\to\infty
\quad\text{as }\beta\to\infty.
\end{equation*}
\end{remark}

The Dirichlet Green function $G$ of $\Omega$ satisfies
\begin{equation*}
\begin{cases}
-\Delta_xG(x,y)=\delta_y&\text{in }\Omega,\\
G(x,y)=0&\text{on }\partial\Omega,
\end{cases}
\end{equation*}
where $\delta_y$ is the unit Dirac mass at $y$. Write
\begin{equation*}
G(x,y)=\frac{1}{2\pi}\log\frac{1}{|x-y|}+H(x,y),\quad (Gg)(x)=\int_\Omega G(x,y)g(y)\mathrm{d}y.
\end{equation*}
Here $H$ is the regular part of $G$, and $Gg$ denotes the Green potential of $g$.

\vskip 0.1cm

The following gradient estimate follows from \cite{DongKim2021}.

\begin{lemma}
There exists a constant $C>0$, depending only on $\Omega$, such that
\begin{equation} \label{6}
|\nabla_xG(x,y)|\leq\frac{C}{|x-y|}, \quad \text{for all $x,y\in\Omega$ with $x\neq y$}.
\end{equation}

\end{lemma}

Set $d_\Omega=2\operatorname{diam}\Omega$. Since $\Omega\subset B_{d_\Omega}(y)$ for every $y\in\Omega$, the domain monotonicity of the Green function yields
\begin{equation} \label{7}
0\leq G_\Omega(x,y)\leq G_{B_{d_\Omega}(y)}(x,y)=\frac{1}{2\pi}\log\frac{d_\Omega}{|x-y|} \quad\text{for }x,y\in\Omega,\ x\neq y.
\end{equation}

We shall use the following consequence of the Brezis-Merle estimate in \cite{BrezisMerle1991}.

\begin{lemma} \label{lem-green-jensen}
Let $g\in L^\infty(\Omega)$ satisfy $g\geq0$ and $\displaystyle\int_\Omega g\mathrm{d}x=1$, and denote $w=Gg$. Then for every $0<\alpha<4\pi$,
\begin{equation} \label{8}
\int_\Omega\mathrm{e}^{\alpha w}\mathrm{d}x \leq C_\alpha,\quad C_\alpha=\frac{4\pi^2d_\Omega^2}{4\pi-\alpha}.
\end{equation}
Moreover, if $r>0$ and $0\leq b<\alpha<4\pi$, then
\begin{equation} \label{9}
\int_\Omega w^r\mathrm{e}^{bw}\mathrm{d}x \leq\left(\frac{r}{\mathrm{e}(\alpha-b)}\right)^r C_\alpha.
\end{equation}
\end{lemma}

\begin{proof}
Using Jensen's inequality and \eqref{7}, we get
\begin{equation*}
\mathrm{e}^{\alpha w(x)} \leq\int_\Omega\mathrm{e}^{\alpha G(x,y)}g(y)\mathrm{d}y \leq\int_\Omega\left(\frac{d_\Omega}{|x-y|}\right)^{\frac{\alpha}{2\pi}}g(y)\mathrm{d}y.
\end{equation*}
Since $\Omega\subset B_{d_\Omega}(y)$,
\begin{equation*}
\begin{aligned}
\int_\Omega\mathrm{e}^{\alpha w(x)}\mathrm{d}x &\leq d_\Omega^{\frac{\alpha}{2\pi}}\int_\Omega g(y) \left(\int_\Omega|x-y|^{-\frac{\alpha}{2\pi}}\mathrm{d}x\right)\mathrm{d}y\\
&\leq2\pi d_\Omega^{\frac{\alpha}{2\pi}}\int_0^{d_\Omega}t^{1-\frac{\alpha}{2\pi}}\mathrm{d}t=\left.\frac{2\pi d_\Omega^{\frac{\alpha}{2\pi}}t^{2-\frac{\alpha}{2\pi}}}{2-\frac{\alpha}{2\pi}}\right|_{t=0}^{t=d_\Omega}=\frac{4\pi^2d_\Omega^2}{4\pi-\alpha}.
\end{aligned}
\end{equation*}
For $t\geq0$,
\begin{equation*}
t^r\mathrm{e}^{bt}=t^r\mathrm{e}^{-(\alpha-b)t}\mathrm{e}^{\alpha t},\quad \max_{t\geq0}t^r\mathrm{e}^{-(\alpha-b)t} =\left(\frac{r}{\mathrm{e}(\alpha-b)}\right)^r.
\end{equation*}
Integration with $t=w(x)$ and \eqref{8} proves \eqref{9}.
\end{proof}

For $p\in[\frac{3}{2},2]$, fix $0<\kappa<4\pi$ and set
\begin{equation*}
t=\left(\frac{\kappa}{M}\right)^{\frac{1}{p-1}},\quad f_0=f\mathbf{1}_{\{u\leq t\}},\quad F_0=hF(\min\{u,t\}).
\end{equation*}
If $\kappa<2\pi$, denote $v_0=Gf_0$.

\begin{lemma} \label{lem-low-source}
There exists a constant $C>0$, depending only on $\Omega$, $\|h\|_{L^\infty(\Omega)}$ and $\kappa$, such that every positive solution of \eqref{1} satisfies
\begin{equation} \label{10}
\int_\Omega f_0\mathrm{d}x\leq C\lambda M^{p-1},\quad \int_{\Omega}uf_0\mathrm{d}x\leq C\lambda M^p, \quad
\frac{p}{2}\int_{\Omega}u^{p-1}f_0\mathrm{d}x\leq C\lambda M^{2p-2},
\end{equation}
and
\begin{equation} \label{11}
\int_\Omega F_0\mathrm{d}x\leq C\lambda M^p.
\end{equation}
If $\kappa<2\pi$, then
\begin{equation} \label{12}
\|v_0\|_{C^1(\overline\Omega)}\leq C\lambda M^{p-1}.
\end{equation}
The constants are independent of $p\in[\frac{3}{2},2]$.
\end{lemma}

\begin{proof}
Set $w=\frac{u}{M}=G\left(\frac{f}{M}\right)$. In $\{u\leq t\}$, we have $u^p\leq t^{p-1}u=\kappa w$. Since $f_0=0$ outside this set, we obtain
\begin{equation*}
\begin{aligned}
f_0&\leq p\lambda\|h\|_{L^\infty(\Omega)}M^{p-1}w^{p-1}\mathrm{e}^{\kappa w},\\
uf_0&\leq p\lambda\|h\|_{L^\infty(\Omega)}M^pw^p\mathrm{e}^{\kappa w},\\
\frac{p}{2}u^{p-1}f_0&\leq\frac{p^2\lambda\|h\|_{L^\infty(\Omega)}}{2}M^{2p-2}w^{2p-2}\mathrm{e}^{\kappa w}.
\end{aligned}
\end{equation*}
Choose $\alpha=\frac{\kappa+4\pi}{2}\in(\kappa,4\pi)$. Since $\displaystyle\int_\Omega\frac{f}{M}\mathrm{d}x=1$ and $p\in[\frac{3}{2},2]$, estimates \eqref{8} and \eqref{9} imply
\begin{equation*}
\begin{aligned}
\int_\Omega\left(w^{p-1}+w^p+w^{2p-2}\right)\mathrm{e}^{\kappa w}\mathrm{d}x
&\leq3\int_\Omega(1+w^2)\mathrm{e}^{\kappa w}\mathrm{d}x
\leq3\left(1+\frac{4}{\mathrm{e}^2(\alpha-\kappa)^2}\right)C_\alpha
\leq C.
\end{aligned}
\end{equation*}
Integrating the three pointwise bounds, we obtain \eqref{10}.

\vskip 0.1cm

Since $\min\{u,t\}\leq u=Mw$ and $\min\{u,t\}\leq t$, we have
\begin{equation*}
\min\{u,t\}^p\leq M^pw^p,\quad
\min\{u,t\}^p\leq ut^{p-1}=\kappa w.
\end{equation*}
Using $\mathrm{e}^{s^p}-1\leq s^p\mathrm{e}^{s^p}$ for $s\geq0$, we obtain
\begin{equation*}
\begin{aligned}
F_0&\leq\lambda h\min\{u,t\}^p\mathrm{e}^{\min\{u,t\}^p}\leq\lambda\|h\|_{L^\infty(\Omega)}M^pw^p\mathrm{e}^{\kappa w},
\end{aligned}
\end{equation*}
which proves \eqref{11}.

\vskip 0.1cm

Finally, if $\kappa<2\pi$, choose $2<r<\frac{4\pi}{\kappa}$ and then $\alpha=\frac{r\kappa+4\pi}{2}\in(r\kappa,4\pi)$.
Since $r(p-1)\in[\frac{r}{2},r]$, we have
\begin{equation*}
\begin{aligned}
\|f_0\|_{L^r(\Omega)}^r
&\leq\left(p\lambda\|h\|_{L^\infty(\Omega)}\right)^rM^{r(p-1)}
\int_\Omega w^{r(p-1)}\mathrm{e}^{r\kappa w}\mathrm{d}x\\
&\leq C\lambda^rM^{r(p-1)}
\int_\Omega(1+w^r)\mathrm{e}^{r\kappa w}\mathrm{d}x
\leq C\lambda^rM^{r(p-1)}.
\end{aligned}
\end{equation*}
The $L^r$ theory and the Sobolev embedding theorem yield
\begin{equation*}
\|v_0\|_{C^1(\overline\Omega)} \leq C\|v_0\|_{W^{2,r}(\Omega)} \leq C\|f_0\|_{L^r(\Omega)}\leq C\lambda M^{p-1}.
\end{equation*}
\end{proof}

\begin{remark}
The height $t$ is chosen so that
\begin{equation*}
u^p\leq t^{p-1}u=\kappa\frac{u}{M}\quad\text{in }\{u\leq t\}.
\end{equation*}
This allows us to estimate $f_0$ and $F_0$ using \eqref{8}. These bounds will be used in Section \ref{sec-maximal}.
\end{remark}

\section{Superlevel components and maximum estimates} \label{sec-cores}

In this section, we establish diameter and integral bounds for superlevel components. We then derive the local blow-up profile and apply radial comparison near a global maximum. These arguments yield the required maximum estimate for sequences with $\beta_n\to\infty$, without assuming an a priori bound on $E(u_n)$.

\subsection{Diameter and integral bounds for superlevel components}\

\vskip .1cm

In this subsection, we estimate the diameters of superlevel components using the Green representation. The isoperimetric inequality then yields an integral bound on each component.

\begin{lemma} \label{lem-diameter}
Let $g\in L^\infty(\Omega)$ satisfy $g\geq0$ and $q=\displaystyle\int_\Omega g\mathrm{d}x>0$, and denote $v=Gg$. Then for every $t>0$, each nonempty connected component $D$ of $\{v>t\}$ satisfies
\begin{equation*}
\operatorname{diam}D\leq2\mathrm{e}d_\Omega\mathrm{e}^{-\frac{2\pi t}{q}},
\end{equation*}
where $d_\Omega=2\operatorname{diam}\Omega$.
\end{lemma}

\begin{proof}
Fix two distinct $x_0,x_1\in D$, and denote $\ell=|x_1-x_0|$ and $e=\frac{x_1-x_0}{\ell}$. Choose a polygonal path $\gamma\colon [0,1]\to D$ joining these points. For $0\leq s\leq\ell$, set
\begin{equation*}
\tau(s)=\min\big\{\tau\in[0,1]\mid e\cdot(\gamma(\tau)-x_0)=s\big\},\quad z(s)=\gamma\big(\tau(s)\big).
\end{equation*}
The function $\tau$ is well defined and nondecreasing. For $y\in\Omega$ and $c=e\cdot(y-x_0)$,
\begin{equation*}
|z(s)-y|\geq|e\cdot(z(s)-y)|=|s-c|.
\end{equation*}
By \eqref{7}, we have
\begin{equation*}
\int_0^\ell G(z(s),y)\mathrm{d}s \leq\frac{1}{2\pi}\int_0^\ell\log\frac{d_\Omega}{|s-c|}\mathrm{d}s.
\end{equation*}
For $c\leq0$, the integral $\displaystyle\int_0^\ell\log\frac{d_\Omega}{|s-c|}\mathrm{d}s$ is increasing in $c$, while for $c\geq\ell$ it is decreasing. For $0<c<\ell$,
\begin{equation*}
\frac{\mathrm{d}}{\mathrm{d}c}\int_0^\ell\log\frac{d_\Omega}{|s-c|}\mathrm{d}s =\log\frac{\ell-c}{c}.
\end{equation*}
Hence the maximum occurs at $c=\frac{\ell}{2}$, with
\begin{equation*}
\begin{aligned}
\int_0^\ell\log\frac{d_\Omega}{|s-c|}\mathrm{d}s &\leq2\int_0^{\frac{\ell}{2}}\log\frac{d_\Omega}{s}\mathrm{d}s=\left.2s\left(\log\frac{d_\Omega}{s}+1\right)\right|_{s=0}^{s=\frac{\ell}{2}}=\ell\left(\log\frac{2d_\Omega}{\ell}+1\right).
\end{aligned}
\end{equation*}
Since $D$ is a component of $\{v>t\}$, we have $v(z(s))>t$ and
\begin{equation*}
\begin{aligned}
t\ell&<\int_0^\ell v(z(s))\mathrm{d}s=\int_\Omega g(y)\int_0^\ell G(z(s),y)\mathrm{d}s\mathrm{d}y \leq\frac{q\ell}{2\pi}\left(\log\frac{2d_\Omega}{\ell}+1\right).
\end{aligned}
\end{equation*}
It follows that $\ell\leq2\mathrm{e}d_\Omega\mathrm{e}^{-\frac{2\pi t}{q}}$. Taking the supremum over $x_0,x_1\in D$ proves the result.
\end{proof}

\begin{lemma} \label{lem-interior}
There exists a constant $\delta_*>0$, depending only on $\Omega$ and $\|\nabla\log h\|_{L^\infty(\Omega)}$, such that every critical point $x$ of every positive solution of \eqref{1} satisfies $\operatorname{dist}(x,\partial\Omega)\geq\delta_*$.
\end{lemma}

\begin{proof}
For $h\equiv1$ and $p=2$, the conclusion follows from Proposition 2.1 in  \cite{MalchiodiMartinazziThizy2024}. The general case is proved in Appendix \ref{app-boundary}. This gives the boundary exclusion for positive weights and $1<p<2$ discussed in Remark 2.6 of \cite{MalchiodiMartinazziThizy2024}.
\end{proof}

Let $t>0$ and let $D$ be a nonempty connected component of $\{u>t\}$. Then
\begin{equation} \label{component-boundary}
\overline D\subset\subset\Omega,\quad u=t\quad\text{on }\partial D.
\end{equation}
Indeed, since $u=0$ on $\partial\Omega$ and $u\in C(\overline{\Omega})$, the set $\{u\geq t\}$ is compactly contained in $\Omega$. Continuity implies $u\geq t$ on $\partial D$. If $x\in\partial D$ and $u(x)>t$, continuity gives a ball $B_r(x)\subset\{u>t\}$. Since $B_r(x)$ meets $D$, the set $D\cup B_r(x)$ is connected, contradicting the maximality of $D$. Thus $u=t$ on $\partial D$.

\begin{remark}
Let $x\in\Omega$ satisfy $\operatorname{dist}(x,\partial\Omega)<\frac{\delta_*}{2}$. For $0<t<u(x)$, let $D$ be the connected component of $\{u>t\}$ containing $x$, and choose $a\in D$ with $u(a)=\max_{\overline D}u$. Since $\nabla u(a)=0$, Lemma \ref{lem-interior} gives $\operatorname{dist}(a,\partial\Omega)\geq\delta_*$. Lemma \ref{lem-diameter} then yields
\begin{equation*}
\frac{\delta_*}{2}\leq|a-x|\leq\operatorname{diam}D\leq2\mathrm{e}d_\Omega\mathrm{e}^{-\frac{2\pi t}{M}}.
\end{equation*}
Taking logarithms and letting $t\to u(x)$, we obtain from Proposition \ref{pro:bound} that
\begin{equation*}
u(x)\leq\frac{M}{2\pi}\log\frac{4\mathrm{e}d_\Omega}{\delta_*}\leq C.
\end{equation*}
Since $C$ and $\delta_*$ are independent of $u$, $p$ and $\lambda$, all positive solutions of \eqref{1} are uniformly bounded in a fixed neighborhood of $\partial\Omega$.
\end{remark}

\begin{proposition} \label{prop-localization}
Fix $\kappa>0$ and let $p\in[\frac{3}{2},2]$. There exists $\lambda_0>0$, depending only on $\Omega$, $\|h\|_{L^\infty(\Omega)}$, $\|\nabla\log h\|_{L^\infty(\Omega)}$ and $\kappa$, such that, for every $0<\lambda<\lambda_0$, each connected component $D_j$ of $\{u>\left(\frac{\kappa}{M}\right)^{\frac{1}{p-1}}\}$ satisfies
\begin{equation*}
D_j\subset\mathcal{K}=\left\{x\in\Omega\bigg|\operatorname{dist}(x,\partial\Omega)
\geq\frac{\delta_*}{2}\right\}\subset\subset\Omega
\end{equation*}
and
\begin{equation*}
\operatorname{diam}D_j\leq C_0\mathrm{e}^{-\frac{2\pi\kappa^{\frac{1}{p-1}}}{M^{\frac{p}{p-1}}}},
\end{equation*}
where $C_0>0$ is a constant depending only on $\Omega$. The constants are independent of $p\in[\frac{3}{2},2]$.
\end{proposition}

\begin{proof}
Fix a component $D=D_j$. Since $u=\left(\frac{\kappa}{M}\right)^{\frac{1}{p-1}}$ on $\partial D$ by \eqref{component-boundary}, there exists $a_D\in D$ such that
\begin{equation*}
u(a_D)=\max_{\overline D}u,\quad \nabla u(a_D)=0.
\end{equation*}
Lemma \ref{lem-interior} therefore implies
$\operatorname{dist}(a_D,\partial\Omega)\geq\delta_*$.

\vskip 0.1cm

Since $u=Gf$ and $\displaystyle\int_\Omega f\mathrm{d}x=M$, Lemma \ref{lem-diameter} yields
\begin{equation*}
\operatorname{diam}D
\leq C_0\mathrm{e}^{-\frac{2\pi\kappa^{\frac{1}{p-1}}}{M^{\frac{p}{p-1}}}}.
\end{equation*}
By \eqref{3}, we may choose $\lambda_0>0$ sufficiently small, independently of $p\in[\frac{3}{2},2]$, such that
$\operatorname{diam}D\leq\frac{\delta_*}{2}$ whenever $0<\lambda<\lambda_0$.
Consequently, for every $x\in D$,
\begin{equation*}
\begin{aligned}
\operatorname{dist}(x,\partial\Omega)&\geq\operatorname{dist}(a_D,\partial\Omega)-|x-a_D|
\geq\delta_*-\operatorname{diam}D
\geq\frac{\delta_*}{2}.
\end{aligned}
\end{equation*}
Thus $D\subset\mathcal{K}$.
\end{proof}

\begin{remark}
The estimates in Lemma \ref{lem-low-source} and Proposition \ref{prop-localization} are stated for $p\in[\frac{3}{2},2]$. In the proof of Theorem \ref{thm-simply-connected}, both are used only after the local blow-up analysis has established $p_n\to2$.
\end{remark}

For a connected component $D$ of $\{u>t\}$ with $t>0$, set
\begin{equation*}
q_D=\int_D f\mathrm{d}x,\quad P_D=h(F(u)-F(t))\mathbf{1}_D.
\end{equation*}
Since $u=t$ on $\partial D$, we have $P_D=0$ on $\partial D$.

\begin{proposition} \label{prop-pressure}
The zero extension of $P_D$ to $\mathbb{R}^2$ is compactly supported and Lipschitz. Moreover,
\begin{equation*}
0\leq\int_\Omega P_D\mathrm{d}x\leq\frac{\sup_Dh}{\inf_Dh}\frac{q_D^2}{8\pi},
\end{equation*}
and for every Lipschitz vector field $Y\colon \overline\Omega\to\mathbb{R}^2$,
\begin{equation} \label{pressure-vector}
\int_D fY\cdot\nabla u\mathrm{d}x=-\int_\Omega P_D\operatorname{div}Y\mathrm{d}x-\int_\Omega P_DY\cdot\nabla\log h\mathrm{d}x.
\end{equation}
\end{proposition}

\begin{proof}
Since $\overline D\subset\subset\Omega$, the gradient of $h(F(u)-F(t))$ is bounded on $\overline D$. For $x,y\in\mathbb{R}^2$, it suffices to consider $P_D(x)>P_D(y)\geq0$, so that $x\in D$. If the segment joining $x$ and $y$ is contained in $D$, then integration along this segment yields
\begin{equation*}
|P_D(x)-P_D(y)|\leq C|x-y|.
\end{equation*}
Otherwise, let $z$ be the first point where the segment from $x$ to $y$ meets $\partial D$. Since $P_D(z)=0$ by \eqref{component-boundary},
\begin{equation*}
|P_D(x)-P_D(y)|\leq |P_D(x)-P_D(z)|\leq C|x-z|\leq C|x-y|.
\end{equation*}
Thus the zero extension of $P_D$ is Lipschitz and compactly supported in $\Omega$. By Theorem 4.4 in \cite{EvansGariepy2025}, we deduce that
\begin{equation*}
\nabla P_D=f\mathbf{1}_D\nabla u+P_D\nabla\log h\quad\text{almost everywhere in }\Omega.
\end{equation*}
Hence, it holds
\begin{equation*}
\begin{aligned}
\int_D fY\cdot\nabla u\mathrm{d}x&=\int_\Omega Y\cdot\nabla P_D\mathrm{d}x-\int_\Omega P_DY\cdot\nabla\log h\mathrm{d}x\\
&=-\int_\Omega P_D\operatorname{div}Y\mathrm{d}x-\int_\Omega P_DY\cdot\nabla\log h\mathrm{d}x.
\end{aligned}
\end{equation*}

Set $Z=\{\nabla u=0\}$. Using Theorem 4.4 in \cite{EvansGariepy2025} again, we also have
\begin{equation*}
\nabla(\partial_j u)=0\quad\text{almost everywhere in }\{\partial_j u=0\}\quad\text{for }j=1,2.
\end{equation*}
Thus $D^2u=0$ almost everywhere in $Z$. Since $-\Delta u=f>0$ in $\Omega$,
\begin{equation} \label{critical-set-zero}
|Z|=0.
\end{equation}
For $s>t$, define
\begin{equation*}
a(s)=|D\cap\{u>s\}|,\quad q(s)=\int_{D\cap\{u>s\}}f\mathrm{d}x.
\end{equation*}
By \eqref{critical-set-zero} and the coarea formula in \cite{EvansGariepy2025}, the functions $a$ and $q$ are locally absolutely continuous on $(t,\infty)$, with
\begin{equation} \label{coarea-a}
-a'(s)=\int_{D\cap\{u=s\}}\frac{1}{|\nabla u|}\mathrm{d}s,
\end{equation}
and
\begin{equation} \label{coarea-aq}
-q'(s)=F'(s)\int_{D\cap\{u=s\}}\frac{h}{|\nabla u|}\mathrm{d}s\geq(\inf_Dh)F'(s)(-a'(s))
\end{equation}
for almost every $s>t$.

\vskip 0.1cm

By Sard's theorem, almost every $s\in(t,\max_Du)$ is a regular value of $u$. Denote $E_s=D\cap\{u>s\}$. Then $\overline E_s\subset\subset D$ and
$
\partial E_s=D\cap\{u=s\}$
is a regular $C^2$ level curve, with outward normal $-\frac{\nabla u}{|\nabla u|}$. Thus
\begin{equation} \label{level-flux}
q(s)=\int_{E_s}f\mathrm{d}x =-\int_{\partial E_s}\partial_\nu u\mathrm{d}s =\int_{\partial E_s}|\nabla u|\mathrm{d}s.
\end{equation}
By \eqref{coarea-a}, \eqref{level-flux}, the Cauchy-Schwarz inequality and the isoperimetric inequality, we have
\begin{equation} \label{aq-isoperimetric}
\begin{aligned}
q(s)(-a'(s)) &=\left(\int_{\partial E_s}|\nabla u|\mathrm{d}s\right) \left(\int_{\partial E_s}\frac{1}{|\nabla u|}\mathrm{d}s\right) \geq\left(\int_{\partial E_s}\mathrm{d}s\right)^2 \geq4\pi|E_s|=4\pi a(s),
\end{aligned}
\end{equation}
for almost every $s>t$.

\vskip 0.1cm

Since $F'(s)=p\lambda s^{p-1}\mathrm{e}^{s^p}$, Tonelli's theorem yields
\begin{equation} \label{pressure-layer-cake}
\begin{aligned}
\int_\Omega P_D\mathrm{d}x&=\int_Dh(x)\int_t^{u(x)}F'(s)\mathrm{d}s\mathrm{d}x\\
&=\int_t^\infty F'(s)\left(\int_{D\cap\{u>s\}}h(x)\mathrm{d}x\right)\mathrm{d}s
\leq(\sup_Dh)\int_t^\infty F'(s)a(s)\mathrm{d}s.
\end{aligned}
\end{equation}
Since $q(t^+)=q_D$ and $q(s)=0$ for $s\geq\max_Du$,
\begin{equation*}
\int_t^\infty q(s)(-q'(s))\mathrm{d}s =\left.-\frac{1}{2}q(s)^2\right|_{s=t^+}^{s=\infty} =\frac{q_D^2}{2}.
\end{equation*}
By \eqref{coarea-aq} and \eqref{aq-isoperimetric}, we have
\begin{equation*}
F'(s)a(s)\leq\frac{F'(s)q(s)(-a'(s))}{4\pi}\leq\frac{q(s)(-q'(s))}{4\pi\inf_Dh}
\end{equation*}
for almost every $s>t$. Substituting this bound into \eqref{pressure-layer-cake}, we obtain
\begin{equation*}
\begin{aligned}
0\leq\int_\Omega P_D\mathrm{d}x&\leq(\sup_Dh)\int_t^\infty F'(s)a(s)\mathrm{d}s\\
&\leq\frac{\sup_Dh}{4\pi\inf_Dh}\int_t^\infty q(s)(-q'(s))\mathrm{d}s=\frac{\sup_Dh}{\inf_Dh}\frac{q_D^2}{8\pi}.
\end{aligned}
\end{equation*}
\end{proof}

Fix $t>0$ and let $D_j$ be the connected components of $\{u>t\}$. Summing the bound in Proposition \ref{prop-pressure}, we obtain
\begin{equation*}
\int_{\{u>t\}}h(F(u)-F(t))\mathrm{d}x\leq\frac{\max_{\overline\Omega}h}{\min_{\overline\Omega}h}\frac{1}{8\pi}\sum_jq_{D_j}^2.
\end{equation*}
Since the components are disjoint, then we have 
\begin{equation*}
\sum_jq_{D_j}^2\leq\Big(\sum_jq_{D_j}\Big)^2=\Big(\int_{\{u>t\}}f\mathrm{d}x\Big)^2\leq M^2.
\end{equation*}
Letting $t\to0$ and applying the monotone convergence theorem yields
\begin{equation} \label{global-primitive}
\int_\Omega hF(u)\mathrm{d}x\leq\frac{\max_{\overline\Omega}h}{\min_{\overline\Omega}h}\frac{M^2}{8\pi}.
\end{equation}

Testing \eqref{1} with $u$, we have
\begin{equation*}
E(u)=\int_\Omega uf\mathrm{d}x\leq mM.
\end{equation*}
Together with \eqref{beta-definition} and \eqref{global-primitive}, this yields
\begin{equation} \label{beta-source}
\begin{aligned}
\beta^{\frac{p}{2}}&=\frac{p}{2^{\frac{p}{2}}}\left(\int_\Omega hF(u)\mathrm{d}x\right)^{\frac{2-p}{2}}E(u)^{p-1}\leq CMm^{p-1},
\end{aligned}
\end{equation}
where $C$ is independent of $p\in[1,2]$.

\subsection{Local profiles and radial comparison}\

\vskip 0.1cm

In this subsection, we develop a local blow-up analysis that applies to suitably selected high points without requiring them to be critical points of $u$. We then derive a radial comparison estimate from \cite{DeMarchisMalchiodiMartinazziThizy2022}. At $p=2$, see also \cite{Druet2006,DruetThizy2020}.

\vskip 0.1cm

For a sequence of positive solutions $u_n$ of \eqref{1} with exponents $p_n$ and parameters $\lambda_n$, choose $x_{0,n}\in\Omega$ such that $u_n(x_{0,n})=m_n$. Let $\lambda_1=\lambda_1(\Omega)$ be the first Dirichlet eigenvalue and let $\varphi_1>0$ be a corresponding eigenfunction. For $1\leq p<2$, set $a=\frac{2-p}{p}\in(0,1]$. Since
\begin{equation*}
\frac{\mathrm{d}}{\mathrm{d}z}(z^{-a}\mathrm{e}^z)=z^{-a-1}\mathrm{e}^z(z-a)\quad\text{for }z>0,
\end{equation*}
we obtain for every $s>0$,
\begin{equation*}
ps^{p-2}\mathrm{e}^{s^p}=p(s^p)^{-a}\mathrm{e}^{s^p}\geq pa^{-a}\mathrm{e}^a\geq p\mathrm{e}^a>p(1+a)=2.
\end{equation*}
For $p=2$, the same strict inequality follows from $2\mathrm{e}^{s^2}>2$ for $s>0$.
Testing \eqref{1} against $\varphi_1$, we obtain
\begin{equation*}
\begin{aligned}
\lambda_1\int_\Omega u_n\varphi_1\mathrm{d}x&=p_n\lambda_n\int_\Omega hu_n^{p_n-1}\mathrm{e}^{u_n^{p_n}}\varphi_1\mathrm{d}x
>2\lambda_n\min_{\overline\Omega}h\int_\Omega u_n\varphi_1\mathrm{d}x.
\end{aligned}
\end{equation*}
Therefore, it holds
\begin{equation} \label{lambda-m-basic}
0<2\lambda_n\min_{\overline\Omega}h<\lambda_1.
\end{equation}
Moreover, for $p_n=2$, the bound $u_n\leq m_n$ yields
\begin{equation*}
\begin{aligned}
\lambda_1\int_\Omega u_n\varphi_1\mathrm{d}x&=2\lambda_n\int_\Omega hu_n\mathrm{e}^{u_n^2}\varphi_1\mathrm{d}x
\leq2\lambda_n\max_{\overline\Omega}h\mathrm{e}^{m_n^2}\int_\Omega u_n\varphi_1\mathrm{d}x.
\end{aligned}
\end{equation*}
Thus, we also have
\begin{equation} \label{lambda-m-critical}
m_n^2\geq\log\frac{\lambda_1}{2\lambda_n\max_{\overline\Omega}h}
=L_n+\log\frac{\lambda_1}{2\max_{\overline\Omega}h}.
\end{equation}

For $x_n\in\Omega$, set
\begin{equation} \label{point-scale}
b_n=u_n(x_n),\quad \nu_n=\frac{2\sqrt{2}}{p_n\sqrt{\lambda_n h(x_n)}b_n^{p_n-1}\mathrm{e}^{\frac{b_n^{p_n}}{2}}},
\end{equation}
and define
\begin{equation*}
\Omega_n=\frac{\Omega-x_n}{\nu_n},\quad \xi_n(y)=\frac{p_n}{2}b_n^{p_n-1}\big(u_n(x_n+\nu_ny)-b_n\big).
\end{equation*}
In the rescaled variables $y$, we write $B_R=B_R(0)$.

\begin{lemma} \label{lem-critical-core}
Assume that
\begin{equation*}
\operatorname{dist}(x_n,\partial\Omega)\geq\delta_1>0,\quad b_n\to\infty,\quad \frac{L_n}{b_n^{p_n}}\to0.
\end{equation*}
Suppose that one of the following conditions holds:
\begin{enumerate}
\item[(i)] $\nabla u_n(x_n)=0$ and for every fixed $R>0$, $\displaystyle\sup_{B_R\cap\Omega_n}\xi_n\leq C_R$, with a constant $C_R$ independent of $n$.
\item[(ii)] For every fixed $R>0$, $\displaystyle\sup_{B_R\cap\Omega_n}\xi_n\to0$.
\end{enumerate}
Then $p_n\to2$ and
\begin{equation} \label{core-limit}
\xi_n\to-\log(1+|y|^2)\quad\text{in }C^2_{\mathrm{loc}}(\mathbb{R}^2).
\end{equation}
Moreover, for every fixed $R>0$,
\begin{equation} \label{core-masses}
\frac{p_n}{2}b_n^{p_n-1}\int_{B_{\nu_nR}(x_n)}f_n\mathrm{d}x \to\frac{4\pi R^2}{1+R^2},\quad \frac{p_n}{2}\int_{B_{\nu_nR}(x_n)}u_n^{p_n-1}f_n\mathrm{d}x \to\frac{4\pi R^2}{1+R^2}.
\end{equation}
\end{lemma}

\begin{proof}
For each fixed $R>0$,
\begin{equation} \label{core-scale}
\log\frac{\delta_1}{\nu_nR} =\frac{b_n^{p_n}-L_n}{2}+(p_n-1)\log b_n+\log\frac{p_n\sqrt{h(x_n)}\delta_1}{2\sqrt{2}R} \to\infty.
\end{equation}
Thus $B_R\subset\Omega_n$ for sufficiently large $n$, and in both cases
\begin{equation} \label{upper-bound}
\sup_{B_R}\xi_n\leq C_R.
\end{equation}
Under the change of variables $x=x_n+\nu_ny$, we have
\begin{equation*}
u_n(x_n+\nu_ny)=b_n+\frac{2\xi_n(y)}{p_nb_n^{p_n-1}},\quad \frac{\lambda_np_n^2h(x_n)}{2}b_n^{2p_n-2}\mathrm{e}^{b_n^{p_n}}\nu_n^2=4.
\end{equation*}
Therefore, it holds
\begin{equation} \label{core-rescaled}
-\Delta\xi_n=k_n,\quad \xi_n(0)=0,\quad \xi_n>-\frac{p_n}{2}b_n^{p_n},
\end{equation}
where
\begin{equation*}
k_n=4\frac{h(x_n+\nu_ny)}{h(x_n)}\left(1+\frac{2\xi_n}{p_nb_n^{p_n}}\right)^{p_n-1}\exp\left[b_n^{p_n}\left(\left(1+\frac{2\xi_n}{p_nb_n^{p_n}}\right)^{p_n}-1\right)\right].
\end{equation*}
In $\{\xi_n\leq0\}$,
\begin{equation*}
\begin{aligned}
&0<1+\frac{2\xi_n}{p_nb_n^{p_n}}=\frac{u_n(x_n+\nu_ny)}{b_n}\leq1,\\
&b_n^{p_n}\left(\left(1+\frac{2\xi_n}{p_nb_n^{p_n}}\right)^{p_n}-1\right)=u_n(x_n+\nu_ny)^{p_n}-b_n^{p_n}\leq0.
\end{aligned}
\end{equation*}
Consequently, \eqref{core-rescaled} implies
\begin{equation*}
0\leq k_n\leq4\frac{h(x_n+\nu_ny)}{h(x_n)}\leq4\frac{\max_{\overline\Omega}h}{\min_{\overline\Omega}h}.
\end{equation*}
In $\{\xi_n>0\}\cap B_{2R}$, using \eqref{upper-bound} and $b_n\geq1$ for sufficiently large $n$, we obtain
\begin{equation*}
\begin{aligned}
&b_n^{p_n}\left(\left(1+\frac{2\xi_n}{p_nb_n^{p_n}}\right)^{p_n}-1\right)=2\xi_n\int_0^1\left(1+\frac{2t\xi_n}{p_nb_n^{p_n}}\right)^{p_n-1}\mathrm{d}t\leq2C_{2R}(1+2C_{2R}).
\end{aligned}
\end{equation*}
Consequently,
\begin{equation*}
k_n\leq C(1+2C_{2R})\mathrm{e}^{2C_{2R}(1+2C_{2R})}\leq C_R.
\end{equation*}
Thus $0\leq k_n\leq C_R$ in $B_{2R}$, with $C_R$ independent of $p_n\in[1,2]$.
Let $w_n$ solve
\begin{equation*}
\begin{cases}
-\Delta w_n=k_n &\quad\text{in }B_{2R},\\
w_n=0 &\quad\text{on }\partial B_{2R}.
\end{cases}
\end{equation*}
The comparison principle implies
\begin{equation*}
0\leq w_n(y)\leq\frac{C_R}{4}(4R^2-|y|^2)\leq C_RR^2 \quad\text{for }y\in B_{2R}.
\end{equation*}
Choose $A_R>0$ with $\xi_n\leq A_R$ in $B_{2R}$ and denote
$h_n=w_n-\xi_n+A_R$. Then $h_n\geq0$ and $\Delta h_n=0$.
Harnack's inequality in \cite{GilbargTrudinger2001} yields
\begin{equation*}
\sup_{B_R}h_n\leq Ch_n(0)=C(w_n(0)+A_R)\leq C_R.
\end{equation*}
Hence we have
\begin{equation*}
-C_R\leq\xi_n\leq C_R\quad\text{in }B_R.
\end{equation*}
Fix $r>2$ and $0<\alpha<1-\frac{2}{r}$. The $L^r$ theory  yields
\begin{equation*}
\|\xi_n\|_{W^{2,r}(B_{3R})} \leq C_R\left(\|k_n\|_{L^r(B_{4R})}+\|\xi_n\|_{L^r(B_{4R})}\right) \leq C_R.
\end{equation*}
The Sobolev embedding theorem implies
\begin{equation*}
\|\xi_n\|_{C^{1,\alpha}(\overline{B_{2R}})}\leq C_R.
\end{equation*}
For $|s|\leq C_R$, put $r_n(s)=1+\frac{2s}{p_nb_n^{p_n}}$. Then $r_n(s)\in[\frac{1}{2},\frac{3}{2}]$ for sufficiently large $n$, and
\begin{equation*}
\begin{aligned}
&\left|\frac{\mathrm{d}}{\mathrm{d}s}\left(4r_n(s)^{p_n-1}\mathrm{e}^{b_n^{p_n}(r_n(s)^{p_n}-1)}
\right)\right|=4\mathrm{e}^{b_n^{p_n}(r_n(s)^{p_n}-1)}\left(\frac{2(p_n-1)}{p_nb_n^{p_n}}r_n(s)^{p_n-2}+2r_n(s)^{2p_n-2}\right)\leq C_R.
\end{aligned}
\end{equation*}
Since $h\in C^2(\overline\Omega)$ is positive and $\nu_n\to0$,
\begin{equation*}
\left\|\frac{h(x_n+\nu_n\cdot)}{h(x_n)}-1\right\|_{C^2(\overline{B_{2R}})}\leq C_R\frac{\|h\|_{C^2(\overline\Omega)}}{\min_{\overline\Omega}h}(\nu_n+\nu_n^2)\to0.
\end{equation*}
Thus $\|k_n\|_{C^\alpha(\overline{B_{2R}})}\leq C_R$. The Classical Schauder theory implies
\begin{equation*}
\|\xi_n\|_{C^{2,\alpha}(\overline{B_R})} \leq C_R\left(\|k_n\|_{C^\alpha(\overline{B_{2R}})} +\|\xi_n\|_{C^0(\overline{B_{2R}})}\right)\leq C_R.
\end{equation*}
Moreover, Taylor's formula and the uniform bound for $\xi_n$ on $B_R$ yield
\begin{equation*}
0\leq b_n^{p_n}\left(\left(1+\frac{2\xi_n}{p_nb_n^{p_n}}\right)^{p_n}-1\right)-2\xi_n
\leq\frac{C_R}{b_n^{p_n}}.
\end{equation*}
Therefore, after passing to a subsequence, $p_n\to p_* \in [1,2]$ and
\begin{equation} \label{core-subsequence}
\xi_n\to\xi_\infty\quad\text{in }C^2_{\mathrm{loc}}(\mathbb{R}^2),
\end{equation}
with
\begin{equation} \label{core-limit-eq}
-\Delta\xi_\infty=4\mathrm{e}^{2\xi_\infty}\quad\text{in }\mathbb{R}^2, \quad \xi_\infty(0)=0.
\end{equation}
In case (i), $\nabla\xi_n(0)=\frac{p_n}{2}b_n^{p_n-1}\nu_n\nabla u_n(x_n)=0$. In case (ii), $\xi_\infty\leq0=\xi_\infty(0)$. Then
\begin{equation*}
\nabla\xi_\infty(0)=0.
\end{equation*}

For $0<|x-x_n|<\delta_1$, the domain monotonicity of the Green function yields
\begin{equation*}
G_\Omega(x_n,x)\geq G_{B_{\delta_1}(x_n)}(x_n,x) =\frac{1}{2\pi}\log\frac{\delta_1}{|x-x_n|}.
\end{equation*}
For every fixed $R>0$ and sufficiently large $n$, the Green representation yields
\begin{equation*}
\begin{aligned}
\frac{p_n}{2}b_n^{p_n}
&\geq\frac{p_nb_n^{p_n-1}}{4\pi}
\int_{B_{\nu_nR}(x_n)}\log\frac{\delta_1}{|x-x_n|}f_n(x)\mathrm{d}x\\
&=\frac{1}{2\pi}\int_{B_R}\log\frac{\delta_1}{\nu_n|y|}k_n(y)\mathrm{d}y
\geq\frac{1}{2\pi}\log\frac{\delta_1}{\nu_nR}
\int_{B_R}k_n(y)\mathrm{d}y.
\end{aligned}
\end{equation*}
By \eqref{core-scale}, $b_n^{-p_n}\log\frac{\delta_1}{\nu_nR}\to\frac{1}{2}$. Using \eqref{core-rescaled} and \eqref{core-subsequence} yields
\begin{equation*}
4\int_{B_R}\mathrm{e}^{2\xi_\infty}\mathrm{d}y\leq2\pi p_*.
\end{equation*}
Letting $R\to\infty$, we obtain
\begin{equation} \label{core-finite-mass}
4\int_{\mathbb{R}^2}\mathrm{e}^{2\xi_\infty}\mathrm{d}y\leq2\pi p_*\leq4\pi.
\end{equation}
By \eqref{core-limit-eq} and \eqref{core-finite-mass}, the function $v=2\xi_\infty+\log8$ satisfies
\begin{equation*}
-\Delta v=\mathrm{e}^v\quad\text{in }\mathbb{R}^2,\quad \int_{\mathbb{R}^2}\mathrm{e}^v\mathrm{d}y =8\int_{\mathbb{R}^2}\mathrm{e}^{2\xi_\infty}\mathrm{d}y<\infty.
\end{equation*}
The classification theorem in \cite{ChenLi1991} provides $A>0$ and $y_0\in\mathbb{R}^2$ such that
\begin{equation*}
\mathrm{e}^{v(y)}=\frac{8A^2}{(1+A^2|y-y_0|^2)^2},\quad \xi_\infty(y)=\log A-\log(1+A^2|y-y_0|^2).
\end{equation*}
Since
\begin{equation*}
\nabla\xi_\infty(0)=\frac{2A^2y_0}{1+A^2|y_0|^2}=0,\quad \xi_\infty(0)=0,
\end{equation*}
we have $y_0=0$ and $A=1$. For this profile, \eqref{core-finite-mass} gives
\begin{equation*}
4\pi=4\int_{\mathbb{R}^2}\mathrm{e}^{2\xi_\infty}\mathrm{d}y
\leq2\pi p_*\leq4\pi.
\end{equation*}
Thus $p_*=2$. Both the limit and the profile are independent of the subsequence. Therefore, $p_n\to2$ and \eqref{core-limit} holds.

\vskip 0.1cm

Finally, recall that
\begin{equation*}
\mathrm{d}x=\nu_n^2\mathrm{d}y,\quad \frac{u_n(x_n+\nu_ny)}{b_n}=1+\frac{2\xi_n(y)}{p_nb_n^{p_n}},\quad \frac{p_n}{2}b_n^{p_n-1}\nu_n^2f_n(x_n+\nu_ny)=k_n(y).
\end{equation*}
Hence we get
\begin{equation*}
\begin{aligned}
\frac{p_n}{2}b_n^{p_n-1}\int_{B_{\nu_nR}(x_n)}f_n\mathrm{d}x&=\int_{B_R}\frac{p_n}{2}b_n^{p_n-1}\nu_n^2f_n(x_n+\nu_n y)\mathrm{d}y=\int_{B_R}k_n\mathrm{d}y,\\
\frac{p_n}{2}\int_{B_{\nu_nR}(x_n)}u_n^{p_n-1}f_n\mathrm{d}x&=\int_{B_R}\left(1+\frac{2\xi_n}{p_nb_n^{p_n}}\right)^{p_n-1}k_n\mathrm{d}y.
\end{aligned}
\end{equation*}
Since $k_n\to4(1+|y|^2)^{-2}$ and $\frac{\xi_n}{b_n^{p_n}}\to0$ uniformly in $B_R$, both integrals converge to
\begin{equation*}
\int_{B_R}\frac{4}{(1+|y|^2)^2}\mathrm{d}y=8\pi\int_0^R\frac{r}{(1+r^2)^2}\mathrm{d}r=\left.-\frac{4\pi}{1+r^2}\right|_{r=0}^{r=R}=\frac{4\pi R^2}{1+R^2}.
\end{equation*}
\end{proof}

With $b_n$ and $\nu_n$ as in \eqref{point-scale}, define
\begin{equation*}
\overline u_n(r)=\frac{1}{2\pi}\int_0^{2\pi} u_n(x_n+r(\cos\theta,\sin\theta))\mathrm{d}\theta,\quad t_n(r)=\log\left(1+\frac{r^2}{\nu_n^2}\right).
\end{equation*}
We have the following radial comparison estimate.

\begin{lemma} \label{lem-radial}
Assume that $p_n\in[1,2]$, $b_n\to\infty$, $\nabla u_n(x_n)=0$ and
\begin{equation*}
B_{\varrho_n}(x_n)\subset\subset\Omega,\quad \varrho_n\to0,\quad \frac{\varrho_n}{\nu_n}\to\infty.
\end{equation*}
Suppose that
\begin{equation*}
\frac{p_n}{2}b_n^{p_n-1}\left(u_n(x_n+\nu_ny)-b_n\right)\to-\log(1+|y|^2)\quad\text{in }C^2_{\mathrm{loc}}(\mathbb{R}^2),
\end{equation*}
and for some fixed $C_*>0$ and $\eta_0\in(0,1)$,
\begin{equation*}
u_n\leq C_*b_n,\quad |x-x_n||\nabla u_n(x)|\leq\frac{C_*}{b_n^{p_n-1}}\quad\text{in }B_{\varrho_n}(x_n),
\end{equation*}
and
\begin{equation*}
t_n(\varrho_n)\leq\eta_0\frac{p_n}{2}b_n^{p_n}.
\end{equation*}
Then there exists a constant $C>0$, depending only on $C_*$, $\eta_0$, $\|h\|_{C^2(\overline\Omega)}$ and $(\min_{\overline\Omega}h)^{-1}$, such that
\begin{equation} \label{radial-mean}
\Big|\frac{p_n}{2}b_n^{p_n-1}\overline u_n(r)-\left(\frac{p_n}{2}b_n^{p_n}-t_n(r)\right)\Big|\leq C\quad\text{for $0\leq r\leq\varrho_n$ and all sufficiently large $n$}.
\end{equation}
\end{lemma}

\begin{proof}
By \eqref{point-scale} and the assumptions on $h$, we have
\begin{equation*}
\begin{aligned}
\lambda_np_n^2h(x_n)b_n^{2(p_n-1)}\nu_n^2\mathrm{e}^{b_n^{p_n}}=8,\quad
\left\|\frac{h(x_n+\varrho_n\cdot)}{h(x_n)}-1\right\|_{C^2(\overline{B_1})}\leq C\frac{\|h\|_{C^2(\overline\Omega)}}{\min_{\overline\Omega}h}(\varrho_n+\varrho_n^2)\to0.
\end{aligned}
\end{equation*}
The assumed bounds for $u_n$ and $\nabla u_n$ imply
\begin{equation*}
|x-x_n|u_n^{p_n-1}|\nabla u_n|\leq C_*^{p_n}\leq C\quad\text{in }B_{\varrho_n}(x_n).
\end{equation*}
To apply Proposition 3.1 in \cite{DeMarchisMalchiodiMartinazziThizy2022}, set
\begin{equation*}
\varphi_n(x)=\frac{1}{2}\log\frac{h(x_n+x)}{h(x_n)}.
\end{equation*}
Then $\varphi_n(0)=0$ and $\varphi_n(\varrho_n\cdot)\to0$ in $C^2(\overline{B_1})$. From the proof of Theorem 2.1 in \cite{MalchiodiMartinazziThizy2024}, when the linear zeroth-order coefficient is zero, neither an $L^4$ bound for $u_n$ nor a bound on $b_n^{2p_n}\varrho_n^2$ is needed. Therefore, Proposition 3.1 in \cite{DeMarchisMalchiodiMartinazziThizy2022} gives
\begin{equation} \label{radial-approx}
\big|u_n(x_n+x)-B_n(|x|)\big|\leq\frac{C|x|}{b_n^{p_n-1}\varrho_n}\quad\text{for }|x|\leq\varrho_n,
\end{equation}
where
\begin{equation*}
-B_n''-\frac{1}{r}B_n'=p_n\lambda_nh(x_n)B_n^{p_n-1}\mathrm{e}^{B_n^{p_n}},\quad B_n(0)=b_n,\quad B_n'(0)=0.
\end{equation*}
The radial estimate in Proposition 2.1 of \cite{DeMarchisMalchiodiMartinazziThizy2022} yields
\begin{equation} \label{radial-expansion}
B_n(r)=b_n-\frac{2t_n(r)}{p_nb_n^{p_n-1}}+O\left(\frac{t_n(r)}{b_n^{2p_n-1}}\right)\quad\text{for }0\leq r\leq\varrho_n.
\end{equation}
Since $B_n$ is radial, \eqref{radial-approx} implies
\begin{equation*}
\left|\overline u_n(r)-B_n(r)\right|\leq\frac{Cr}{b_n^{p_n-1}\varrho_n}.
\end{equation*}
Combining this with \eqref{radial-expansion}, we obtain for $0\leq r\leq\varrho_n$,
\begin{equation*}
\left|\frac{p_n}{2}b_n^{p_n-1}\overline u_n(r)-\left(\frac{p_n}{2}b_n^{p_n}-t_n(r)\right)\right|\leq C\frac{r}{\varrho_n}+C\frac{t_n(r)}{b_n^{p_n}}\leq C.
\end{equation*}
This proves \eqref{radial-mean}.
\end{proof}

\subsection{The maximum estimate}\

\vskip 0.1cm

In this subsection, we prove the maximum estimate for a sequence with $\beta_n\to\infty$. We first exclude a second core near a global maximum and then apply the radial comparison estimate.

\vskip 0.1cm

Set
\begin{equation*}
\mathcal{F}_n=\frac{p_n^2\lambda_n}{2}hu_n^{2(p_n-1)}\mathrm{e}^{u_n^{p_n}}.
\end{equation*}
Let $\mathcal{U}_n\subset\Omega$ and let $d_n\colon\mathcal{U}_n\to[0,\infty)$ be $1$-Lipschitz, that is,
\begin{equation*}
|d_n(x)-d_n(y)|\leq|x-y|\quad\text{for all }x,y\in\mathcal{U}_n.
\end{equation*}
Suppose that $x_n\in\mathcal{U}_n$ maximizes $d_n^2\mathcal{F}_n$. Write
\begin{equation} \label{point-scale-1}
\Lambda_n=d_n(x_n)^2\mathcal{F}_n(x_n),\quad b_n=u_n(x_n),\quad \nu_n=\frac{2}{\sqrt{\mathcal{F}_n(x_n)}}.
\end{equation}
Suppose that $\Lambda_n\to\infty$, $\nu_n\to0$ and that $\overline{B_{R\nu_n}(x_n)}\subset\mathcal{U}_n$ for each fixed $R>0$ and sufficiently large $n$. Since $\mathcal{F}_n(x_n)>0$ and $\Lambda_n>0$, we have $d_n(x_n)>0$ for sufficiently large $n$, and
\begin{equation*}
\frac{\nu_n}{d_n(x_n)}=\frac{2}{\sqrt{\Lambda_n}}\to0,\quad d_n(x_n+\nu_nz)\geq d_n(x_n)-R\nu_n>0\quad\text{for }|z|\leq R.
\end{equation*}
For $|z|\leq R$, denote $y=x_n+\nu_nz$. The choice of $x_n$ yields
\begin{equation*}
\frac{\mathcal{F}_n(y)}{\mathcal{F}_n(x_n)} \leq\left(\frac{d_n(x_n)}{d_n(y)}\right)^2 \leq\left(\frac{d_n(x_n)}{d_n(x_n)-R\nu_n}\right)^2.
\end{equation*}
If $u_n(y)\geq b_n$, then
\begin{equation*}
\begin{aligned}
p_nb_n^{p_n-1}(u_n(y)-b_n) &\leq u_n(y)^{p_n}-b_n^{p_n}\\
&=\log\frac{\mathcal{F}_n(y)}{\mathcal{F}_n(x_n)}-2(p_n-1)\log\frac{u_n(y)}{b_n}-\log\frac{h(y)}{h(x_n)}\\
&\leq\log\frac{\mathcal{F}_n(y)}{\mathcal{F}_n(x_n)}+C|y-x_n|.
\end{aligned}
\end{equation*}
Since $u_n(x_n)=b_n$ and $\frac{R\nu_n}{d_n(x_n)}=\frac{2R}{\sqrt{\Lambda_n}}$,
\begin{equation} \label{selection-upper}
\begin{aligned}
0&\leq\sup_{|z|\leq R}\frac{p_n}{2}b_n^{p_n-1}\left(u_n(x_n+\nu_nz)-b_n\right)\\
&\leq\log\frac{d_n(x_n)}{d_n(x_n)-R\nu_n}+CR\nu_n
=-\log\left(1-\frac{2R}{\sqrt{\Lambda_n}}\right)+CR\nu_n\to0.
\end{aligned}
\end{equation}
Recall that $x_{0,n}\in\Omega$ is a global maximum point of $u_n$ with $u_n(x_{0,n})=m_n$. We now apply the point-selection argument near $x_{0,n}$.
For $s>0$, set
\begin{equation*}
\Theta_n(s)=\lambda_ns^{2(p_n-1)}\mathrm{e}^{\frac{2-p_n}{2}s^{p_n}}.
\end{equation*}

\begin{proposition} \label{prop-uniform-max}
Assume that $m_n\to\infty$ and $\frac{L_n}{m_n^{p_n}}\to0$, and let
\begin{equation*}
\overline u_n(r)=\frac{1}{2\pi}\int_0^{2\pi} u_n(x_{0,n}+r(\cos\theta,\sin\theta))\mathrm{d}\theta.
\end{equation*}
For each fixed $\delta>0$ such that $B_{2\delta}(x_{0,n})\subset\subset\Omega$ for all sufficiently large $n$, we have
\begin{equation} \label{fixed-radius-mean}
\frac{p_n}{2}m_n^{p_n-1}\overline u_n(\delta)\leq-\log\Theta_n(m_n)+C,
\end{equation}
where $C>0$ depends only on $\Omega$, $\delta$, $\|h\|_{C^2(\overline\Omega)}$ and $(\min_{\overline\Omega}h)^{-1}$. Moreover, $p_n\to2$, $\lambda_n\to0$ and for all sufficiently large $n$,
\begin{equation} \label{uniform-max-bound}
\Theta_n(m_n)\leq C,\quad m_n^{p_n-1}\leq\frac{C}{\sqrt{\lambda_n}},\quad M_n\geq c\sqrt{\lambda_n}.
\end{equation}
\end{proposition}

\begin{proof}
Set
\begin{equation*}
\begin{aligned}
\mu_n&=\frac{2\sqrt{2}}{p_n\sqrt{\lambda_n h(x_{0,n})}m_n^{p_n-1}\mathrm{e}^{\frac{m_n^{p_n}}{2}}}, \quad
\frac{1}{4} <a=\frac{3}{8}<\frac{1}{2},\quad d_0=\frac{\delta_*}{4},\quad r_n=d_0\mathrm{e}^{-am_n^{p_n}}.
\end{aligned}
\end{equation*}
By Lemma \ref{lem-interior}, $\overline{B_{2d_0}(x_{0,n})}\subset\Omega$.
Since $x_{0,n}$ is a global maximum of $u_n$,
\begin{equation*}
\nabla u_n(x_{0,n})=0,\quad \frac{p_n}{2}m_n^{p_n-1}(u_n(x_{0,n}+\mu_ny)-m_n)\leq0 \quad\text{for }x_{0,n}+\mu_ny\in\Omega.
\end{equation*}
Lemma \ref{lem-critical-core}-(i) gives $p_n\to2$. In particular, $p_n\in[\frac{3}{2},2]$ for all sufficiently large $n$. Moreover,
\begin{equation} \label{standard-core}
\eta_n(y):=\frac{p_n}{2}m_n^{p_n-1}(u_n(x_{0,n}+\mu_ny)-m_n) \to-\log(1+|y|^2)\quad\text{in }C^2_{\mathrm{loc}}(\mathbb{R}^2),
\end{equation}
and for each fixed $R>0$,
\begin{equation} \label{relation-1}
\frac{p_n}{2}\int_{B_{R\mu_n}(x_{0,n})}u_n^{p_n-1}f_n\mathrm{d}x \to\frac{4\pi R^2}{1+R^2}.
\end{equation}
Also, by the definition of $r_n$ and $\mu_n$,
\begin{equation} \label{r_nmu_n}
\log\frac{r_n}{\mu_n} =\left(\frac{1}{2}-a\right)m_n^{p_n}-\frac{L_n}{2} +(p_n-1)\log m_n+\log\frac{p_n\sqrt{h(x_{0,n})}d_0}{2\sqrt{2}}\to\infty.
\end{equation}
For $0<|z-x_{0,n}|<d_0$, the domain monotonicity of the Green function yields
\begin{equation*}
G(x_{0,n},z)\geq G_{B_{d_0}(x_{0,n})}(x_{0,n},z) =\frac{1}{2\pi}\log\frac{d_0}{|z-x_{0,n}|}.
\end{equation*}
Therefore,
\begin{equation*}
\begin{aligned}
m_n&=\int_\Omega G_\Omega(x_{0,n},x)f_n(x)\mathrm{d}x\geq\frac{1}{2\pi}\log\frac{d_0}{r_n}\int_{B_{r_n}(x_{0,n})}f_n\mathrm{d}x=\frac{am_n^{p_n}}{2\pi}\int_{B_{r_n}(x_{0,n})}f_n\mathrm{d}x.
\end{aligned}
\end{equation*}
Since $u_n\leq m_n$,
\begin{equation} \label{source-reaction}
\begin{aligned}
\int_{B_{r_n}(x_{0,n})}f_n\mathrm{d}x&\leq\frac{2\pi}{am_n^{p_n-1}},\quad
\frac{p_n}{2}\int_{B_{r_n}(x_{0,n})}u_n^{p_n-1}f_n\mathrm{d}x\leq\frac{\pi p_n}{a}\leq\frac{2\pi}{a}=\frac{16\pi}{3}<8\pi.
\end{aligned}
\end{equation}

We first prove
\begin{equation} \label{density-bound}
|z-x_{0,n}|^2\mathcal{F}_n(z)\leq C \quad\text{in }B_{\frac{r_n}{2}}(x_{0,n}).
\end{equation}
Otherwise, after passing to a subsequence,
\begin{equation*}
\sup_{B_{\frac{r_n}{2}}(x_{0,n})}|z-x_{0,n}|^2\mathcal{F}_n(z)\to\infty.
\end{equation*}
Define
\begin{equation*}
d_n(z)=\min\Big\{|z-x_{0,n}|,r_n-|z-x_{0,n}|\Big\},
\end{equation*}
which is $1$-Lipschitz, and choose $x_n$ maximizing $d_n^2\mathcal{F}_n$ in $\overline{B_{r_n}(x_{0,n})}$. Use $\Lambda_n,b_n,\nu_n$ from \eqref{point-scale-1}. Then
\begin{equation*}
\Lambda_n\geq\sup_{B_{\frac{r_n}{2}}(x_{0,n})}|z-x_{0,n}|^2\mathcal{F}_n(z) \to\infty,\quad 0<d_n(x_n),
\end{equation*}
because $d_n=0$ at $x_{0,n}$ and on $\partial B_{r_n}(x_{0,n})$.
By the definition of $\Lambda_n$,
\begin{equation*}
\mathrm{e}^{b_n^{p_n}}=\frac{2\Lambda_n}{p_n^2\lambda_nh(x_n)d_n(x_n)^2b_n^{2(p_n-1)}}.
\end{equation*}
Using \eqref{lambda-m-basic}, $\lambda_n<\frac{\lambda_1}{2\min_{\overline\Omega}h}$, and hence
\begin{equation*}
L_n=\log\frac{1}{\lambda_n}>\log\frac{2\min_{\overline\Omega}h}{\lambda_1}.
\end{equation*}
Using this lower bound, $b_n\leq m_n$, $d_n(x_n)\leq r_n$ and $\log\Lambda_n\geq0$ for sufficiently large $n$, we obtain
\begin{equation*}
\begin{aligned}
b_n^{p_n} &=L_n-2(p_n-1)\log b_n+\log\frac{2\Lambda_n}{p_n^2h(x_n)}-2\log d_n(x_n)\\
&\geq L_n-2\log m_n+\log\Lambda_n-2\log d_0+2am_n^{p_n}-C\\
&\geq2am_n^{p_n}-2\log m_n-C\geq am_n^{p_n},
\end{aligned}
\end{equation*}
for sufficiently large $n$. Thus
\begin{equation*}
\begin{aligned}
&b_n\to\infty,\quad \frac{|L_n|}{b_n^{p_n}}\leq\frac{|L_n|}{am_n^{p_n}}\to0,\quad \operatorname{dist}(x_n,\partial\Omega)\geq\delta_*-r_n\geq\frac{\delta_*}{2}.
\end{aligned}
\end{equation*}
Since $d_n(x_n)\leq r_n$ and $\Lambda_n\to\infty$,
\begin{equation*}
\nu_n=\frac{2d_n(x_n)}{\sqrt{\Lambda_n}}\leq\frac{2r_n}{\sqrt{\Lambda_n}}\to0,\quad
\frac{R\nu_n}{d_n(x_n)}=\frac{2R}{\sqrt{\Lambda_n}}\to0,
\end{equation*}
for every fixed $R>0$. Since
\begin{equation*}
d_n(x_n)\leq r_n-|x_n-x_{0,n}|,
\end{equation*}
we have $\overline{B_{R\nu_n}(x_n)}\subset B_{r_n}(x_{0,n})$ for sufficiently large $n$. Estimate \eqref{selection-upper} verifies condition (ii) of Lemma \ref{lem-critical-core}. Therefore,
\begin{equation} \label{relation-2}
\frac{p_n}{2}\int_{B_{R\nu_n}(x_n)}u_n^{p_n-1}f_n\mathrm{d}x \to\frac{4\pi R^2}{1+R^2}.
\end{equation}

Since $d_n(x_n)\leq|x_n-x_{0,n}|$ and $b_n\leq m_n$, the definition of $\mu_n$ yields
\begin{equation*}
\begin{aligned}
\Lambda_n&\leq|x_n-x_{0,n}|^2\frac{p_n^2\lambda_nh(x_n)}{2}m_n^{2(p_n-1)}\mathrm{e}^{m_n^{p_n}}=4\frac{h(x_n)}{h(x_{0,n})}\frac{|x_n-x_{0,n}|^2}{\mu_n^2}\leq C\frac{|x_n-x_{0,n}|^2}{\mu_n^2}.
\end{aligned}
\end{equation*}
Since $\Lambda_n\to\infty$, we conclude that
\begin{equation*}
\begin{aligned}
\frac{|x_n-x_{0,n}|}{\mu_n}\geq\sqrt{\frac{\Lambda_n}{C}}\to\infty,\quad
\frac{\nu_n}{|x_n-x_{0,n}|}\leq\frac{\nu_n}{d_n(x_n)}=\frac{2}{\sqrt{\Lambda_n}}\to0.
\end{aligned}
\end{equation*}
Consequently, for each fixed $R>0$ and sufficiently large $n$,
\begin{equation*}
B_{R\mu_n}(x_{0,n})\cup B_{R\nu_n}(x_n)\subset B_{r_n}(x_{0,n}),\quad B_{R\mu_n}(x_{0,n})\cap B_{R\nu_n}(x_n)=\emptyset.
\end{equation*}
Fix $R>\sqrt2$. Equations \eqref{relation-1}, \eqref{source-reaction}  and \eqref{relation-2} yield
\begin{equation*}
\begin{aligned}
\frac{2\pi}{a}&\geq\liminf_{n\to\infty}\frac{p_n}{2}\int_{B_{r_n}(x_{0,n})}u_n^{p_n-1}f_n\mathrm{d}x\\
&\geq\lim_{n\to\infty}\frac{p_n}{2}\left(\int_{B_{R\mu_n}(x_{0,n})}u_n^{p_n-1}f_n\mathrm{d}x+\int_{B_{R\nu_n}(x_n)}u_n^{p_n-1}f_n\mathrm{d}x\right)\\
&=\frac{4\pi R^2}{1+R^2}+\frac{4\pi R^2}{1+R^2}=\frac{8\pi R^2}{1+R^2}>\frac{2\pi}{a},
\end{aligned}
\end{equation*}
which is a contradiction. This proves \eqref{density-bound}.

\vskip 0.1cm

Set $R_n=\frac{r_n}{\mu_n}$. Using \eqref{r_nmu_n}, we have $\frac{\log R_n}{m_n^{p_n}}\to\frac{1}{8}$. By \eqref{standard-core},
\begin{equation*}
u_n\geq m_n-\frac{C}{m_n^{p_n-1}}\quad\text{on }\partial B_{\mu_n}(x_{0,n}).
\end{equation*}
Since
\begin{equation*}
-\Delta_yu_n(x_{0,n}+\mu_ny)=\mu_n^2f_n(x_{0,n}+\mu_ny)\geq0,\quad u_n>0,
\end{equation*}
the comparison principle implies
\begin{equation} \label{harmonic-lower}
u_n(x_{0,n}+\mu_ny)\geq \left(m_n-\frac{C}{m_n^{p_n-1}}\right)\frac{\log\frac{R_n}{|y|}}{\log R_n} \quad\text{for }1\leq|y|\leq R_n.
\end{equation}
Choose $\sigma=\frac{1}{128}$ and $\varrho_n=\mu_n\mathrm{e}^{\sigma m_n^{p_n}}$.
Since
\begin{equation*} 
\log\frac{\varrho_n}{r_n} =\sigma m_n^{p_n}-\log R_n =\left(\sigma-\frac{1}{8}+o(1)\right)m_n^{p_n} =-\left(\frac{15}{128}+o(1)\right)m_n^{p_n},
\end{equation*}
we have
\begin{equation} \label{varrho_nr_n}
\frac{\varrho_n}{\mu_n}\to\infty,\quad \frac{\varrho_n}{r_n}\to0,\quad m_n^{p_n-1}\frac{\varrho_n}{r_n}\to0.
\end{equation}
For $1\leq|y|\leq\frac{2\varrho_n}{\mu_n}$,
\begin{equation*}
\frac{\log\frac{R_n}{|y|}}{\log R_n} \geq1-\frac{\sigma m_n^{p_n}+\log2}{\log R_n} =\frac{15}{16}+o(1).
\end{equation*}
Combining \eqref{harmonic-lower} with \eqref{standard-core}, we obtain
\begin{equation} \label{half-height}
u_n\geq\frac{m_n}{2}\quad\text{in }B_{2\varrho_n}(x_{0,n}).
\end{equation}

We next prove
\begin{equation} \label{global-max-gradient}
|x-x_{0,n}||\nabla u_n(x)|\leq\frac{C}{m_n^{p_n-1}} \quad\text{in }B_{\varrho_n}(x_{0,n}).
\end{equation}
Differentiating the Green representation formula of $u_n$ implies
\begin{equation*}
\nabla u_n(x)=\int_\Omega\nabla_xG(x,y)f_n(y)\mathrm{d}y.
\end{equation*}
At $x=x_{0,n}$, the left-hand side of \eqref{global-max-gradient} is zero. Fix $x\in B_{\varrho_n}(x_{0,n})\setminus\{x_{0,n}\}$ and set $d=|x-x_{0,n}|>0$. For $y\in B_{\frac{d}{2}}(x)$,
\begin{equation*}
\frac{d}{2}\leq|y-x_{0,n}|\leq\frac{3d}{2}\leq2\varrho_n<\frac{r_n}{2}.
\end{equation*}
From \eqref{density-bound} and \eqref{half-height},
\begin{equation*}
f_n(y)=\frac{2\mathcal{F}_n(y)}{p_nu_n(y)^{p_n-1}}\leq\frac{C}{m_n^{p_n-1}d^2}.
\end{equation*}
By \eqref{6} and \eqref{source-reaction},
\begin{equation*}
\begin{aligned}
\int_{B_{\frac{d}{2}}(x)}|\nabla_xG(x,y)|f_n(y)\mathrm{d}y&\leq\frac{C}{m_n^{p_n-1}d^2}\int_{B_{\frac{d}{2}}(x)}\frac{\mathrm{d}y}{|x-y|}\leq\frac{C}{m_n^{p_n-1}d},
\end{aligned}
\end{equation*}
and
\begin{equation*}
\int_{B_{r_n}(x_{0,n})\setminus B_{\frac{d}{2}}(x)} |\nabla_xG(x,y)|f_n(y)\mathrm{d}y \leq\frac{C}{d}\int_{B_{r_n}(x_{0,n})}f_n\mathrm{d}y \leq\frac{C}{m_n^{p_n-1}d}.
\end{equation*}
On the other hand, for $x\in B_{\varrho_n}(x_{0,n})$ and $y\in\Omega\setminus B_{r_n}(x_{0,n})$, we have
\begin{equation*}
|x-y|\geq r_n-\varrho_n\geq\frac{r_n}{2}.
\end{equation*}
Using \eqref{6} and the bound $M_n\leq C$ in Proposition \ref{pro:bound}, we obtain
\begin{equation*}
\int_{\Omega\setminus B_{r_n}(x_{0,n})}|\nabla_xG(x,y)|f_n(y)\mathrm{d}y
\leq\frac{CM_n}{r_n}\leq\frac{C}{r_n}.
\end{equation*}
Combining these three estimates and $m_n^{p_n-1}\frac{\varrho_n}{r_n}\to0$ in \eqref{varrho_nr_n}, we obtain
\begin{equation*}
\begin{aligned}
d|\nabla u_n(x)|\leq\frac{C}{m_n^{p_n-1}}+\frac{Cd}{r_n}\leq\frac{C}{m_n^{p_n-1}}+\frac{C\varrho_n}{r_n}\leq\frac{C}{m_n^{p_n-1}}.
\end{aligned}
\end{equation*}
This proves \eqref{global-max-gradient}.

\vskip 0.1cm

Set $t_n(s)=\log(1+\frac{s^2}{\mu_n^2})$. Since $\varrho_n=\mu_n\mathrm{e}^{\sigma m_n^{p_n}}$ and $p_n\to2$,
\begin{equation*}
\frac{t_n(\varrho_n)}{\frac{p_n}{2}m_n^{p_n}}
=\frac{\log(1+\mathrm{e}^{2\sigma m_n^{p_n}})}{\frac{p_n}{2}m_n^{p_n}}\to2\sigma<\frac{1}{2}.
\end{equation*}
Thus the radius condition of Lemma \ref{lem-radial} holds with $\eta_0=\frac{1}{2}$. Since $u_n\leq m_n$, Lemma \ref{lem-radial}, together with \eqref{standard-core} and \eqref{global-max-gradient}, gives
\begin{equation} \label{max-radial}
\left|\frac{p_n}{2}m_n^{p_n-1}\overline u_n(s)-\left(\frac{p_n}{2}m_n^{p_n}-t_n(s)\right)\right|\leq C \quad\text{for }0\leq s\leq\varrho_n
\end{equation}
for sufficiently large $n$. In particular,
\begin{equation} \label{exp-scale-height}
\frac{p_n}{2}m_n^{p_n-1}\overline u_n(\varrho_n) \leq\frac{p_n}{2}m_n^{p_n}-2\log\frac{\varrho_n}{\mu_n}+C.
\end{equation}
Differentiating the circle mean with respect to $s$, we obtain
\begin{equation*}
\begin{aligned}
\overline u_n'(s)=\frac{1}{2\pi s}\int_{\partial B_s(x_{0,n})}\partial_\nu u_n\mathrm{d}s.
\end{aligned}
\end{equation*}
The divergence theorem and $-\Delta u_n=f_n$ then yield
\begin{equation} \label{derivate}
-2\pi s\overline u_n'(s)=\int_{B_s(x_{0,n})}f_n\mathrm{d}x
\end{equation}
for $0<s<\operatorname{dist}(x_{0,n},\partial\Omega)$. Hence
\begin{equation*}
\begin{aligned}
\frac{p_n}{2}m_n^{p_n-1}\left(\overline u_n(\mu_n)-\overline u_n(\varrho_n)\right)&=\frac{p_nm_n^{p_n-1}}{4\pi}\int_{\mu_n}^{\varrho_n}\frac{1}{s}
\left(\int_{B_s(x_{0,n})}f_n\mathrm{d}x\right)\mathrm{d}s\\
&\leq\frac{p_nm_n^{p_n-1}}{4\pi}\left(\int_{B_{\varrho_n}(x_{0,n})}f_n\mathrm{d}x\right)\log\frac{\varrho_n}{\mu_n}.
\end{aligned}
\end{equation*}
On the other hand, since
\begin{equation*}
t_n(\mu_n)=\log2,\quad t_n(\varrho_n)=2\sigma m_n^{p_n}+\log(1+\mathrm{e}^{-2\sigma m_n^{p_n}}),
\end{equation*}
then estimate \eqref{max-radial} implies
\begin{equation*}
\begin{aligned}
\frac{p_n}{2}m_n^{p_n-1}(\overline u_n(\mu_n)-\overline u_n(\varrho_n)) &\geq t_n(\varrho_n)-t_n(\mu_n)-C
\geq2\sigma m_n^{p_n}-C.
\end{aligned}
\end{equation*}
Combining these inequalities and using $\log\frac{\varrho_n}{\mu_n}=\sigma m_n^{p_n}>0$, we obtain
\begin{equation} \label{exp-scale-flux}
\frac{p_n}{2}m_n^{p_n-1}\int_{B_{\varrho_n}(x_{0,n})}f_n\mathrm{d}x \geq4\pi-\frac{C}{m_n^{p_n}}.
\end{equation}

Fix $\delta>0$ such that $B_{2\delta}(x_{0,n})\subset\subset\Omega$ for all sufficiently large $n$. Using \eqref{r_nmu_n} and the definition of $\varrho_n$,
\begin{equation*}
\begin{aligned}
\frac{1}{m_n^{p_n}}\log\frac{\delta}{\varrho_n}&=\frac{1}{2}-\sigma-\frac{L_n}{2m_n^{p_n}}+\frac{(p_n-1)\log m_n+\log\frac{p_n\sqrt{h(x_{0,n})}\delta}{2\sqrt{2}}}{m_n^{p_n}}\to\frac{1}{2}-\sigma\in(0,1).
\end{aligned}
\end{equation*}
Hence $0<\log\frac{\delta}{\varrho_n}\leq m_n^{p_n}$ for sufficiently large $n$.
Integrating from $\varrho_n$ to $\delta$ in \eqref{derivate}, we obtain
\begin{equation*}
\begin{aligned}
\frac{p_n}{2}m_n^{p_n-1}\overline u_n(\delta) &=\frac{p_n}{2}m_n^{p_n-1}\overline u_n(\varrho_n) -\frac{p_nm_n^{p_n-1}}{4\pi}\int_{\varrho_n}^{\delta}\frac{1}{s} \left(\int_{B_s(x_{0,n})}f_n\mathrm{d}x\right)\mathrm{d}s\\
&\leq \frac{p_n}{2}m_n^{p_n-1}\overline u_n(\varrho_n) -\frac{p_nm_n^{p_n-1}}{4\pi}\left(\int_{B_{\varrho_n}(x_{0,n})}f_n\mathrm{d}x\right) \log\frac{\delta}{\varrho_n}.
\end{aligned}
\end{equation*}
Inserting \eqref{exp-scale-height} and \eqref{exp-scale-flux} into this inequality, we obtain
\begin{equation*}
\begin{aligned}
\frac{p_n}{2}m_n^{p_n-1}\overline u_n(\delta) &\leq\frac{p_n}{2}m_n^{p_n}-2\log\frac{\varrho_n}{\mu_n} -\left(2-\frac{C}{m_n^{p_n}}\right)\log\frac{\delta}{\varrho_n}+C\\
&=\frac{p_n}{2}m_n^{p_n}-2\log\frac{\delta}{\mu_n} +\frac{C}{m_n^{p_n}}\log\frac{\delta}{\varrho_n}+C\\
&\leq\frac{p_n}{2}m_n^{p_n}-2\log\frac{\delta}{\mu_n}+C,
\end{aligned}
\end{equation*}
and
\begin{equation*}
\begin{aligned}
\frac{p_n}{2}m_n^{p_n}-2\log\frac{\delta}{\mu_n}
&=\frac{p_n}{2}m_n^{p_n}-2\log\delta+2\log\frac{2\sqrt{2}}{p_n\sqrt{h(x_{0,n})}}\\
&\quad+L_n-2(p_n-1)\log m_n-m_n^{p_n}\\
&=L_n-\frac{2-p_n}{2}m_n^{p_n}-2(p_n-1)\log m_n+2\log\frac{2\sqrt{2}}{p_n\delta\sqrt{h(x_{0,n})}}\\
&=-\log\Theta_n(m_n)+2\log\frac{2\sqrt{2}}{p_n\delta\sqrt{h(x_{0,n})}}.
\end{aligned}
\end{equation*}
This proves \eqref{fixed-radius-mean}.

\vskip 0.1cm

Taking $\delta=d_0$ in \eqref{fixed-radius-mean} and using $\overline u_n(d_0)>0$, we obtain
\begin{equation*}
\lambda_nm_n^{2(p_n-1)}\leq\Theta_n(m_n)\leq C.
\end{equation*}
Since $p_n\to2$ and $m_n\to\infty$, it follows that
\begin{equation*}
0<\lambda_n\leq\frac{C}{m_n^{2(p_n-1)}}\to0,\quad m_n^{p_n-1}\leq\frac{C}{\sqrt{\lambda_n}}.
\end{equation*}
Moreover, \eqref{core-masses} with $R=1$ yields
\begin{equation*}
M_n\geq\int_{B_{\mu_n}(x_{0,n})}f_n\mathrm{d}x\geq\frac{c}{m_n^{p_n-1}}\geq c\sqrt{\lambda_n},
\end{equation*}
which proves \eqref{uniform-max-bound}.
\end{proof}

\begin{remark} \label{rem-fixed-radius}
Fixing $\delta$ allows the term $2\log\frac{1}{\delta}$ to be absorbed into a constant independent of $n$. With this choice, \eqref{fixed-radius-mean} yields $\Theta_n(m_n)\leq C$ directly from the positivity of $\overline u_n(\delta)$.
\end{remark}

\begin{lemma} \label{lem-unbounded-sequence}
Let $u_n$ be positive solutions of \eqref{1} with $p_n\in[1,2]$ and $\beta_n\to\infty$. After passing to a subsequence,
\begin{equation} \label{bad-sequence}
p_n\to2,\quad \lambda_n\to0,\quad m_n\to\infty,\quad \frac{L_n}{m_n^{p_n}}\to0,\quad M_nm_n^{p_n-1}\to\infty.
\end{equation}
Moreover, \eqref{uniform-max-bound} holds for all sufficiently large $n$, and
\begin{equation} \label{low-error-small}
M_n\to0,\quad 0\leq\lambda_nM_n^{p_n-2}\leq C\lambda_n^{\frac{p_n}{2}}\to0.
\end{equation}
\end{lemma}

\begin{proof}
By Proposition \ref{pro:bound} and \eqref{beta-source}, we have
\begin{equation*}
M_n\leq C,\quad M_nm_n^{p_n-1}\geq c\beta_n^{\frac{p_n}{2}}\geq c\sqrt{\beta_n}\to\infty.
\end{equation*}
Thus $m_n\to\infty$. By \eqref{lambda-m-basic}, after passing to a subsequence, we may assume that $\lambda_n\to\lambda_*\geq0$.

\vskip 0.1cm

If $\lambda_*>0$, then $L_n$ is bounded and
\begin{equation*}
\frac{|L_n|}{m_n^{p_n}}\leq\frac{C}{m_n^{p_n}}\to0.
\end{equation*}
If $\lambda_*=0$, then $L_n\to\infty$. Since $0\leq\frac{p_n-1}{p_n}\leq\frac{1}{2}$, estimate \eqref{3} implies
\begin{equation*}
M_nm_n^{p_n-1}\leq C\left(\frac{m_n^{p_n}}{L_n}\right)^{\frac{p_n-1}{p_n}}\leq C\max\left\{1,\left(\frac{m_n^{p_n}}{L_n}\right)^{\frac{1}{2}}\right\}.
\end{equation*}
Together with $M_nm_n^{p_n-1}\to\infty$, this yields
\begin{equation*}
\frac{m_n^{p_n}}{L_n}\to\infty,\quad \frac{L_n}{m_n^{p_n}}\to0.
\end{equation*}

In both cases, Proposition \ref{prop-uniform-max} implies $p_n\to2$, $\lambda_n\to0$ and \eqref{uniform-max-bound}. Therefore, we have $p_n\in[\frac{3}{2},2]$ and $L_n\geq1$ for sufficiently large $n$. Applying \eqref{3}, we obtain
\begin{equation*}
0<M_n\leq CL_n^{-\frac{p_n-1}{p_n}}\leq CL_n^{-\frac{1}{3}}\to0.
\end{equation*}
Using $M_n\geq c\sqrt{\lambda_n}$ from \eqref{uniform-max-bound} and $p_n-2\leq0$, we also obtain
\begin{equation*}
0\leq\lambda_nM_n^{p_n-2}\leq C\lambda_n\lambda_n^{\frac{p_n-2}{2}}=C\lambda_n^{\frac{p_n}{2}}\to0.
\end{equation*}
This proves \eqref{bad-sequence} and \eqref{low-error-small}.
\end{proof}

\section{A maximal component and a stronger source bound} \label{sec-maximal}

In this section, we improve the bound for $M_n$ by combining estimates on a component of maximal source mass with \eqref{8}. Let $u_n$ be a sequence of positive solutions of \eqref{1} with $\beta_n\to\infty$. By Lemma \ref{lem-unbounded-sequence}, after passing to a subsequence, we may assume that \eqref{bad-sequence} and \eqref{low-error-small} hold and that $p_n\in[\frac{3}{2},2]$. Set
\begin{equation*}
\ell_n=\frac{p_n}{p_n-1},\quad 2\leq\ell_n\leq3,\quad \ell_n\to2.
\end{equation*}
We prove that $M_n^{\ell_n}L_n\to0$.

\subsection{Source outside a maximal component}\

\vskip 0.1cm

In this subsection, we select a component $D_n$ of a high superlevel set with maximal source mass $q_n$ and prove that $q_n\geq cM_n$. We then use a truncated radial vector field to bound the source in the other high components with a distance weight and derive an integral estimate on $D_n$.

\vskip 0.1cm

Fix $0<\kappa_0<2\pi$ and set
\begin{equation*}
t_{0,n}=\left(\frac{\kappa_0}{M_n}\right)^{\frac{1}{p_n-1}},\quad f_{1,n}=f_n\mathbf{1}_{\{u_n>t_{0,n}\}},\quad f_{0,n}=f_n\mathbf{1}_{\{u_n\leq t_{0,n}\}}.
\end{equation*}
For each connected component $D_{j,n}$ of $\{u_n>t_{0,n}\}$, denote
\begin{equation*}
q_{j,n}=\int_{D_{j,n}}f_n\mathrm{d}x,\quad P_{j,n}=h(F_n(u_n)-F_n(t_{0,n}))\mathbf{1}_{D_{j,n}}.
\end{equation*}
For all sufficiently large $n$, Proposition \ref{prop-localization} yields $D_{j,n}\subset\mathcal{K}\subset\subset\Omega$ and
\begin{equation*}
\operatorname{diam}D_{j,n}\leq\delta_n:=C_0\mathrm{e}^{-\frac{2\pi\kappa_0^{\frac{1}{p_n-1}}}{M_n^{\ell_n}}}.
\end{equation*}
Since $h$ is positive and of class $C^1$ in $\overline\Omega$, for sufficiently large $n$,
\begin{equation} \label{component-weight}
\frac{\sup_{D_{j,n}}h}{\inf_{D_{j,n}}h}=1+\frac{\sup_{D_{j,n}}h-\inf_{D_{j,n}}h}{\inf_{D_{j,n}}h}\leq1+C\operatorname{diam}D_{j,n}\leq1+C\delta_n.
\end{equation}

\begin{proposition} \label{prop-fixed-heavy}
There exists a constant $c_*>0$, depending only on $\Omega$, $\kappa_0$, $\|h\|_{C^2(\overline\Omega)}$ and $(\min_{\overline\Omega}h)^{-1}$, such that
\begin{equation*}
\max_jq_{j,n}\geq c_*M_n
\end{equation*}
for sufficiently large $n$.
\end{proposition}

\begin{proof}
On $\{u_n\leq t_{0,n}\}$, $\min\{u_n,t_{0,n}\}=u_n$ and $P_{j,n}=0$ for every $j$. On $D_{j,n}$, $\min\{u_n,t_{0,n}\}=t_{0,n}$ and
\begin{equation*}
hF_n(u_n)=hF_n(t_{0,n})+P_{j,n}.
\end{equation*}
Since the components $D_{j,n}$ are disjoint, we obtain
\begin{equation*}
hF_n(u_n)=hF_n\big(\min\{u_n,t_{0,n}\}\big)+\sum_jP_{j,n}.
\end{equation*}
Moreover,
\begin{equation} \label{sum_jq_{j,n}}
\sum_jq_{j,n}=\sum_j\int_{D_{j,n}}f_n\mathrm{d}x=\int_\Omega f_{1,n}\mathrm{d}x\leq M_n.
\end{equation}
Combining \eqref{2}, \eqref{11} and Proposition \ref{prop-pressure}, we obtain
\begin{equation*}
\begin{aligned}
M_n^2&\leq C\int_\Omega hF_n(u_n)\mathrm{d}x\\
&=C\int_\Omega hF_n\big(\min\{u_n,t_{0,n}\}\big)\mathrm{d}x +C\sum_j\int_{D_{j,n}}P_{j,n}\mathrm{d}x\\
&\leq C\lambda_nM_n^{p_n}+C\sum_jq_{j,n}^2.
\end{aligned}
\end{equation*}
By \eqref{low-error-small}, $C\lambda_nM_n^{p_n-2}\leq\frac{1}{2}$ for sufficiently large $n$. Then
\begin{equation*}
\sum_jq_{j,n}^2\geq\frac{M_n^2}{2C}>0.
\end{equation*}
Since $\displaystyle\sum_jq_{j,n}>0$, \eqref{sum_jq_{j,n}} yields
\begin{equation*}
\sup_jq_{j,n}\geq\frac{\displaystyle\sum_jq_{j,n}^2}{\displaystyle\sum_jq_{j,n}}\geq cM_n.
\end{equation*}
Since $\displaystyle\sum_jq_{j,n}<\infty$ and $\sup_jq_{j,n}>0$, the supremum is attained.
\end{proof}

By Proposition \ref{prop-fixed-heavy}, choose a connected component $D_n$ of $\{u_n>t_{0,n}\}$ with maximal source mass, and fix $a_n\in D_n$. Set
\begin{equation*}
q_n=\int_{D_n}f_n\mathrm{d}x,\quad \alpha_n=\frac{q_n}{M_n}.
\end{equation*}
Then $c_*\leq\alpha_n\leq1$, $q_{j,n}\leq q_n$ for every $j$, and $D_n\subset B_{\delta_n}(a_n)$. Define
\begin{equation*}
\rho_n=\frac{f_{1,n}}{M_n},\quad \omega_n=\frac{1}{M_n^2}\sum_jP_{j,n}.
\end{equation*}
Both functions are nonnegative and supported in $\mathcal{K}$, with
\begin{equation} \label{int-rho}
\int_{\mathcal{K}}\rho_n\mathrm{d}x\leq1,\quad \int_{D_n}\rho_n\mathrm{d}x=\alpha_n.
\end{equation}
Moreover, since the components are disjoint, $\omega_n=\frac{P_{j,n}}{M_n^2}$ on $D_{j,n}$. By Proposition \ref{prop-pressure} and \eqref{component-weight},
\begin{equation} \label{int-rho-1}
\begin{aligned}
\int_{D_{j,n}}\omega_n\mathrm{d}x&=\frac{1}{M_n^2}\int_{D_{j,n}}h(F_n(u_n)-F_n(t_{0,n}))\mathrm{d}x \leq\frac{\sup_{D_{j,n}}h}{\inf_{D_{j,n}}h}\frac{q_{j,n}^2}{8\pi M_n^2}
\leq(1+C\delta_n)\frac{q_{j,n}^2}{8\pi M_n^2}.
\end{aligned}
\end{equation}

For a Lipschitz vector field $Y\colon \overline\Omega\to\mathbb{R}^2$, write
\begin{equation*}
S_Y(x,y)=\frac{(Y(x)-Y(y))\cdot(x-y)}{|x-y|^2}\quad\text{for }x\neq y,\quad S_Y(x,x)=0.
\end{equation*}

\begin{lemma}
For all sufficiently large $n$ and every Lipschitz vector field $Y\colon \overline\Omega\to\mathbb{R}^2$,
\begin{equation} \label{0831-2}
\frac{1}{4\pi}\int_{\mathcal{K}}\int_{\mathcal{K}}S_Y(x,y)\rho_n(x)\rho_n(y)\mathrm{d}x\mathrm{d}y =\int_{\mathcal{K}}\operatorname{div}Y\omega_n\mathrm{d}x+\mathcal{R}_n(Y),
\end{equation}
where
\begin{equation} \label{0831-3}
|\mathcal{R}_n(Y)|\leq C\|Y\|_{L^\infty(\mathcal{K})}.
\end{equation}
The constant $C>0$ depends only on $\Omega$, $\kappa_0$, $\|h\|_{C^2(\overline\Omega)}$ and $(\min_{\overline\Omega}h)^{-1}$.
\end{lemma}

\begin{proof}
Denote $v_{0,n}=Gf_{0,n}$. Since the components $D_{j,n}$ are disjoint,
\begin{equation*}
\sum_j\int_{D_{j,n}}|f_nY\cdot\nabla u_n|\mathrm{d}x \leq\|Y\|_{L^\infty(\mathcal{K})}\|\nabla u_n\|_{L^\infty(\Omega)}M_n,
\end{equation*}
and
\begin{equation*}
\begin{aligned}
&\sum_j\int_{D_{j,n}}P_{j,n}\left|\operatorname{div}Y+Y\cdot\nabla\log h\right|\mathrm{d}x\\
&\leq\left(\|\operatorname{div}Y\|_{L^\infty(\mathcal{K})}+\|Y\|_{L^\infty(\mathcal{K})}\|\nabla\log h\|_{L^\infty(\Omega)}\right)\sum_j\int_{D_{j,n}}P_{j,n}\mathrm{d}x\\
&\leq\left(\|\operatorname{div}Y\|_{L^\infty(\mathcal{K})}+C\|Y\|_{L^\infty(\mathcal{K})}\right)\int_\Omega hF_n(u_n)\mathrm{d}x<\infty.
\end{aligned}
\end{equation*}
Summing \eqref{pressure-vector} over $D_{j,n}$, we obtain
\begin{equation*}
\int_{\mathcal{K}}f_{1,n}Y\cdot\nabla u_n\mathrm{d}x=-M_n^2\int_{\mathcal{K}}(\operatorname{div}Y+Y\cdot\nabla\log h)\omega_n\mathrm{d}x.
\end{equation*}
Since $f_{1,n}=M_n\rho_n$ and $u_n=M_nG\rho_n+v_{0,n}$,
\begin{equation} \label{nabla(Grho_n)-1}
\int_{\mathcal{K}}Y\cdot\nabla(G\rho_n)\rho_n\mathrm{d}x +\frac{1}{M_n}\int_{\mathcal{K}}Y\cdot\nabla v_{0,n}\rho_n\mathrm{d}x =-\int_{\mathcal{K}}(\operatorname{div}Y+Y\cdot\nabla\log h)\omega_n\mathrm{d}x.
\end{equation}

For each fixed $n$, the continuity of $u_n$ and $M_n>0$ imply $\rho_n=\frac{f_{1,n}}{M_n}\in L^\infty(\Omega)$. Estimates \eqref{6} and \eqref{int-rho} yield
\begin{equation*}
\begin{aligned}
\int_{\mathcal{K}}\int_{\mathcal{K}}|\nabla_xG(x,y)|\rho_n(x)\rho_n(y)\mathrm{d}x\mathrm{d}y &\leq C\|\rho_n\|_{L^\infty(\Omega)} \left(\sup_{y\in\Omega}\int_\Omega\frac{\mathrm{d}x}{|x-y|}\right) \int_{\mathcal{K}}\rho_n\mathrm{d}y \leq C\|\rho_n\|_{L^\infty(\Omega)}<\infty.
\end{aligned}
\end{equation*}
By Fubini's theorem and the symmetry of $G$, we have
\begin{equation*}
\begin{aligned}
\int_{\mathcal{K}}Y\cdot\nabla(G\rho_n)\rho_n\mathrm{d}x &=\int_{\mathcal{K}}\int_{\mathcal{K}}Y(x)\cdot\nabla_xG(x,y)\rho_n(x)\rho_n(y)\mathrm{d}x\mathrm{d}y\\
&=\frac{1}{2}\int_{\mathcal{K}}\int_{\mathcal{K}}\left[Y(x)\cdot\nabla_xG(x,y)+Y(y)\cdot\nabla_yG(x,y)\right]\rho_n(x)\rho_n(y)\mathrm{d}x\mathrm{d}y.
\end{aligned}
\end{equation*}
Since
\begin{equation*}
G(x,y)=\frac{1}{2\pi}\log\frac{1}{|x-y|}+H(x,y),
\end{equation*}
direct differentiation gives
\begin{equation*}
\nabla_xG(x,y)=-\frac{x-y}{2\pi|x-y|^2}+\nabla_xH(x,y),\quad \nabla_yG(x,y)=\frac{x-y}{2\pi|x-y|^2}+\nabla_yH(x,y).
\end{equation*}
Consequently,
\begin{equation*}
\begin{aligned}
&Y(x)\cdot\nabla_xG(x,y)+Y(y)\cdot\nabla_yG(x,y)\\
&=-\frac{(Y(x)-Y(y))\cdot(x-y)}{2\pi|x-y|^2}+Y(x)\cdot\nabla_xH(x,y)+Y(y)\cdot\nabla_yH(x,y)\\
&=-\frac{1}{2\pi}S_Y(x,y)+Y(x)\cdot\nabla_xH(x,y)+Y(y)\cdot\nabla_yH(x,y).
\end{aligned}
\end{equation*}
It follows that
\begin{equation} \label{nabla(Grho_n)-2}
\begin{aligned}
\int_{\mathcal{K}}Y\cdot\nabla(G\rho_n)\rho_n\mathrm{d}x =&-\frac{1}{4\pi}\int_{\mathcal{K}}\int_{\mathcal{K}}S_Y(x,y)\rho_n(x)\rho_n(y)\mathrm{d}x\mathrm{d}y\\
&+\frac{1}{2}\int_{\mathcal{K}}\int_{\mathcal{K}}\left[Y(x)\cdot\nabla_xH(x,y)+Y(y)\cdot\nabla_yH(x,y)\right]\rho_n(x)\rho_n(y)\mathrm{d}x\mathrm{d}y.
\end{aligned}
\end{equation}
Substituting \eqref{nabla(Grho_n)-2} into \eqref{nabla(Grho_n)-1} and rearranging, we obtain
\begin{equation*}
\begin{aligned}
&\frac{1}{4\pi}\int_{\mathcal{K}}\int_{\mathcal{K}}S_Y(x,y)\rho_n(x)\rho_n(y)\mathrm{d}x\mathrm{d}y-\int_{\mathcal{K}}\operatorname{div}Y\omega_n\mathrm{d}x\\
&=\frac{1}{2}\int_{\mathcal{K}}\int_{\mathcal{K}}\left[Y(x)\cdot\nabla_xH(x,y)+Y(y)\cdot\nabla_yH(x,y)\right]\rho_n(x)\rho_n(y)\mathrm{d}x\mathrm{d}y\\
&\quad+\frac{1}{M_n}\int_{\mathcal{K}}Y\cdot\nabla v_{0,n}\rho_n\mathrm{d}x+\int_{\mathcal{K}}\omega_nY\cdot\nabla\log h\mathrm{d}x\\
&=:\mathcal{R}_n(Y).
\end{aligned}
\end{equation*}
This proves \eqref{0831-2}.

\vskip 0.1cm

Moreover, by \eqref{sum_jq_{j,n}} and \eqref{int-rho-1}, we have
\begin{equation*}
\int_{\mathcal{K}}\omega_n\mathrm{d}x\leq\frac{1+C\delta_n}{8\pi M_n^2}\sum_jq_{j,n}^2\leq C.
\end{equation*}
Therefore, the smoothness of $H$ on $\mathcal{K}\times\mathcal{K}$, \eqref{12} and \eqref{int-rho} give
\begin{equation*}
\begin{aligned}
|\mathcal{R}_n(Y)| &\leq C\|Y\|_{L^\infty(\mathcal{K})} \left(\int_{\mathcal{K}}\rho_n\mathrm{d}x\right)^2 +\frac{\|Y\|_{L^\infty(\mathcal{K})}}{M_n} \|\nabla v_{0,n}\|_{L^\infty(\Omega)}\int_{\mathcal{K}}\rho_n\mathrm{d}x +C\|Y\|_{L^\infty(\mathcal{K})}\int_{\mathcal{K}}\omega_n\mathrm{d}x\\
&\leq C(1+\lambda_nM_n^{p_n-2})\|Y\|_{L^\infty(\mathcal{K})} \leq C\|Y\|_{L^\infty(\mathcal{K})},
\end{aligned}
\end{equation*}
where the last inequality follows from \eqref{low-error-small}.
\end{proof}

We have the following source isolation estimate.

\begin{proposition} \label{prop-max-iso}
For all sufficiently large $n$ and every $r\geq32\delta_n$,
\begin{equation} \label{0831-4}
\int_{\mathcal{K}\setminus D_n}\frac{\rho_n(x)}{\max\{r,|x-a_n|\}}\mathrm{d}x\leq C
\end{equation}
and
\begin{equation} \label{0831-5}
-C\delta_nq_n^2\leq\frac{q_n^2}{8\pi}-\int_{D_n}P_{D_n}\mathrm{d}x\leq CM_n^2r.
\end{equation}
The constant $C>0$ depends only on $\Omega$, $\kappa_0$, $\|h\|_{C^2(\overline\Omega)}$ and $(\min_{\overline\Omega}h)^{-1}$, and is independent of $n$, $\lambda_n$, $D_n$, $a_n$ and $r$.
\end{proposition}

\begin{proof}
Set
\begin{equation*}
Y_{r,n}(x)=\frac{x-a_n}{\max\{r,|x-a_n|\}},\quad k_{r,n}(x)=\frac{1}{\max\{r,|x-a_n|\}}.
\end{equation*}
For $|x-a_n|\neq r$,
\begin{equation*}
DY_{r,n}(x)=\begin{cases}
\frac{1}{r}I_2, & |x-a_n|<r,\\
\frac{1}{|x-a_n|}I_2-\frac{(x-a_n)\otimes(x-a_n)}{|x-a_n|^3}, & |x-a_n|>r.
\end{cases}
\end{equation*}
Also, we have
\begin{equation*}
\operatorname{div}Y_{r,n}(x)=\begin{cases}
\frac{2}{r}, & |x-a_n|<r,\\
\frac{1}{|x-a_n|}, & |x-a_n|>r.
\end{cases}
\end{equation*}
For $|x-a_n|>r$ and $v\in\mathbb{R}^2$, the Cauchy-Schwarz inequality yields
\begin{equation*}
\begin{aligned}
0\leq v\cdot DY_{r,n}(x)v&=\frac{|v|^2}{|x-a_n|}-\frac{((x-a_n)\cdot v)^2}{|x-a_n|^3}\leq\frac{|v|^2}{r},
\end{aligned}
\end{equation*}and
\begin{equation*}
\begin{aligned}
|DY_{r,n}(x)v|^2&=\frac{|v|^2}{|x-a_n|^2}-\frac{((x-a_n)\cdot v)^2}{|x-a_n|^4}\leq\frac{|v|^2}{r^2}.
\end{aligned}
\end{equation*}
Both bounds also hold in $B_r(a_n)$. Using
\begin{equation*}
Y_{r,n}(x)-Y_{r,n}(y)=\int_0^1DY_{r,n}(y+t(x-y))(x-y)\mathrm{d}t,
\end{equation*}
we have
\begin{equation*}
0\leq(Y_{r,n}(x)-Y_{r,n}(y))\cdot(x-y)\leq\frac{|x-y|^2}{r}, \quad |Y_{r,n}(x)-Y_{r,n}(y)|\leq\frac{|x-y|}{r}.
\end{equation*}
Therefore, $Y_{r,n}$ is Lipschitz and satisfies
\begin{equation*}
\begin{aligned}
0\leq S_{Y_{r,n}}(x,y)\leq\frac{1}{r},\quad \|Y_{r,n}\|_{L^\infty(\mathbb{R}^2)} \leq1,\quad
0\leq\operatorname{div}Y_{r,n}\leq2k_{r,n},
\end{aligned}
\end{equation*}almost everywhere in $\mathbb{R}^2$.
Since $|x-a_n|\leq\delta_n<r$ for $x\in D_n$,
\begin{equation} \label{div}
\operatorname{div}Y_{r,n}=\frac{2}{r}\quad\text{in }D_n,\quad S_{Y_{r,n}}(x,y)=\frac{1}{r}\quad\text{for }x,y\in D_n,\ x\neq y.
\end{equation}

Let $\varepsilon_n=\frac{\delta_n}{r}\leq\frac{1}{32}$. Fix $x\in D_n$ and $y\in\mathcal{K}\setminus D_n$. If $|y-a_n|\leq r$, then
\begin{equation*}
S_{Y_{r,n}}(x,y)=\frac{1}{r}=k_{r,n}(y).
\end{equation*}
If $t=|y-a_n|>r$, then $\delta_n=\varepsilon_nr\leq\varepsilon_nt$. Using $|x-a_n|\leq\delta_n$ and the Cauchy-Schwarz inequality, we obtain
\begin{equation*}
\begin{aligned}
(Y_{r,n}(y)-Y_{r,n}(x))\cdot(y-x)&=t+\frac{|x-a_n|^2}{r}-\left(\frac{1}{r}+\frac{1}{t}\right)(x-a_n)\cdot(y-a_n)\\
&\geq t-\left(\frac{1}{r}+\frac{1}{t}\right)\delta_nt\geq t(1-2\varepsilon_n).
\end{aligned}
\end{equation*}
Moreover,
\begin{equation*}
|x-y|\leq|x-a_n|+|y-a_n|\leq\delta_n+t\leq t(1+\varepsilon_n).
\end{equation*}
Therefore,
\begin{equation*}
S_{Y_{r,n}}(x,y)\geq\frac{1-2\varepsilon_n}{t(1+\varepsilon_n)^2}\geq\frac{1-4\varepsilon_n}{t}=(1-4\varepsilon_n)k_{r,n}(y).
\end{equation*}
Set
\begin{equation*}
J_{r,n}=\int_{\mathcal{K}\setminus D_n}k_{r,n}(x)\rho_n(x)\mathrm{d}x.
\end{equation*}
Since $S_{Y_{r,n}}$ is symmetric and nonnegative, \eqref{int-rho} and \eqref{div} yield
\begin{equation} \label{0831-6}
\begin{aligned}
&\frac{1}{4\pi}\int_{\mathcal{K}}\int_{\mathcal{K}}S_{Y_{r,n}}(x,y)\rho_n(x)\rho_n(y)\mathrm{d}y\mathrm{d}x\\
&\geq\frac{1}{4\pi}\int_{D_n}\int_{D_n} S_{Y_{r,n}}(x,y)\rho_n(x)\rho_n(y)\mathrm{d}y\mathrm{d}x+\frac{1}{2\pi}\int_{D_n}\int_{\mathcal{K}\setminus D_n}S_{Y_{r,n}}(x,y)\rho_n(x)\rho_n(y)\mathrm{d}y\mathrm{d}x \\
&=\frac{1}{4\pi r}\left(\int_{D_n}\rho_n\mathrm{d}x\right)^2+\frac{1}{2\pi}\int_{D_n}\int_{\mathcal{K}\setminus D_n}S_{Y_{r,n}}(x,y)\rho_n(x)\rho_n(y)\mathrm{d}y\mathrm{d}x\\
&\geq\frac{1}{4\pi r}\left(\int_{D_n}\rho_n\mathrm{d}x\right)^2+\frac{1-4\varepsilon_n}{2\pi}\left(\int_{D_n}\rho_n\mathrm{d}x\right)\left(\int_{\mathcal{K}\setminus D_n}k_{r,n}\rho_n\mathrm{d}x\right)\\
&=\frac{\alpha_n^2}{4\pi r}+\frac{\alpha_n}{2\pi}(1-4\varepsilon_n)J_{r,n}.
\end{aligned}
\end{equation}

On the other hand, for $x,y\in D_{j,n}$,
\begin{equation*}
\max\{r,|x-a_n|\}\leq\max\{r,|y-a_n|\}+\delta_n\leq(1+\varepsilon_n)\max\big\{r,|y-a_n|\big\}.
\end{equation*}
Since $\rho_n=\frac{f_n}{M_n}$ in $D_{j,n}$, we obtain
\begin{equation*}
\begin{aligned}
\sup_{D_{j,n}}k_{r,n}&\leq(1+\varepsilon_n)\inf_{D_{j,n}}k_{r,n},\quad
\int_{D_{j,n}}k_{r,n}\rho_n\mathrm{d}x \geq\left(\inf_{D_{j,n}}k_{r,n}\right)\int_{D_{j,n}}\rho_n\mathrm{d}x=\frac{q_{j,n}}{M_n}\inf_{D_{j,n}}k_{r,n}.
\end{aligned}
\end{equation*}
For $D_{j,n}\neq D_n$, Proposition \ref{prop-pressure}, \eqref{component-weight} and $q_{j,n}\leq q_n=\alpha_nM_n$ imply
\begin{equation*}
\begin{aligned}
\int_{D_{j,n}}\operatorname{div}Y_{r,n}\omega_n\mathrm{d}x&\leq2\sup_{D_{j,n}}k_{r,n}\int_{D_{j,n}}\omega_n\mathrm{d}x
\leq\frac{(1+C\delta_n)q_{j,n}^2}{4\pi M_n^2}\sup_{D_{j,n}}k_{r,n}\\
&\leq\frac{(1+\varepsilon_n)(1+C\delta_n)q_{j,n}^2}{4\pi M_n^2}\inf_{D_{j,n}}k_{r,n}\\
&\leq\frac{(1+\varepsilon_n)(1+C\delta_n)q_{j,n}}{4\pi M_n}\int_{D_{j,n}}k_{r,n}\rho_n\mathrm{d}x\\
&\leq\frac{(1+\varepsilon_n)(1+C\delta_n)\alpha_n}{4\pi}\int_{D_{j,n}}k_{r,n}\rho_n\mathrm{d}x.
\end{aligned}
\end{equation*}
Summing over $D_{j,n}\neq D_n$ and using \eqref{div}, we obtain
\begin{equation*}
\int_{\mathcal{K}}\operatorname{div}Y_{r,n}\omega_n\mathrm{d}x
\leq\frac{2}{r}\int_{D_n}\omega_n\mathrm{d}x+\frac{(1+\varepsilon_n)(1+C\delta_n)\alpha_n}{4\pi}J_{r,n}.
\end{equation*}

Combining \eqref{0831-2}-\eqref{0831-3} with \eqref{0831-6}, we obtain
\begin{equation} \label{max-iso}
\frac{2}{r}\left(\frac{\alpha_n^2}{8\pi}-\int_{D_n}\omega_n\mathrm{d}x\right)+\frac{\alpha_n}{4\pi}(1-9\varepsilon_n-C\delta_n)J_{r,n}\leq C.
\end{equation}
By \eqref{int-rho-1} and $q_n=\alpha_nM_n$,
\begin{equation*}
\begin{aligned}
\int_{D_n}\omega_n\mathrm{d}x\leq\frac{\sup_{D_n}h}{\inf_{D_n}h}\frac{q_n^2}{8\pi M_n^2}
\leq(1+C\delta_n)\frac{\alpha_n^2}{8\pi}.
\end{aligned}
\end{equation*}
Since $\alpha_n\leq1$, it follows that
\begin{equation*}
\frac{\alpha_n^2}{8\pi}-\int_{D_n}\omega_n\mathrm{d}x\geq-C\delta_n\alpha_n^2\geq-C\delta_n.
\end{equation*}
Since $\varepsilon_n\leq\frac{1}{32}$ and $\delta_n\to0$, for sufficiently large $n$,
\begin{equation*}
\frac{\alpha_n}{4\pi}(1-9\varepsilon_n-C\delta_n)\geq\frac{\alpha_n}{8\pi}\geq\frac{c_*}{8\pi}.
\end{equation*}
Substituting these bounds into \eqref{max-iso} and using $r\geq32\delta_n$, we obtain
\begin{equation*}
\frac{c_*}{8\pi}J_{r,n}\leq C+\frac{C\delta_n}{r}\leq C.
\end{equation*}
This proves \eqref{0831-4}.

\vskip 0.1cm

Finally, dropping the nonnegative term in \eqref{max-iso}, we obtain
\begin{equation*}
\frac{2}{r}\left(\frac{\alpha_n^2}{8\pi}-\int_{D_n}\omega_n\mathrm{d}x\right)\leq C.
\end{equation*}
Since $q_n=\alpha_nM_n$ and $M_n^2\omega_n=P_{D_n}$ in $D_n$,
\begin{equation*}
\frac{q_n^2}{8\pi}-\int_{D_n}P_{D_n}\mathrm{d}x=M_n^2\left(\frac{\alpha_n^2}{8\pi}-\int_{D_n}\omega_n\mathrm{d}x\right)\leq CM_n^2r.
\end{equation*}
On the other hand, Proposition \ref{prop-pressure} and \eqref{component-weight} imply
\begin{equation*}
\begin{aligned}
\frac{q_n^2}{8\pi}-\int_{D_n}P_{D_n}\mathrm{d}x&\geq\left(1-\frac{\sup_{D_n}h}{\inf_{D_n}h}\right)\frac{q_n^2}{8\pi}
\geq-C\delta_nq_n^2.
\end{aligned}
\end{equation*}
This proves \eqref{0831-5}.
\end{proof}

\begin{remark}
The diameter estimate does not control the distances between distinct components. The truncated field $Y_{r,n}$ yields the error bound $CM_n^2r$ in \eqref{0831-5} without any separation assumption on the components.
\end{remark}

\subsection{A stronger source bound}\

\vskip 0.1cm

In this subsection, we prove that $M_n^{\ell_n}L_n\to0$. We first derive an integral estimate on $D_n$. Combining this estimate with the bound for the source outside $D_n$ and \eqref{8}, we rule out any subsequence along which $M_n^{\ell_n}L_n$ stays bounded away from zero.

\vskip 0.1cm

Set
\begin{equation*}
A_{D_n}=2\int_{D_n}hF_n(u_n)\mathrm{d}x,\quad J_{D_n}=\frac{2\lambda_n}{q_n^2}\int_{D_n}h\mathrm{e}^{u_n^{p_n}}\mathrm{d}x.
\end{equation*}
By the definition of $F_n$, we have
\begin{equation*}
J_{D_n}-\frac{A_{D_n}}{q_n^2}=\frac{2\lambda_n}{q_n^2}\int_{D_n}h\mathrm{d}x.
\end{equation*}

We have the following integral estimates.

\begin{lemma} \label{lem-max-inverse}
There exists a constant $C>0$, depending only on $\Omega$, $\kappa_0$, $\|h\|_{C^2(\overline\Omega)}$ and $(\min_{\overline\Omega}h)^{-1}$, such that, for all sufficiently large $n$,
\begin{equation} \label{0831-7}
\frac{1}{4\pi}-C\delta_n\leq\frac{A_{D_n}}{q_n^2}\leq\frac{1}{4\pi}+C\delta_n+C\lambda_nM_n^{p_n-2}
\end{equation}
and
\begin{equation} \label{0831-8}
\left|J_{D_n}-\frac{1}{4\pi}\right| \leq C\left(\delta_n+\lambda_nM_n^{p_n-2}\right).
\end{equation}
\end{lemma}

\begin{proof}
Taking $r=32\delta_n$ in \eqref{0831-5} and using $q_n\geq c_*M_n$, we obtain
\begin{equation*}
\left|\frac{1}{4\pi}-\frac{2}{q_n^2}\int_{D_n}P_{D_n}\mathrm{d}x\right|\leq C\delta_n+\frac{CM_n^2\delta_n}{q_n^2}\leq C\delta_n.
\end{equation*}
Since $hF_n(u_n)=P_{D_n}+hF_n(t_{0,n})$ in $D_n$, we have
\begin{equation*}
\begin{aligned}
A_{D_n}&=2\int_{D_n}P_{D_n}\mathrm{d}x+2F_n(t_{0,n})\int_{D_n}h\mathrm{d}x\\
&=2\int_{D_n}P_{D_n}\mathrm{d}x+2\lambda_n(\mathrm{e}^{t_{0,n}^{p_n}}-1)\int_{D_n}h\mathrm{d}x.
\end{aligned}
\end{equation*}
In $D_n$, $F_n(\min\{u_n,t_{0,n}\})=F_n(t_{0,n})$. By \eqref{11} and $q_n\geq c_*M_n$,
\begin{equation} \label{0831-111}
\begin{aligned}
0\leq2\lambda_n(\mathrm{e}^{t_{0,n}^{p_n}}-1)\int_{D_n}h\mathrm{d}x
&=2\int_{D_n}hF_n(\min\{u_n,t_{0,n}\})\mathrm{d}x\\
&\leq2\int_\Omega hF_n(\min\{u_n,t_{0,n}\})\mathrm{d}x\\
&\leq C\lambda_nM_n^{p_n}\leq C\lambda_nM_n^{p_n-2}q_n^2.
\end{aligned}
\end{equation}
This proves \eqref{0831-7}.

\vskip 0.1cm

Since $t_{0,n}^{p_n}\to\infty$, we have $\mathrm{e}^{t_{0,n}^{p_n}}-1\geq1$ for sufficiently large $n$. Using \eqref{0831-111}, we obtain
\begin{equation*}
\begin{aligned}
0\leq J_{D_n}-\frac{A_{D_n}}{q_n^2}
&=\frac{2\lambda_n}{q_n^2}\int_{D_n}h\mathrm{d}x \leq\frac{2\lambda_n(\mathrm{e}^{t_{0,n}^{p_n}}-1)}{q_n^2}\int_{D_n}h\mathrm{d}x
\leq C\lambda_nM_n^{p_n-2}.
\end{aligned}
\end{equation*}
Together with \eqref{0831-7}, this proves \eqref{0831-8}.
\end{proof}

Estimate \eqref{3} gives $M_n^{\ell_n}L_n\leq C$. We next show that this quantity tends to zero.

\begin{proposition} \label{prop-strong-scale}
Under \eqref{bad-sequence} and \eqref{low-error-small}, it holds
\begin{equation} \label{0831-9}
M_n^{\ell_n}L_n\to0 \quad \text{as } n\to\infty.
\end{equation}
Consequently, for every fixed $c,N>0$,
\begin{equation} \label{0831-10}
\mathrm{e}^{-\frac{c}{M_n^{\ell_n}}}=o(\lambda_n^N) \quad \text{as } n\to\infty.
\end{equation}
\end{proposition}

\begin{proof}
If \eqref{0831-9} fails, then there exist a constant $\varepsilon_0>0$ and a subsequence, still denoted by $n$, such that
\begin{equation*}
M_n^{\ell_n}L_n\geq\varepsilon_0>0.
\end{equation*}
Fix $0<\kappa_0<2\pi$ and set
\begin{equation*}
\begin{aligned}
t_{0,n}&=\left(\frac{\kappa_0}{M_n}\right)^{\frac{1}{p_n-1}},\quad \delta_n=C_0\mathrm{e}^{-\frac{2\pi\kappa_0^{\frac{1}{p_n-1}}}{M_n^{\ell_n}}}, \quad
f_{1,n} =f_n\mathbf{1}_{\{u_n>t_{0,n}\}},\quad f_{0,n}=f_n-f_{1,n}.
\end{aligned}
\end{equation*}
Choose a connected component $D_n$ of $\{u_n>t_{0,n}\}$ with maximal source mass. Fix $a_n\in D_n$ and set
\begin{equation*}
q_n=\int_{D_n}f_n\mathrm{d}x,\quad J_{D_n}=\frac{2\lambda_n}{q_n^2}\int_{D_n}h\mathrm{e}^{u_n^{p_n}}\mathrm{d}x.
\end{equation*}
By Proposition \ref{prop-fixed-heavy}, we can deduce
\begin{equation} \label{0831-11}
q_n\geq c_*M_n,\quad q_n^{\ell_n}L_n\geq c_*^{\ell_n}M_n^{\ell_n}L_n\geq c_*^3\varepsilon_0=:\varepsilon>0,
\end{equation}
where we used $c_*\leq1$ and $\ell_n\leq3$. By \eqref{low-error-small}, $M_n\to0$ and hence $\delta_n\to0$.

\vskip 0.1cm

Fix $0<r<\frac{1}{4}\operatorname{dist}(\mathcal{K},\partial\Omega)$.
For sufficiently large $n$, we have $32\delta_n<r$ and $D_n\subset B_{\delta_n}(a_n)\subset B_r(a_n)$. Set
\begin{equation*}
Q_n=\int_{B_{2r}(a_n)}f_n\mathrm{d}x.
\end{equation*}
Since $D_n\subset B_{2r}(a_n)$, we have $Q_n\geq q_n>0$. For $x\in B_{2r}(a_n)$,
$
\max\{r,|x-a_n|\}\leq2r.
$
Proposition \ref{prop-max-iso} and \eqref{10} imply
\begin{equation*}
\begin{aligned}
0\leq Q_n-q_n&=\int_{B_{2r}(a_n)\setminus D_n}f_n\mathrm{d}x\\
&\leq\int_{B_{2r}(a_n)\setminus D_n}f_{1,n}\mathrm{d}x+\int_\Omega f_{0,n}\mathrm{d}x\\
&\leq2r\int_{\mathcal{K}\setminus D_n}\frac{f_{1,n}(x)}{\max\{r,|x-a_n|\}}\mathrm{d}x+C\lambda_nM_n^{p_n-1}\\
&\leq CM_n(r+\lambda_nM_n^{p_n-2}).
\end{aligned}
\end{equation*}
Using \eqref{0831-11}, we obtain
\begin{equation} \label{0831-13}
1\leq\frac{Q_n}{q_n}\leq1+\frac{CM_n}{q_n}(r+\lambda_nM_n^{p_n-2})\leq1+C(r+\lambda_nM_n^{p_n-2}).
\end{equation}
Set
\begin{equation*}
W_n=G(f_n\mathbf{1}_{B_{2r}(a_n)}),\quad w_n=\frac{W_n}{Q_n}.
\end{equation*}
For $x\in D_n$ and $y\notin B_{2r}(a_n)$, we have $|x-y|\geq2r-\delta_n\geq r$. Hence \eqref{7} gives
\begin{equation*}
\begin{aligned}
0\leq u_n(x)-W_n(x) &=\int_{\Omega\setminus B_{2r}(a_n)}G(x,y)f_n(y)\mathrm{d}y\\
&\leq\frac{1}{2\pi}\log\frac{d_\Omega}{r} \int_{\Omega\setminus B_{2r}(a_n)}f_n(y)\mathrm{d}y
\leq C_rM_n\leq C_rq_n.
\end{aligned}
\end{equation*}
Using $W_n=Q_nw_n$ and $q_n\leq Q_n$, we obtain
\begin{equation} \label{0831-14}
u_n\leq Q_n(w_n+C_r)\quad\text{in }D_n.
\end{equation}
The source $f_n\mathbf{1}_{B_{2r}(a_n)}/Q_n$ has integral one. For every fixed $0<b<4\pi$, \eqref{8} therefore yields
\begin{equation} \label{0831-15}
\int_\Omega\mathrm{e}^{bw_n}\mathrm{d}x\leq C_b.
\end{equation}

Choose $\tau\in(0,4\pi)$ such that $2\tau(4\pi+\tau)<\frac{\varepsilon}{8}$ and set $\kappa=4\pi+\tau$, $b=4\pi-\tau$. With $C$ as in \eqref{0831-13}, choose $r>0$ smaller if necessary so that
\begin{equation*}
\frac{2Cr}{1+2Cr}b\kappa<\frac{\varepsilon}{8}.
\end{equation*}
For sufficiently large $n$, we have $\lambda_nM_n^{p_n-2}<r$. Hence \eqref{0831-13} implies
$
\frac{1}{1+2Cr}\leq\frac{q_n}{Q_n}\leq1.
$
Set
\begin{equation*}
K_n=\left(\frac{2\kappa}{p_n}\right)^{\frac{1}{p_n-1}}\to\kappa.
\end{equation*}
Since $p_n\to2$ and $\kappa-b=2\tau$,
\begin{equation*}
\begin{aligned}
K_n^{p_n}-\frac{bK_n}{1+2Cr}&\to\kappa^2-\frac{b\kappa}{1+2Cr}=2\tau\kappa+\frac{2Cr}{1+2Cr}b\kappa<\frac{\varepsilon}{4}.
\end{aligned}
\end{equation*}
Consequently, for sufficiently large $n$,
\begin{equation*}
K_n^{p_n}-b\frac{q_n}{Q_n}K_n\leq K_n^{p_n}-\frac{bK_n}{1+2Cr}<\frac{\varepsilon}{4}.
\end{equation*}
For $0\leq z\leq K_n$, we have $z^{p_n}\leq K_n^{p_n-1}z$, and hence
\begin{equation*}
z^{p_n}-b\frac{q_n}{Q_n}z\leq\frac{z}{K_n}\left(K_n^{p_n}-b\frac{q_n}{Q_n}K_n\right)
\leq\frac{z}{K_n}\frac{\varepsilon}{4}\leq\frac{\varepsilon}{4}.
\end{equation*}
Set $z=q_n^{\frac{1}{p_n-1}}u_n$ in $D_n$. On $D_n\cap\{z\leq K_n\}$, using \eqref{0831-11}, we obtain
\begin{equation*}
u_n^{p_n}-\frac{bu_n}{Q_n}=\frac{z^{p_n}-b\frac{q_n}{Q_n}z}{q_n^{\ell_n}}
\leq\frac{\varepsilon}{4q_n^{\ell_n}},\quad
-L_n+\frac{\varepsilon}{4q_n^{\ell_n}}\leq-\frac{3\varepsilon}{4q_n^{\ell_n}}.
\end{equation*}
Combining these inequalities with \eqref{0831-14}-\eqref{0831-15}, we obtain
\begin{equation} \label{0831-16}
\begin{aligned}
\frac{2\lambda_n}{q_n^2}\int_{D_n\cap\{z\leq K_n\}}h\mathrm{e}^{u_n^{p_n}}\mathrm{d}x
&\leq\frac{2\lambda_n\|h\|_{L^\infty(\Omega)}}{q_n^2}\mathrm{e}^{\frac{\varepsilon}{4q_n^{\ell_n}}}\int_{D_n}\mathrm{e}^{\frac{bu_n}{Q_n}}\mathrm{d}x\\
&\leq\frac{2\lambda_n\|h\|_{L^\infty(\Omega)}}{q_n^2}\mathrm{e}^{\frac{\varepsilon}{4q_n^{\ell_n}}+bC_r}\int_\Omega\mathrm{e}^{bw_n}\mathrm{d}x\\
&\leq C_{b,r}q_n^{-2}\mathrm{e}^{-\frac{3\varepsilon}{4q_n^{\ell_n}}}.
\end{aligned}
\end{equation}
On $D_n\cap\{z>K_n\}$, we have $p_nq_nu_n^{p_n-1}=p_nz^{p_n-1}>p_nK_n^{p_n-1}=2\kappa$. Thus
\begin{equation} \label{0831-17}
\begin{aligned}
\frac{2\lambda_n}{q_n^2}\int_{D_n\cap\{z>K_n\}}h\mathrm{e}^{u_n^{p_n}}\mathrm{d}x
&\leq\frac{p_n\lambda_n}{\kappa q_n}\int_{D_n}hu_n^{p_n-1}\mathrm{e}^{u_n^{p_n}}\mathrm{d}x =\frac{1}{\kappa q_n}\int_{D_n}f_n\mathrm{d}x=\frac{1}{\kappa}.
\end{aligned}
\end{equation}
By the definition of $J_{D_n}$, estimates \eqref{0831-16} and \eqref{0831-17} yield
\begin{equation*}
J_{D_n}\leq C_{b,r}q_n^{-2}\mathrm{e}^{-\frac{3\varepsilon}{4q_n^{\ell_n}}}+\frac{1}{\kappa}
=\frac{1}{\kappa}+o(1).
\end{equation*}
Lemma \ref{lem-max-inverse} then gives
\begin{equation*}
\frac{1}{4\pi}=\lim_{n\to\infty}J_{D_n}\leq\frac{1}{\kappa}<\frac{1}{4\pi},
\end{equation*}
which is a contradiction.

\vskip 0.1cm

Finally, for fixed $c,N>0$, we have $NM_n^{\ell_n}L_n\leq\frac{c}{2}$ for sufficiently large $n$ by \eqref{0831-9}. Hence
\begin{equation*}
\log\frac{\mathrm{e}^{-\frac{c}{M_n^{\ell_n}}}}{\lambda_n^N}=-\frac{c}{M_n^{\ell_n}}+NL_n\leq-\frac{c}{2M_n^{\ell_n}}\to-\infty,
\end{equation*}
which proves \eqref{0831-10}.
\end{proof}

\section{Estimates near a maximal component} \label{sec-background}

In this section, we consider a sequence of positive solutions of \eqref{1} with $\beta_n\to\infty$ and select a high superlevel component $D_n$ with maximal source mass.
Using $M_n^{\ell_n}L_n\to0$ and the estimates for the source outside $D_n$, we control the Green potential generated by this source and its oscillation near $D_n$. Combining these bounds with Lemma \ref{lem-max-inverse}, we prove that the integral of $u_n^{p_n-1}f_n$ over $D_n$ tends to $4\pi$.
We then use local blow-up analysis to relate the maximum of $u_n$ on $D_n$ to the source mass of this component.

\subsection{The maximal component}\

\vskip 0.1cm

In this subsection, we select a component $D_n$ with maximal source mass and record the estimates for the source outside $D_n$.

\vskip 0.1cm
Assume that a sequence of positive solutions $u_n$ of \eqref{1}, with exponents $p_n$ and parameters $\lambda_n$, satisfies
\begin{equation} \label{1005-1}
\beta_n\to\infty.
\end{equation}
By Lemma \ref{lem-unbounded-sequence}, after passing to a subsequence, we may assume that $p_n\in[\frac{3}{2},2]$ and that \eqref{bad-sequence} and \eqref{low-error-small} hold. Set $\ell_n=\frac{p_n}{p_n-1}$, fix $0<\kappa_0<2\pi$ and define
\begin{equation*}
t_{0,n}=\left(\frac{\kappa_0}{M_n}\right)^{\frac{1}{p_n-1}},\quad f_{1,n}=f_n\mathbf{1}_{\{u_n>t_{0,n}\}},\quad f_{0,n}=f_n-f_{1,n}.
\end{equation*}
In particular,
\begin{equation*}
m_n\to\infty,\quad \frac{L_n}{m_n^{p_n}}\to0,\quad m_n^{p_n-1}M_n\to\infty.
\end{equation*}
Now let $x_{0,n}$ be a global maximum point of $u_n$ and fix $0<\delta_0<\frac{\delta_*}{4}$.
By Lemma \ref{lem-interior} and Proposition \ref{prop-localization},
\begin{equation*}
\operatorname{dist}(x_{0,n},\partial\Omega)\geq\delta_*,\quad \partial B_{\delta_0}(x_{0,n})\cup\operatorname{supp}f_{1,n}\subset\mathcal{K}.
\end{equation*}
For two distinct $x,y\in\mathcal{K}$, we have $G(x,y)\geq c_{\mathcal{K}}>0$. Since $\operatorname{supp}f_{1,n}\subset\mathcal{K}$, the Green representation formula gives
\begin{equation*}
\begin{aligned}
u_n(x)&=\int_\Omega G(x,y)f_n(y)\mathrm{d}y\geq\int_{\mathcal{K}}G(x,y)f_{1,n}(y)\mathrm{d}y\geq c_{\mathcal{K}}\int_\Omega f_{1,n}\mathrm{d}y
\end{aligned}
\end{equation*}
for every $x\in\partial B_{\delta_0}(x_{0,n})\subset\mathcal{K}$.
Using \eqref{10} and \eqref{low-error-small}, we obtain
\begin{equation*}
\begin{aligned}
\overline u_n(\delta_0)&\geq c_{\mathcal{K}}\left(M_n-\int_\Omega f_{0,n}\mathrm{d}x\right)
\geq c_{\mathcal{K}}(1-C\lambda_nM_n^{p_n-2})M_n\geq cM_n
\end{aligned}
\end{equation*}
for sufficiently large $n$. Using this lower bound and \eqref{fixed-radius-mean}, we obtain
\begin{equation*}
cm_n^{p_n-1}M_n\leq\frac{p_n}{2}m_n^{p_n-1}\overline u_n(\delta_0)\leq-\log\Theta_n(m_n)+C,
\end{equation*}
where
\begin{equation*}
\Theta_n(s)=\lambda_ns^{2(p_n-1)}\mathrm{e}^{\frac{2-p_n}{2}s^{p_n}}.
\end{equation*}
Since $m_n^{p_n-1}M_n\to\infty$ by \eqref{bad-sequence} and $p_n\leq2$, we have
\begin{equation} \label{1005-2}
\Theta_n(m_n)\leq C\mathrm{e}^{-cm_n^{p_n-1}M_n}\to0,\quad m_n^{p_n-1}\sqrt{\lambda_n}\leq\sqrt{\Theta_n(m_n)}\to0.
\end{equation}
Thus it holds
\begin{equation} \label{bound-1}
\frac{M_n}{\sqrt{\lambda_n}}=\frac{m_n^{p_n-1}M_n}{m_n^{p_n-1}\sqrt{\lambda_n}}\to\infty.
\end{equation}

Set
\begin{equation*}
\delta_n=C_0\mathrm{e}^{-\frac{2\pi\kappa_0^{\frac{1}{p_n-1}}}{M_n^{\ell_n}}},\quad s_n=\sqrt{\delta_n}.
\end{equation*}
Since $p_n\in[\frac{3}{2},2]$ for sufficiently large $n$,
\begin{equation*}
\kappa_0^{\frac{1}{p_n-1}}\geq\min\{\kappa_0,\kappa_0^2\}>0.
\end{equation*}
By Proposition \ref{prop-strong-scale}, for every fixed $N>0$,
\begin{equation*}
\delta_n=o(\lambda_n^N),\quad s_n=o(\lambda_n^N).
\end{equation*}
Combining \eqref{3} and \eqref{uniform-max-bound}, we obtain
\begin{equation} \label{bound-2}
m_n^{p_n-1}M_ns_n\leq C\frac{s_n}{\sqrt{\lambda_n}L_n^{\frac{p_n-1}{p_n}}}\to0.
\end{equation}
Choose a connected component $D_n$ of $\{u_n>t_{0,n}\}$ with maximal source mass and set
\begin{equation*}
q_n=\int_{D_n}f_n\mathrm{d}x,\quad \rho_n=\frac{f_{1,n}}{M_n}.
\end{equation*}
Proposition \ref{prop-fixed-heavy} yields
\begin{equation} \label{1005-4}
c_*M_n\leq q_n\leq M_n.
\end{equation}

Fix $a_n\in D_n$. Then
\begin{equation*}
D_n\subset B_{\delta_n}(a_n)\subset B_{s_n}(a_n),\quad \frac{\delta_n}{s_n}=s_n\to0.
\end{equation*}
In particular, $s_n\geq32\delta_n$ for sufficiently large $n$.
Taking $r=s_n$ in Proposition \ref{prop-max-iso}, we obtain
\begin{equation} \label{1005-5}
\int_{\{|x-a_n|>s_n\}}\frac{\rho_n(x)}{|x-a_n|}\mathrm{d}x \leq \int_{\mathcal{K}\setminus D_n}\frac{\rho_n(x)}{\max\{s_n,|x-a_n|\}}\mathrm{d}x \leq C,
\end{equation}
and
\begin{equation} \label{1005-6}
\int_{B_{4s_n}(a_n)\setminus D_n}\rho_n(x)\mathrm{d}x\leq4s_n\int_{\mathcal{K}\setminus D_n} \frac{\rho_n(x)}{\max\{s_n,|x-a_n|\}}\mathrm{d}x\leq Cs_n.
\end{equation}

\subsection{The component integral}\

\vskip 0.1cm

In this subsection, we estimate the Green potential generated by $f_n$ outside $D_n$ and its oscillation near $D_n$. We then prove that the integral of $u_n^{p_n-1}f_n$ over $D_n$ tends to $4\pi$.

\vskip 0.1cm

Recall that
\begin{equation*}
A_{D_n}=2\int_{D_n}hF_n(u_n)\mathrm{d}x,\quad J_{D_n}=\frac{2\lambda_n}{q_n^2}\int_{D_n}h\mathrm{e}^{u_n^{p_n}}\mathrm{d}x.
\end{equation*}
By Lemma \ref{lem-max-inverse}, we have
\begin{equation} \label{1005-10}
\frac{1}{4\pi}-C\delta_n\leq\frac{A_{D_n}}{q_n^2}\leq\frac{1}{4\pi}+C\delta_n+C\lambda_nM_n^{p_n-2}
\end{equation}
for sufficiently large $n$. Moreover, it holds
\begin{equation} \label{1005-11}
J_{D_n}\to\frac{1}{4\pi}.
\end{equation}

Denote
\begin{equation} \label{1005-7}
v_n=G(f_n\mathbf{1}_{D_n}),\quad h_{D_n}=u_n-v_n=G(f_n\mathbf{1}_{\Omega\setminus D_n}).
\end{equation}
For a bounded function $v$ on a set $U$, write
\begin{equation*}
\operatorname{osc}_U v=\sup_U v-\inf_U v.
\end{equation*}

\begin{proposition}
If $\beta_n\to \infty$, then there exists a constant $C>0$, depending only on $\Omega$, $\kappa_0$, $\|h\|_{C^2(\overline\Omega)}$ and $(\min_{\overline\Omega}h)^{-1}$, such that, for all sufficiently large $n$,
\begin{equation} \label{1005-8}
0\leq h_{D_n}\leq CM_n\quad\text{in }B_{s_n}(a_n).
\end{equation}
Moreover,
\begin{equation} \label{1005-9}
m_n^{p_n-1}\operatorname{osc}_{B_{s_n}(a_n)}h_{D_n}\to0.
\end{equation}
\end{proposition}

\begin{proof}
Let $g\in L^\infty(\Omega)$ satisfy $0\leq g\leq B_g$ and $a_g=\int_\Omega g\mathrm{d}x>0$. Choose $r=\sqrt{\frac{a_g}{\pi B_g}}$. Since $a_g\leq B_g|\Omega|\leq\pi B_g d_\Omega^2$, we have $0<r\leq d_\Omega$.
Using \eqref{7} and the identity
\begin{equation*}
\int_0^r s\log\frac{r}{s}\mathrm{d}s=\left.\left(\frac{s^2}{2}\log\frac{r}{s}+\frac{s^2}{4}\right)\right|_{s=0}^{s=r}=\frac{r^2}{4},
\end{equation*}
we obtain for every $x\in\Omega$,
\begin{equation} \label{1005-12}
\begin{aligned}
Gg(x) &\leq\frac{1}{2\pi}\int_{\Omega\cap B_r(x)}\log\frac{d_\Omega}{|x-y|}g(y)\mathrm{d}y+\frac{1}{2\pi}\int_{\Omega\setminus B_r(x)}\log\frac{d_\Omega}{|x-y|}g(y)\mathrm{d}y\\
&\leq\frac{a_g}{2\pi}\log\frac{d_\Omega}{r}+\frac{1}{2\pi}\int_{\Omega\cap B_r(x)}\log\frac{r}{|x-y|}g(y)\mathrm{d}y\\
&\leq\frac{a_g}{2\pi}\log\frac{d_\Omega}{r}+B_g\int_0^r s\log\frac{r}{s}\mathrm{d}s\\
&=\frac{a_g}{2\pi}\log\frac{d_\Omega}{r}+\frac{B_gr^2}{4}=\frac{a_g}{4\pi}\left(1+\log\frac{\pi B_g d_\Omega^2}{a_g}\right).
\end{aligned}
\end{equation}
The right-hand side is increasing in $a_g$ for $0<a_g\leq\pi B_g d_\Omega^2$.

\vskip 0.1cm

Since $D_n\subset B_{4s_n}(a_n)$, we write
\begin{equation*}
f_n\mathbf{1}_{\Omega\setminus D_n}=f_{0,n}+g_{2,n}+g_{3,n},
\end{equation*}
where
\begin{equation*}
g_{2,n}=f_{1,n}\mathbf{1}_{B_{4s_n}(a_n)\setminus D_n},\quad g_{3,n}=f_{1,n}\mathbf{1}_{\Omega\setminus B_{4s_n}(a_n)}.
\end{equation*}
By Lemma \ref{lem-low-source}, we have
\begin{equation} \label{1005-13}
\|Gf_{0,n}\|_{C^1(\overline\Omega)}\leq C\lambda_nM_n^{p_n-1}.
\end{equation}
For $x\in B_{s_n}(a_n)$ and $y\in\Omega\setminus B_{4s_n}(a_n)$,
\begin{equation*}
|x-y|\geq|y-a_n|-s_n\geq\frac{3}{4}|y-a_n|.
\end{equation*}
Estimates \eqref{6} and \eqref{7} imply
\begin{equation*}
G(x,y)\leq\frac{1}{2\pi}\log\frac{4d_\Omega}{3|y-a_n|}\leq\frac{C}{|y-a_n|},\quad |\nabla_xG(x,y)|\leq\frac{C}{|y-a_n|}.
\end{equation*}
Using $f_{1,n}=M_n\rho_n$ and \eqref{1005-5}, we obtain
\begin{equation} \label{1005-14}
\|Gg_{3,n}\|_{C^1(B_{s_n}(a_n))}\leq CM_n\int_{\Omega\setminus B_{s_n}(a_n)}\frac{\rho_n(y)}{|y-a_n|}\mathrm{d}y\leq CM_n.
\end{equation}

On the other hand, using $f_{1,n}=M_n\rho_n$ and \eqref{1005-6}, we obtain
\begin{equation*}
\int_\Omega g_{2,n}\mathrm{d}x=M_n\int_{B_{4s_n}(a_n)\setminus D_n}\rho_n\mathrm{d}x\leq CM_ns_n.
\end{equation*}
Moreover,
\begin{equation*}
0\leq g_{2,n}\leq\|f_n\|_{L^\infty(\Omega)}\leq p_n\lambda_n\|h\|_{L^\infty(\Omega)}m_n^{p_n-1}\mathrm{e}^{m_n^{p_n}}.
\end{equation*}
Since $M_n\leq|\Omega|\|f_n\|_{L^\infty(\Omega)}$ and $s_n\to0$, for all sufficiently large $n$,
\begin{equation*}
CM_ns_n\leq M_n\leq|\Omega|\|f_n\|_{L^\infty(\Omega)}\leq\pi d_\Omega^2\|f_n\|_{L^\infty(\Omega)}.
\end{equation*}
By \eqref{1005-12} and the monotonicity established after it, if $g_{2,n}\neq0$, then
\begin{equation*}
\begin{aligned}
\|Gg_{2,n}\|_{L^\infty(\Omega)}&\leq\frac{\displaystyle\int_\Omega g_{2,n}\mathrm{d}x}{4\pi}\left(1+\log\frac{\pi d_\Omega^2\|f_n\|_{L^\infty(\Omega)}}{\int_\Omega g_{2,n}\mathrm{d}x}\right)
\leq\frac{CM_ns_n}{4\pi}\left(1+\log\frac{\pi d_\Omega^2\|f_n\|_{L^\infty(\Omega)}}{CM_ns_n}\right).
\end{aligned}
\end{equation*}
Using \eqref{uniform-max-bound}, \eqref{bound-1} and $2\leq\ell_n\leq3$, for sufficiently large $n$,
\begin{equation*}
\begin{aligned}
m_n^{p_n}=(m_n^{p_n-1})^{\ell_n}\leq C\lambda_n^{-\frac{\ell_n}{2}},\quad
\frac{\lambda_nm_n^{p_n-1}}{M_n}\leq C\frac{\sqrt{\lambda_n}}{M_n}\leq C, \quad
M_n^{-\ell_n}=\left(\frac{\sqrt{\lambda_n}}{M_n}\right)^{\ell_n}\lambda_n^{-\frac{\ell_n}{2}}\leq\lambda_n^{-\frac{\ell_n}{2}}.
\end{aligned}
\end{equation*}
By the definition of $s_n$ and $f_n$, we obtain
\begin{equation*}
\begin{aligned}
0&\leq\log\frac{\pi d_\Omega^2\|f_n\|_{L^\infty(\Omega)}}{CM_ns_n}\\
&\leq m_n^{p_n}+\log\frac{\pi d_\Omega^2p_n\lambda_n\|h\|_{L^\infty(\Omega)}m_n^{p_n-1}}{CM_n}+\frac{\pi\kappa_0^{\frac{1}{p_n-1}}}{M_n^{\ell_n}}-\frac{1}{2}\log C_0\\
&\leq C\lambda_n^{-\frac{\ell_n}{2}}+C+CM_n^{-\ell_n}\leq C\lambda_n^{-\frac{\ell_n}{2}}.
\end{aligned}
\end{equation*}
Consequently, since $s_n=o(\lambda_n^2)$ and $\ell_n\leq3$,
\begin{equation} \label{1005-15}
\|Gg_{2,n}\|_{L^\infty(\Omega)}\leq CM_ns_n\left(1+\lambda_n^{-\frac{\ell_n}{2}}\right)\leq CM_ns_n\lambda_n^{-\frac{\ell_n}{2}}=o(M_n).
\end{equation}
This estimate also holds when $g_{2,n}=0$. Combining \eqref{1005-13}-\eqref{1005-15} with \eqref{low-error-small} proves \eqref{1005-8}.

\vskip 0.1cm

Finally, from \eqref{1005-13}, \eqref{1005-14} and \eqref{low-error-small}, we obtain
\begin{equation} \label{1005-15-1}
\|\nabla(Gf_{0,n}+Gg_{3,n})\|_{L^\infty(B_{s_n}(a_n))}\leq CM_n(\lambda_nM_n^{p_n-2}+1)\leq CM_n.
\end{equation}
Since $Gg_{2,n}\geq0$, \eqref{1005-15} and \eqref{1005-15-1} imply
\begin{equation*}
\begin{aligned}
\operatorname{osc}_{B_{s_n}(a_n)}h_{D_n}&\leq2s_n\|\nabla(Gf_{0,n}+Gg_{3,n})\|_{L^\infty(B_{s_n}(a_n))}+\|Gg_{2,n}\|_{L^\infty(\Omega)}\\
&\leq CM_ns_n+CM_ns_n\lambda_n^{-\frac{\ell_n}{2}}\leq CM_ns_n\lambda_n^{-\frac{\ell_n}{2}}.
\end{aligned}
\end{equation*}
Estimate \eqref{uniform-max-bound} and Proposition \ref{pro:bound} imply $m_n^{p_n-1}\leq\frac{C}{\sqrt{\lambda_n}}$ and $M_n\leq C$. Therefore,
\begin{equation*}
0\leq m_n^{p_n-1}\operatorname{osc}_{B_{s_n}(a_n)}h_{D_n}\leq CM_ns_n\lambda_n^{-\frac{\ell_n+1}{2}}\leq C\frac{s_n}{\lambda_n^2}\to0.
\end{equation*}
This proves \eqref{1005-9}.
\end{proof}

\begin{remark}
We use $v_n=G(f_n\mathbf{1}_{D_n})$ so that $v_n/q_n$ has source mass one. Using $u_n/M_n$ to estimate $\mathrm{e}^{bu_n/q_n}$ would require $b<4\pi q_n/M_n$, and $q_n/M_n$ need not tend to one. For the same reason, we used $w_n=W_n/Q_n$ in the proof of Proposition \ref{prop-strong-scale}.
\end{remark}

For the component estimate, set
\begin{equation*}
E_{D_n}=\frac{p_n}{2}\int_{D_n}u_n^{p_n-1}f_n\mathrm{d}x.
\end{equation*}
By the Cauchy-Schwarz inequality and the definition of $J_{D_n}$,
\begin{equation*}
\begin{aligned}
q_n^2&=p_n^2\lambda_n^2\left(\int_{D_n}hu_n^{p_n-1}\mathrm{e}^{u_n^{p_n}}\mathrm{d}x\right)^2
\leq p_n^2\lambda_n^2\left(\int_{D_n}h\mathrm{e}^{u_n^{p_n}}\mathrm{d}x\right)\left(\int_{D_n}hu_n^{2(p_n-1)}\mathrm{e}^{u_n^{p_n}}\mathrm{d}x\right)=E_{D_n}q_n^2J_{D_n}.
\end{aligned}
\end{equation*}
Since $q_n>0$, we obtain
\begin{equation} \label{1005-CS}
E_{D_n}\geq\frac{1}{J_{D_n}}.
\end{equation}
Set $Z_n=\frac{p_n}{2}q_nu_n^{p_n-1}$. By the definitions of $E_{D_n}$ and $J_{D_n}$,
\begin{equation*}
E_{D_n}=\frac{1}{q_n}\int_{D_n}Z_nf_n\mathrm{d}x,\quad J_{D_n}=\frac{1}{q_n}\int_{D_n}\frac{f_n}{Z_n}\mathrm{d}x.
\end{equation*}

\begin{proposition}
If $\beta_n\to \infty$, then it holds
\begin{equation} \label{1005-16}
E_{D_n}\to4\pi,
\end{equation}
and
\begin{equation} \label{1005-17}
\frac{p_n}{2}\int_{B_{s_n}(a_n)}u_n^{p_n-1}f_n\mathrm{d}x\to4\pi.
\end{equation}
Moreover,
\begin{equation} \label{1005-18}
\frac{p_n}{2}\int_{D_n}\left(1-\frac{4\pi}{Z_n}\right)^2u_n^{p_n-1}f_n\mathrm{d}x=E_{D_n}-8\pi+16\pi^2J_{D_n}\to0.
\end{equation}
\end{proposition}

\begin{proof}
By \eqref{1005-11} and \eqref{1005-CS}, we have
\begin{equation} \label{1005-19}
\liminf_{n\to\infty}E_{D_n}\geq\liminf_{n\to\infty}\frac{1}{J_{D_n}} \geq 4\pi.
\end{equation}
On the other hand, $u_n\leq m_n$ gives
\begin{equation*}
E_{D_n}=\frac{p_n}{2}\int_{D_n}u_n^{p_n-1}f_n\mathrm{d}x\leq\frac{p_n}{2}m_n^{p_n-1}q_n.
\end{equation*}

For $0<\kappa<2\pi p_n$, set
\begin{equation*}
D_{0,n}=D_n\cap\{Z_n\leq\kappa\},\quad A_{0,n}=2\lambda_n\int_{D_{0,n}}h\mathrm{e}^{u_n^{p_n}}\mathrm{d}x.
\end{equation*}
Let $w_{D_n}=\frac{v_n}{q_n}$. By \eqref{1005-4} and \eqref{1005-8}, we get
\begin{equation*}
u_n=q_nw_{D_n}+h_{D_n}\leq q_nw_{D_n}+CM_n\leq q_n(w_{D_n}+C)\quad\text{in }D_n.
\end{equation*}
In $D_{0,n}$, we have
\begin{equation*}
u_n^{p_n}=\frac{2Z_n}{p_n}\frac{u_n}{q_n}\leq\frac{2\kappa}{p_n}\big(w_{D_n}+C\big).
\end{equation*}
Since the source of $w_{D_n}$ has integral one, estimate \eqref{8} yields
\begin{equation} \label{A0n-bound}
\begin{aligned}
A_{0,n}&\leq2\lambda_n\|h\|_{L^\infty(\Omega)}\mathrm{e}^{\frac{2\kappa C}{p_n}}\int_\Omega\mathrm{e}^{\frac{2\kappa}{p_n}w_{D_n}}\mathrm{d}x
\leq\frac{8\pi^2d_\Omega^2\|h\|_{L^\infty(\Omega)}\mathrm{e}^{4\pi C}\lambda_n}{4\pi-\frac{2\kappa}{p_n}}
\leq\frac{C\lambda_n}{2\pi p_n-\kappa},
\end{aligned}
\end{equation}
with $C$ independent of $n$ and $\kappa\in(0,2\pi p_n)$.

\vskip 0.1cm

Since $u_n\leq m_n$, we have
\begin{equation*}
\begin{aligned}
&Z_n^2-\left(\kappa+\frac{p_n}{2}m_n^{p_n-1}q_n\right)Z_n+\frac{p_n\kappa}{2}m_n^{p_n-1}q_n
=(Z_n-\kappa)\left(Z_n-\frac{p_n}{2}m_n^{p_n-1}q_n\right)\leq0
\end{aligned}
\end{equation*}
in $D_n\cap\{Z_n>\kappa\}$. Using $Z_n\leq\kappa$ in $D_{0,n}$, we obtain
\begin{equation*}
Z_n\leq\kappa+\frac{p_n}{2}m_n^{p_n-1}q_n-\frac{p_n\kappa m_n^{p_n-1}q_n}{2Z_n}\mathbf{1}_{\{Z_n>\kappa\}}\quad\text{in }D_n.
\end{equation*}
Integrating against $\frac{f_n}{q_n}\mathrm{d}x$ gives
\begin{equation*}
E_{D_n}\leq\kappa+\frac{p_n}{2}m_n^{p_n-1}q_n-\frac{p_n\kappa}{2}m_n^{p_n-1}\int_{D_n\cap\{Z_n>\kappa\}}\frac{f_n}{Z_n}\mathrm{d}x.
\end{equation*}
By the definitions of $J_{D_n}$ and $A_{0,n}$,
\begin{equation*}
\begin{aligned}
\frac{1}{q_n}\int_{D_n\cap\{Z_n>\kappa\}}\frac{f_n}{Z_n}\mathrm{d}x
&=\frac{2\lambda_n}{q_n^2}\int_{D_n\cap\{Z_n>\kappa\}}h\mathrm{e}^{u_n^{p_n}}\mathrm{d}x
=J_{D_n}-\frac{A_{0,n}}{q_n^2}.
\end{aligned}
\end{equation*}
Therefore,
\begin{equation} \label{1005-21}
E_{D_n}\leq\kappa+\frac{p_n}{2}m_n^{p_n-1}q_n-\kappa\frac{p_n}{2}m_n^{p_n-1}q_n\left(J_{D_n}-\frac{A_{0,n}}{q_n^2}\right).
\end{equation}

By \eqref{bound-1} and \eqref{1005-4}, we have $\frac{\sqrt{\lambda_n}}{q_n}\to0$. For sufficiently large $n$, set
\begin{equation*}
\zeta_n=\frac{\sqrt{\lambda_n}}{q_n}\in(0,1),\quad \kappa_n=2\pi p_n-\zeta_n\in(0,2\pi p_n).
\end{equation*}
The definitions of $\kappa_n$ and $\zeta_n$ imply
\begin{equation*}
1-\frac{\kappa_n}{4\pi}=\frac{2-p_n}{2}+\frac{\zeta_n}{4\pi},
\end{equation*}
and
\begin{equation*}
m_n^{p_n-1}q_n\zeta_n=m_n^{p_n-1}\sqrt{\lambda_n},\quad \frac{m_n^{p_n-1}\lambda_n}{q_n\zeta_n}=m_n^{p_n-1}\sqrt{\lambda_n}.
\end{equation*}
By \eqref{1005-10},
\begin{equation*}
J_{D_n}\geq\frac{A_{D_n}}{q_n^2}\geq\frac{1}{4\pi}-C\delta_n.
\end{equation*}
Substituting this lower bound and \eqref{A0n-bound} into \eqref{1005-21}, we obtain
\begin{equation*}
\begin{aligned}
E_{D_n}&\leq\kappa_n+\frac{p_n}{2}m_n^{p_n-1}q_n\left(\frac{2-p_n}{2}+\frac{\zeta_n}{4\pi}\right)+Cm_n^{p_n-1}q_n\delta_n+C\frac{m_n^{p_n-1}\lambda_n}{q_n\zeta_n}\\
&\leq4\pi+C(2-p_n)m_n^{p_n-1}q_n+Cm_n^{p_n-1}q_n\delta_n+Cm_n^{p_n-1}\left(q_n\zeta_n+\frac{\lambda_n}{q_n\zeta_n}\right)\\
&\leq4\pi+C(2-p_n)m_n^{p_n-1}q_n+Cm_n^{p_n-1}q_n\delta_n+Cm_n^{p_n-1}\sqrt{\lambda_n}.
\end{aligned}
\end{equation*}
Since $\Theta_n(m_n)\leq C$ by \eqref{uniform-max-bound} and $m_n>1$,
\begin{equation*}
\frac{2-p_n}{2}m_n^{p_n}\leq L_n-2(p_n-1)\log m_n+C\leq L_n+C.
\end{equation*}
Also, $m_n>t_{0,n}$ and $q_n\leq M_n$ imply
\begin{equation*}
\frac{q_n}{m_n}\leq\frac{M_n}{t_{0,n}}=\kappa_0^{-\frac{1}{p_n-1}}M_n^{\ell_n}\leq CM_n^{\ell_n}.
\end{equation*}
Consequently, by \eqref{low-error-small} and \eqref{0831-9},
\begin{equation} \label{exponent-error}
0\leq(2-p_n)q_nm_n^{p_n-1}\leq\frac{2q_n}{m_n}(L_n+C)\leq CM_n^{\ell_n}(L_n+1)\to0.
\end{equation}
Since $q_n\leq M_n$ and $\delta_n=s_n^2\leq s_n$ for sufficiently large $n$, estimates \eqref{1005-2}, \eqref{bound-2} and \eqref{exponent-error} yield
\begin{equation*}
\begin{aligned}
E_{D_n}&\leq4\pi+C(2-p_n)q_nm_n^{p_n-1}+Cm_n^{p_n-1}M_ns_n+Cm_n^{p_n-1}\sqrt{\lambda_n}=4\pi+o(1).
\end{aligned}
\end{equation*}
Therefore,
\begin{equation*}
\limsup_{n\to\infty}E_{D_n}\leq4\pi.
\end{equation*}
Together with \eqref{1005-19}, this proves \eqref{1005-16}.

\vskip 0.1cm

Using $D_n\subset B_{s_n}(a_n)$, \eqref{10} and \eqref{1005-6}, we obtain
\begin{equation*}
\begin{aligned}
0\leq&\frac{p_n}{2}\int_{B_{s_n}(a_n)}u_n^{p_n-1}f_n\mathrm{d}x-E_{D_n}=\frac{p_n}{2}\int_{B_{s_n}(a_n)\setminus D_n}u_n^{p_n-1}f_n\mathrm{d}x\\
\leq&\frac{p_n}{2}m_n^{p_n-1}M_n\int_{B_{4s_n}(a_n)\setminus D_n}\rho_n\mathrm{d}x+\frac{p_n}{2}\int_\Omega u_n^{p_n-1}f_{0,n}\mathrm{d}x\\
\leq& Cm_n^{p_n-1}M_ns_n+C\lambda_nM_n^{2p_n-2}\to0,
\end{aligned}
\end{equation*}
where the last limit follows from \eqref{low-error-small} and \eqref{bound-2}. Together with \eqref{1005-16}, this proves \eqref{1005-17}. Finally, by the definition of $Z_n$,
\begin{equation*}
\begin{aligned}
\frac{p_n}{2}\int_{D_n}\left(1-\frac{4\pi}{Z_n}\right)^2u_n^{p_n-1}f_n\mathrm{d}x
&=\frac{p_n}{2}\int_{D_n}(Z_n-4\pi)^2\frac{u_n^{p_n-1}f_n}{Z_n^2}\mathrm{d}x\\
&=\frac{2\lambda_n}{q_n^2}\int_{D_n}(Z_n-4\pi)^2h\mathrm{e}^{u_n^{p_n}}\mathrm{d}x\\
&=E_{D_n}-\frac{8\pi}{q_n}\int_{D_n}f_n\mathrm{d}x+16\pi^2J_{D_n}\\
&=E_{D_n}-8\pi+16\pi^2J_{D_n}\to0.
\end{aligned}
\end{equation*}
This proves \eqref{1005-18}.
\end{proof}

\subsection{The component maximum}\

\vskip 0.1cm

In this subsection, we combine the local blow-up profile with \eqref{1005-18} to relate the maximum of $u_n$ on $D_n$ to its source mass. Since $u_n=t_{0,n}$ on $\partial D_n$ and $u_n>t_{0,n}$ in $D_n$, choose $x_{1,n}\in D_n$ such that
\begin{equation*}
\gamma_n=u_n(x_{1,n})=\max_{\overline{D_n}}u_n,\quad \nabla u_n(x_{1,n})=0.
\end{equation*}
Recall that $Z_n=\frac{p_n}{2}q_nu_n^{p_n-1}$.

\begin{proposition} \label{prop-component-max}
If $\beta_n\to \infty$, then it holds
\begin{equation} \label{1005-22}
Z_n(x_{1,n})=\frac{p_n}{2}\gamma_n^{p_n-1}q_n\to4\pi.
\end{equation}
Moreover,
\begin{equation} \label{1005-24}
u_n\leq\gamma_n+\frac{C}{\gamma_n^{p_n-1}}\quad\text{in }B_{s_n}(a_n).
\end{equation}
\end{proposition}

\begin{proof}
The function $v_n$ defined in \eqref{1005-7} is harmonic in $\Omega\setminus\overline{D_n}$ and vanishes on $\partial\Omega$. Since $h_{D_n}\geq0$,
\begin{equation*}
0\leq v_n=u_n-h_{D_n}\leq u_n=t_{0,n}<\gamma_n\quad\text{on }\partial D_n.
\end{equation*}
The maximum principle therefore yields
\begin{equation*}
\max_{\overline\Omega}v_n=\max_{\overline{D_n}}v_n.
\end{equation*}
Choose $y_n\in\overline{D_n}$ such that $v_n(y_n)=\max_{\overline\Omega}v_n$. Since $a_n\in D_n$ and $\operatorname{diam}D_n\leq s_n^2<s_n$ for sufficiently large $n$, we have $\overline{D_n}\subset B_{s_n}(a_n)$. For $x\in B_{s_n}(a_n)$,
\begin{equation*}
\begin{aligned}
u_n(x)&=v_n(x)+h_{D_n}(x)\leq v_n(y_n)+h_{D_n}(x)\\
&=u_n(y_n)+h_{D_n}(x)-h_{D_n}(y_n)
\leq\gamma_n+\operatorname{osc}_{B_{s_n}(a_n)}h_{D_n}.
\end{aligned}
\end{equation*}
Since $u_n(x_{1,n})=\gamma_n$ and $x_{1,n}\in B_{s_n}(a_n)$, using $\gamma_n\leq m_n$, $p_n\leq2$ and \eqref{1005-9}, we obtain
\begin{equation} \label{1005-25}
0\leq\sup_{B_{s_n}(a_n)}\frac{p_n}{2}\gamma_n^{p_n-1}(u_n-\gamma_n)\leq m_n^{p_n-1}\operatorname{osc}_{B_{s_n}(a_n)}h_{D_n}\to0.
\end{equation}
This proves \eqref{1005-24} for sufficiently large $n$.

\vskip 0.1cm

Since $u_n\leq\gamma_n$ in $D_n$ and $q_n\leq M_n$, we have
\begin{equation*}
E_{D_n}\leq\frac{p_n}{2}\gamma_n^{p_n-1}q_n\leq\frac{p_n}{2}\gamma_n^{p_n-1}M_n.
\end{equation*}
By \eqref{1005-16} and $M_n\to0$, we obtain $\gamma_n\to\infty$ and
\begin{equation} \label{1005-25-1}
M_n^{\ell_n}\gamma_n^{p_n}\geq\left(\frac{2E_{D_n}}{p_n}\right)^{\ell_n}\to16\pi^2.
\end{equation}
Combining \eqref{1005-25-1} with \eqref{0831-9}, we obtain
\begin{equation} \label{1005-25-2}
0\leq\frac{L_n}{\gamma_n^{p_n}}\leq\frac{M_n^{\ell_n}L_n}{\left(\frac{2E_{D_n}}{p_n}\right)^{\ell_n}}\to0.
\end{equation}

Set
\begin{equation} \label{1005-26}
\begin{aligned}
\mu_n&=\frac{2\sqrt{2}}{p_n\sqrt{\lambda_n h(x_{1,n})}\gamma_n^{p_n-1}\mathrm{e}^{\frac{\gamma_n^{p_n}}{2}}},\quad
\eta_n(y)=\frac{p_n}{2}\gamma_n^{p_n-1}\big(u_n(x_{1,n}+\mu_ny)-\gamma_n\big).
\end{aligned}
\end{equation}
Since $s_n=\sqrt{\delta_n}$ and $\gamma_n>1$ for sufficiently large $n$, \eqref{0831-9} and \eqref{1005-25-1} yield
\begin{equation*}
\begin{aligned}
M_n^{\ell_n}\log\frac{s_n}{\mu_n}
&=\frac{M_n^{\ell_n}\gamma_n^{p_n}}{2}-\frac{M_n^{\ell_n}L_n}{2}+(p_n-1)M_n^{\ell_n}\log\gamma_n\\
&\quad-\pi\kappa_0^{\frac{1}{p_n-1}}+M_n^{\ell_n}\log\frac{p_n\sqrt{C_0h(x_{1,n})}}{2\sqrt{2}}\\
&\geq\frac{M_n^{\ell_n}\gamma_n^{p_n}}{2}-\frac{M_n^{\ell_n}L_n}{2}-\pi\kappa_0^{\frac{1}{p_n-1}}-CM_n^{\ell_n}\\
&\geq\frac{1}{2}\left(\frac{2E_{D_n}}{p_n}\right)^{\ell_n}-\pi\kappa_0^{\frac{1}{p_n-1}}+o(1)=8\pi^2-\pi\kappa_0+o(1).
\end{aligned}
\end{equation*}
Since $0<\kappa_0<2\pi$ and $M_n\to0$, we obtain
\begin{equation} \label{1005-27}
\frac{s_n}{\mu_n}\to\infty.
\end{equation}
For every fixed $R>0$, the bound $\operatorname{diam}D_n\leq s_n^2$ implies
\begin{equation*}
\frac{|x_{1,n}-a_n|+R\mu_n}{s_n}\leq s_n+\frac{R\mu_n}{s_n}\to0.
\end{equation*}
Thus $B_{R\mu_n}(x_{1,n})\subset B_{s_n}(a_n)$ for sufficiently large $n$. Estimate \eqref{1005-25} and $\eta_n(0)=0$ yield
\begin{equation*}
0\leq\sup_{B_R}\eta_n\leq\sup_{B_{s_n}(a_n)}\frac{p_n}{2}\gamma_n^{p_n-1}(u_n-\gamma_n)\to0.
\end{equation*}
Since $x_{1,n}\in\mathcal{K}\subset\subset\Omega$, $\nabla u_n(x_{1,n})=0$ and $\gamma_n\to\infty$, we obtain from \eqref{1005-25-2} and Lemma \ref{lem-critical-core}-(i) that
\begin{equation} \label{1005-28}
\eta_n\to-\log(1+|y|^2)\quad\text{in }C^2_{\mathrm{loc}}(\mathbb{R}^2),
\end{equation}
and
\begin{equation} \label{1005-28-1}
\frac{p_n}{2}\int_{B_{R\mu_n}(x_{1,n})}u_n^{p_n-1}f_n\mathrm{d}x\to\frac{4\pi R^2}{1+R^2}.
\end{equation}
Moreover, by \eqref{1005-28},
\begin{equation*}
\sup_{B_{\mu_n}(x_{1,n})}\left|\frac{u_n}{\gamma_n}-1\right|=\frac{2}{p_n\gamma_n^{p_n}}\sup_{B_1}|\eta_n|\leq\frac{C}{\gamma_n^{p_n}}.
\end{equation*}
Since $\frac{Z_n}{Z_n(x_{1,n})}=\left(\frac{u_n}{\gamma_n}\right)^{p_n-1}$ and $p_n\in[\frac{3}{2},2]$, \eqref{1005-28} and the mean value theorem yield
\begin{equation} \label{1005-28-2}
\sup_{B_{\mu_n}(x_{1,n})}\left|\frac{Z_n}{Z_n(x_{1,n})}-1\right|+\sup_{B_{\mu_n}(x_{1,n})}\left|\frac{Z_n(x_{1,n})}{Z_n}-1\right|\leq\frac{C}{\gamma_n^{p_n}}.
\end{equation}
Using \eqref{1005-28-2}, $E_{D_n}\leq\frac{p_n}{2}\gamma_n^{p_n-1}M_n$ and \eqref{1005-16}, we obtain
\begin{equation*}
\begin{aligned}
\inf_{B_{\mu_n}(x_{1,n})}M_nu_n^{p_n-1}
&=M_n\gamma_n^{p_n-1}\inf_{B_{\mu_n}(x_{1,n})}\frac{Z_n}{Z_n(x_{1,n})}\\
&\geq M_n\gamma_n^{p_n-1}\left(1-\frac{C}{\gamma_n^{p_n}}\right)\\
&\geq\frac{2E_{D_n}}{p_n}\left(1-\frac{C}{\gamma_n^{p_n}}\right)=4\pi+o(1)>\kappa_0.
\end{aligned}
\end{equation*}
Since $M_nt_{0,n}^{p_n-1}=\kappa_0$, we have
\begin{equation*}
u_n>t_{0,n} \quad \text{in } B_{\mu_n}(x_{1,n}).
\end{equation*}
This ball is connected and contains $x_{1,n}\in D_n$, so $B_{\mu_n}(x_{1,n})\subset D_n$.

\vskip 0.1cm

Since $Z_n\leq Z_n(x_{1,n})$ in $D_n$, we have
\begin{equation*}
Z_n(x_{1,n})\geq\frac{1}{q_n}\int_{D_n}Z_nf_n\mathrm{d}x=E_{D_n}\geq2\pi
\end{equation*}
for sufficiently large $n$ by \eqref{1005-16}. Hence $|1-\frac{4\pi}{Z_n(x_{1,n})}|\leq1$. Estimate \eqref{1005-28-2} also implies $\frac{4\pi}{Z_n}\leq4$ in $B_{\mu_n}(x_{1,n})$. It follows that
\begin{equation*}
\begin{aligned}
&\sup_{B_{\mu_n}(x_{1,n})}\left|\left(1-\frac{4\pi}{Z_n}\right)^2-\left(1-\frac{4\pi}{Z_n(x_{1,n})}\right)^2\right|\\
&\leq\frac{4\pi}{Z_n(x_{1,n})}\sup_{B_{\mu_n}(x_{1,n})}\left|\frac{Z_n(x_{1,n})}{Z_n}-1\right| \cdot \sup_{B_{\mu_n}(x_{1,n})}\left|2-\frac{4\pi}{Z_n}-\frac{4\pi}{Z_n(x_{1,n})}\right|\\
&\leq\frac{C}{Z_n(x_{1,n})}\sup_{B_{\mu_n}(x_{1,n})}\left|\frac{Z_n(x_{1,n})}{Z_n}-1\right|\leq\frac{C}{\gamma_n^{p_n}}.
\end{aligned}
\end{equation*}
Together with \eqref{1005-18} and \eqref{1005-28-1} with $R=1$, this gives
\begin{equation*}
\begin{aligned}
2\pi\left(1-\frac{4\pi}{Z_n(x_{1,n})}\right)^2
&=\left(1-\frac{4\pi}{Z_n(x_{1,n})}\right)^2\frac{p_n}{2}\int_{B_{\mu_n}(x_{1,n})}u_n^{p_n-1}f_n\mathrm{d}x+o(1)\\
&=\frac{p_n}{2}\int_{B_{\mu_n}(x_{1,n})}\left(1-\frac{4\pi}{Z_n}\right)^2u_n^{p_n-1}f_n\mathrm{d}x+o(1)\\
&\leq\frac{p_n}{2}\int_{D_n}\left(1-\frac{4\pi}{Z_n}\right)^2u_n^{p_n-1}f_n\mathrm{d}x+o(1)\to0.
\end{aligned}
\end{equation*}
Hence $Z_n(x_{1,n})\to4\pi$, which proves \eqref{1005-22}.
\end{proof}

\section{Radial comparison and the parameter bound} \label{sec-final}

In this section, we complete the proof of Theorem \ref{thm-simply-connected}. The integral limit \eqref{1005-17} and the point-selection argument from Section \ref{sec-cores} exclude a second blow-up point near $x_{1,n}$. The resulting pointwise bound and the Green representation provide the gradient estimate needed for radial comparison. We then use the circle mean identity and compare the radial estimate with the Green representation to obtain $\Theta_n(\gamma_n)\geq c>0$. Since $\gamma_n\leq m_n$ and $\Theta_n$ is nondecreasing, this contradicts \eqref{1005-2}.

\vskip 0.1cm

We assume \eqref{1005-1} and use the notations in Section \ref{sec-background}. For a function $w$ defined near $x_{1,n}$, write
\begin{equation*}
\overline w(r)=\frac{1}{2\pi}\int_0^{2\pi}w\big(x_{1,n}+r(\cos\theta,\sin\theta)\big)\mathrm{d}\theta.
\end{equation*}

\begin{lemma}
If \eqref{1005-1} holds, then there exists a constant $C>0$, depending only on $\Omega$, $\kappa_0$, $\|h\|_{C^2(\overline\Omega)}$ and $(\min_{\overline\Omega}h)^{-1}$, such that, for all sufficiently large $n$,
\begin{equation} \label{heavy-density}
|x-x_{1,n}|^2\frac{p_n^2\lambda_nh(x)}{2}u_n(x)^{2(p_n-1)}\mathrm{e}^{u_n(x)^{p_n}}\leq C \quad\text{in }B_{\frac{s_n}{8}}(x_{1,n}),
\end{equation}
and
\begin{equation} \label{heavy-gradient}
|x-x_{1,n}||\nabla u_n(x)|\leq\frac{C}{\gamma_n^{p_n-1}} \quad\text{in }B_{\frac{s_n}{64}}(x_{1,n}).
\end{equation}
\end{lemma}

\begin{proof}
If \eqref{heavy-density} fails, then after passing to a subsequence,
\begin{equation*}
\sup_{B_{\frac{s_n}{8}}(x_{1,n})}|y-x_{1,n}|^2 \frac{p_n^2\lambda_nh(y)}{2}u_n(y)^{2(p_n-1)}\mathrm{e}^{u_n(y)^{p_n}}\to\infty.
\end{equation*}
In $\overline{B_{\frac{s_n}{4}}(x_{1,n})}$, set
\begin{equation*}
d_{1,n}(y)=\min\left\{|y-x_{1,n}|,\frac{s_n}{4}-|y-x_{1,n}|\right\}.
\end{equation*}
This function is $1$-Lipschitz. Choose $x_n$ maximizing $d_{1,n}^2\mathcal{F}_n$ in this closed ball. Set
\begin{equation*}
\begin{aligned}
b_n&=u_n(x_n),\quad \Lambda_n=d_{1,n}(x_n)^2\frac{p_n^2\lambda_nh(x_n)}{2}b_n^{2(p_n-1)}\mathrm{e}^{b_n^{p_n}},\quad
\widehat\mu_n =\frac{2\sqrt{2}}{p_n\sqrt{\lambda_n h(x_n)}b_n^{p_n-1}\mathrm{e}^{\frac{b_n^{p_n}}{2}}}.
\end{aligned}
\end{equation*}
Since $d_{1,n}(y)=|y-x_{1,n}|$ for $|y-x_{1,n}|\leq \frac{s_n}{8}$, the choice of $x_n$ gives
\begin{equation*}
\Lambda_n\to\infty,\quad d_{1,n}(x_n)>0,\quad \frac{\widehat\mu_n}{d_{1,n}(x_n)}=\frac{2}{\sqrt{\Lambda_n}}\to0.
\end{equation*}
In particular, $x_n\neq x_{1,n}$ and $x_n\in B_{\frac{s_n}{4}}(x_{1,n})$ for sufficiently large $n$.

\vskip 0.1cm

Since $|x_{1,n}-a_n|\leq\delta_n=s_n^2$,
\begin{equation*}
\frac{|x_{1,n}-a_n|+\frac{s_n}{4}}{s_n}\leq s_n+\frac{1}{4}<1
\end{equation*}
for sufficiently large $n$. Hence $\overline{B_{\frac{s_n}{4}}(x_{1,n})}\subset B_{s_n}(a_n)$. Since $b_n\leq m_n$ and $p_n\leq2$, \eqref{uniform-max-bound} gives
\begin{equation*}
L_n-2(p_n-1)\log b_n=-\log(\lambda_nb_n^{2(p_n-1)})\geq-\log\Theta_n(m_n)\geq-C.
\end{equation*}
Using $d_{1,n}(x_n)\leq\frac{s_n}{4}$ and $\Lambda_n\geq1$ for sufficiently large $n$, we obtain
\begin{equation} \label{heavy-height}
\begin{aligned}
b_n^{p_n} &=L_n-2(p_n-1)\log b_n+\log\frac{2\Lambda_n}{p_n^2h(x_n)}-2\log d_{1,n}(x_n)\\
&\geq\log\Lambda_n+2\log\frac{1}{s_n}-C
\geq\frac{2\pi\kappa_0^{\frac{1}{p_n-1}}}{M_n^{\ell_n}}-C\geq\frac{c}{M_n^{\ell_n}}\to\infty.
\end{aligned}
\end{equation}
Thus $b_n\to\infty$. Moreover, \eqref{0831-9} implies
\begin{equation} \label{heavy-log}
0\leq\frac{L_n}{b_n^{p_n}}\leq\frac{1}{c}M_n^{\ell_n}L_n\to0.
\end{equation}

Fix $A>0$. For sufficiently large $n$,
\begin{equation*}
A\widehat\mu_n<d_{1,n}(x_n)\leq\frac{s_n}{4}-|x_n-x_{1,n}|,
\end{equation*}
so $B_{A\widehat\mu_n}(x_n)\subset B_{\frac{s_n}{4}}(x_{1,n})\subset\subset\Omega$.
Since $x_{1,n}\in\mathcal{K}\subset\subset\Omega$ and $s_n\to0$,
\begin{equation} \label{heavy-interior}
\operatorname{dist}(x_n,\partial\Omega)\geq\operatorname{dist}(\mathcal{K},\partial\Omega)-\frac{s_n}{4}\geq\frac{1}{2}\operatorname{dist}(\mathcal{K},\partial\Omega)>0
\end{equation}
for sufficiently large $n$. By \eqref{selection-upper},
\begin{equation} \label{heavy-upper}
0\leq\sup_{|z|\leq A}\frac{p_n}{2}b_n^{p_n-1}\left(u_n(x_n+\widehat\mu_nz)-b_n\right)\leq-\log\left(1-\frac{2A}{\sqrt{\Lambda_n}}\right)+CA\widehat\mu_n\to0.
\end{equation}
By \eqref{heavy-height}-\eqref{heavy-upper}, Lemma \ref{lem-critical-core}-(ii) yields
\begin{equation*}
\frac{p_n}{2}\int_{B_{A\widehat\mu_n}(x_n)}u_n^{p_n-1}f_n\mathrm{d}x\to\frac{4\pi A^2}{1+A^2}.
\end{equation*}

For $y\in B_{\frac{s_n}{4}}(x_{1,n})$, \eqref{1005-24} and \eqref{1005-26} yield
\begin{equation*}
\begin{aligned}
\mu_n^2\frac{p_n^2\lambda_nh(y)}{2}u_n(y)^{2(p_n-1)}\mathrm{e}^{u_n(y)^{p_n}}
&=4\frac{h(y)}{h(x_{1,n})}\left(\frac{u_n(y)}{\gamma_n}\right)^{2(p_n-1)}\mathrm{e}^{u_n(y)^{p_n}
-\gamma_n^{p_n}}\\
&\leq C\left(1+\frac{C}{\gamma_n^{p_n}}\right)^{2(p_n-1)}\mathrm{e}^{C}\leq C.
\end{aligned}
\end{equation*}
Applying this bound at $x_n$ and using $d_{1,n}(x_n)\leq|x_n-x_{1,n}|$, we obtain
\begin{equation*}
\Lambda_n=d_{1,n}(x_n)^2\frac{p_n^2\lambda_nh(x_n)}{2}b_n^{2(p_n-1)}\mathrm{e}^{b_n^{p_n}}\leq\frac{Cd_{1,n}(x_n)^2}{\mu_n^2}\leq\frac{C|x_n-x_{1,n}|^2}{\mu_n^2}.
\end{equation*}
Consequently,
\begin{equation*}
\frac{\mu_n}{|x_n-x_{1,n}|}\leq\frac{C}{\sqrt{\Lambda_n}}\to0,\quad
\frac{\widehat\mu_n}{|x_n-x_{1,n}|}\leq\frac{\widehat\mu_n}{d_{1,n}(x_n)}=\frac{2}{\sqrt{\Lambda_n}}\to0.
\end{equation*}
Fix $A>1$. Using \eqref{1005-27}, we obtain
\begin{equation*}
\frac{|x_{1,n}-a_n|+A\mu_n}{s_n}\leq s_n+\frac{A\mu_n}{s_n}\to0,\quad
\frac{A(\mu_n+\widehat\mu_n)}{|x_n-x_{1,n}|}\to0.
\end{equation*}
Together with $B_{A\widehat\mu_n}(x_n)\subset B_{\frac{s_n}{4}}(x_{1,n})$, this implies
\begin{equation*}
B_{A\mu_n}(x_{1,n})\cup B_{A\widehat\mu_n}(x_n)\subset B_{s_n}(a_n),\quad
B_{A\mu_n}(x_{1,n})\cap B_{A\widehat\mu_n}(x_n)=\emptyset
\end{equation*}
for sufficiently large $n$. By \eqref{1005-17}, \eqref{1005-28-1} and Lemma \ref{lem-critical-core}-(ii),
\begin{equation*}
\begin{aligned}
4\pi&=\lim_{n\to\infty}\frac{p_n}{2}\int_{B_{s_n}(a_n)}u_n^{p_n-1}f_n\mathrm{d}x\\
&\geq\lim_{n\to\infty}\left(\frac{p_n}{2}\int_{B_{A\mu_n}(x_{1,n})}u_n^{p_n-1}f_n\mathrm{d}x+\frac{p_n}{2}\int_{B_{A\widehat\mu_n}(x_n)}u_n^{p_n-1}f_n\mathrm{d}x\right)\\
&=\frac{8\pi A^2}{1+A^2}>4\pi,
\end{aligned}
\end{equation*}
which is a contradiction. This proves \eqref{heavy-density}.

\vskip 0.1cm

To prove \eqref{heavy-gradient}, set $v_{0,n}=Gf_{0,n}$. By \eqref{1005-4} and \eqref{1005-22}, for sufficiently large $n$,
\begin{equation} \label{M_nt_{0,n}}
M_n\leq\frac{q_n}{c_*}\leq\frac{C}{\gamma_n^{p_n-1}},\quad t_{0,n}^{p_n-1}=\frac{\kappa_0}{M_n}\geq c\gamma_n^{p_n-1}.
\end{equation}
For $y\in\{u_n>t_{0,n}\}\cap B_{\frac{s_n}{8}}(x_{1,n})$ with $y\neq x_{1,n}$, \eqref{heavy-density} and \eqref{M_nt_{0,n}} give
\begin{equation*}
\begin{aligned}
f_{1,n}(y) =\frac{p_n\lambda_nh(y)u_n(y)^{2(p_n-1)}\mathrm{e}^{u_n(y)^{p_n}}}{u_n(y)^{p_n-1}}
\leq\frac{C}{u_n(y)^{p_n-1}|y-x_{1,n}|^2}\leq\frac{C}{\gamma_n^{p_n-1}|y-x_{1,n}|^2}.
\end{aligned}
\end{equation*}
Since $f_{1,n}=0$ in $\{u_n\leq t_{0,n}\}$, this bound holds whenever $0<|y-x_{1,n}|<\frac{s_n}{8}$.
Fix $x\in B_{\frac{s_n}{64}}(x_{1,n})\setminus\{x_{1,n}\}$ and set $d=|x-x_{1,n}|$. For $y\in B_{\frac{d}{2}}(x)$,
\begin{equation*}
\frac{d}{2}\leq|y-x_{1,n}|\leq\frac{3d}{2}<\frac{s_n}{8},\quad f_{1,n}(y)\leq\frac{C}{\gamma_n^{p_n-1}d^2}.
\end{equation*}
Estimate \eqref{6} yields
\begin{equation*}
\begin{aligned}
\int_{B_{\frac{d}{2}}(x)}|\nabla_xG(x,y)|f_{1,n}(y)\mathrm{d}y
&\leq\frac{C}{\gamma_n^{p_n-1}d^2}\int_{B_{\frac{d}{2}}(x)}\frac{\mathrm{d}y}{|x-y|}\leq\frac{C}{\gamma_n^{p_n-1}d}.
\end{aligned}
\end{equation*}
For $y\in\Omega\setminus B_{\frac{d}{2}}(x)$, we have $|x-y|\geq\frac{d}{2}$. Using \eqref{6} and \eqref{M_nt_{0,n}}, we obtain
\begin{equation*}
\begin{aligned}
\int_{\Omega\setminus B_{\frac{d}{2}}(x)}|\nabla_xG(x,y)|f_{1,n}(y)\mathrm{d}y
&\leq\frac{C}{d}\int_\Omega f_{1,n}(y)\mathrm{d}y\leq\frac{CM_n}{d}\leq\frac{C}{\gamma_n^{p_n-1}d}.
\end{aligned}
\end{equation*}
Differentiating $u_n=Gf_{1,n}+v_{0,n}$ and applying \eqref{12}, we obtain
\begin{equation*}
\begin{aligned}
d|\nabla u_n(x)|&\leq\frac{C}{\gamma_n^{p_n-1}}+d\|\nabla v_{0,n}\|_{L^\infty(\Omega)}\\
&\leq\frac{C}{\gamma_n^{p_n-1}}+C\lambda_ndM_n^{p_n-1}
\leq\frac{C}{\gamma_n^{p_n-1}}(1+\lambda_nM_n^{p_n-2}d)\leq\frac{C}{\gamma_n^{p_n-1}},
\end{aligned}
\end{equation*}
where the last inequality follows from \eqref{low-error-small} and $d\leq\frac{s_n}{64}$. This proves \eqref{heavy-gradient}.
\end{proof}

Recall that
\begin{equation*}
\Theta_n(s)=\lambda_ns^{2(p_n-1)}\mathrm{e}^{\frac{2-p_n}{2}s^{p_n}}.
\end{equation*}
Estimate \eqref{heavy-gradient} allows us to apply radial comparison near $x_{1,n}$.

\begin{proposition} \label{prop-final-height}
If \eqref{1005-1} holds, there exists a constant $c>0$, depending only on $\Omega$, $\kappa_0$, $\|h\|_{C^2(\overline\Omega)}$ and $(\min_{\overline\Omega}h)^{-1}$, such that, for all sufficiently large $n$,
\begin{equation} \label{heavy-height-lower}
\Theta_n(\gamma_n)\geq c.
\end{equation}
\end{proposition}

\begin{proof}
Set $\varrho_n=\frac{s_n}{128}$. By the choice of $x_{1,n}$, \eqref{1005-25-1} and \eqref{1005-27},
\begin{equation*}
\gamma_n\to\infty,\quad \nabla u_n(x_{1,n})=0,\quad B_{\varrho_n}(x_{1,n})\Subset\Omega,\quad \varrho_n\to0,\quad \frac{\varrho_n}{\mu_n}\to\infty.
\end{equation*}
Moreover, \eqref{1005-28} gives
\begin{equation*}
\frac{p_n}{2}\gamma_n^{p_n-1}\left(u_n(x_{1,n}+\mu_ny)-\gamma_n\right)\to-\log(1+|y|^2)\quad\text{in }C^2_{\mathrm{loc}}(\mathbb{R}^2).
\end{equation*}
The bounds \eqref{heavy-density} and \eqref{heavy-gradient} hold in $B_{\varrho_n}(x_{1,n})$.
Since $\gamma_n\leq m_n$, \eqref{uniform-max-bound} implies $\Theta_n(\gamma_n)\leq\Theta_n(m_n)\leq C$. Also, \eqref{1005-4} and \eqref{1005-22} yield
\begin{equation*}
\gamma_n^{p_n-1}M_n\leq\frac{\gamma_n^{p_n-1}q_n}{c_*}\leq C,\quad \gamma_n^{p_n}\leq CM_n^{-\ell_n}.
\end{equation*}
Set $t_n(r)=\log\left(1+\frac{r^2}{\mu_n^2}\right)$. Since $\frac{\mu_n}{\varrho_n}\to0$, \eqref{1005-26} implies
\begin{equation*}
\begin{aligned}
\frac{p_n}{2}\gamma_n^{p_n}-t_n(\varrho_n)
&=-\log\Theta_n(\gamma_n)+2\log\frac{1}{\varrho_n}+\log\frac{8}{p_n^2h(x_{1,n})}-\log\left(1+\frac{\mu_n^2}{\varrho_n^2}\right)\\
&\geq\frac{2\pi\kappa_0^{\frac{1}{p_n-1}}}{M_n^{\ell_n}}-C\geq cM_n^{-\ell_n}\geq c'\frac{p_n}{2}\gamma_n^{p_n}.
\end{aligned}
\end{equation*}
After decreasing $c'>0$ if necessary, set $\eta_0=1-c'\in(0,1)$. Then
\begin{equation*}
t_n(\varrho_n)\leq\eta_0\frac{p_n}{2}\gamma_n^{p_n}.
\end{equation*}
By \eqref{1005-24}, $u_n\leq2\gamma_n$ in $B_{\varrho_n}(x_{1,n})$ for sufficiently large $n$. Lemma \ref{lem-radial} therefore yields
\begin{equation} \label{DT-radial-comparison}
\left|\frac{p_n}{2}\gamma_n^{p_n-1}\overline u_n(r)-\left(\frac{p_n}{2}\gamma_n^{p_n}-t_n(r)\right)\right|\leq C\quad\text{for }0\leq r\leq\varrho_n.
\end{equation}

Set $r_n=2\delta_n=2s_n^2$ and $R_n=\frac{s_n}{256}$. Since $x_{1,n}\in D_n$ and $\operatorname{diam}D_n\leq\delta_n$,
\begin{equation*}
D_n\subset B_{r_n}(x_{1,n}),\quad \frac{r_n}{R_n}=512s_n\to0,\quad r_n<R_n<\varrho_n
\end{equation*}
for sufficiently large $n$. Integrating the circle mean identity \eqref{derivate}, we obtain
\begin{equation*}
\begin{aligned}
\frac{p_n}{2}\gamma_n^{p_n-1}\left(\overline u_n(r_n)-\overline u_n(R_n)\right)
&=\frac{p_n\gamma_n^{p_n-1}}{4\pi}\int_{r_n}^{R_n}\frac{1}{\tau}\left(\int_{B_\tau(x_{1,n})}f_n\mathrm{d}x\right)\mathrm{d}\tau\\
&\geq\frac{p_n\gamma_n^{p_n-1}}{4\pi}\int_{r_n}^{R_n}\frac{1}{\tau}\left(\int_{B_{r_n}(x_{1,n})}f_n\mathrm{d}x\right)\mathrm{d}\tau\\
&\geq\frac{p_n\gamma_n^{p_n-1}q_n}{4\pi}\log\frac{R_n}{r_n}.
\end{aligned}
\end{equation*}
On the other hand, \eqref{DT-radial-comparison} gives
\begin{equation*}
\begin{aligned}
\frac{p_n}{2}\gamma_n^{p_n-1}\left(\overline u_n(r_n)-\overline u_n(R_n)\right)
&\leq t_n(R_n)-t_n(r_n)+C =\log\frac{\mu_n^2+R_n^2}{\mu_n^2+r_n^2}+C\leq2\log\frac{R_n}{r_n}+C.
\end{aligned}
\end{equation*}
For sufficiently large $n$,
\begin{equation*}
\log\frac{R_n}{r_n}=\frac{\pi\kappa_0^{\frac{1}{p_n-1}}}{M_n^{\ell_n}}-\log(512\sqrt{C_0})\geq\frac{c}{M_n^{\ell_n}}>0.
\end{equation*}
Comparing the two estimates, we obtain
\begin{equation} \label{gammaq-upper-M2}
\frac{p_n}{2}\gamma_n^{p_n-1}q_n\leq4\pi+\frac{C}{\log(R_n/r_n)}\leq4\pi+CM_n^{\ell_n}.
\end{equation}

Write $u_n=v_n+h_{D_n}$ as in \eqref{1005-7}. Since $x_{1,n}\in D_n$ and $\operatorname{diam}D_n\leq\delta_n$, for every $y\in D_n$, we have
$
|y-x_{1,n}|\leq\delta_n<R_n.
$
For $|z|<1$, the function
\begin{equation*}
z\mapsto\frac{1}{2\pi}\int_0^{2\pi}\log|(\cos\theta,\sin\theta)-z|\mathrm{d}\theta
\end{equation*}
is harmonic and radial, and equals its value $0$ at $z=0$. Therefore,
\begin{equation*}
\begin{aligned}
&\frac{1}{2\pi}\int_0^{2\pi}\log\frac{1}{|x_{1,n}+R_n(\cos\theta,\sin\theta)-y|}\mathrm{d}\theta\\
&=\log\frac{1}{R_n}-\frac{1}{2\pi}\int_0^{2\pi}\log\left|(\cos\theta,\sin\theta)-\frac{y-x_{1,n}}{R_n}\right|\mathrm{d}\theta
=\log\frac{1}{R_n}.
\end{aligned}
\end{equation*}
Since $H(\cdot,y)$ is harmonic in $B_{R_n}(x_{1,n})$, the mean-value property gives
\begin{equation*}
\frac{1}{2\pi}\int_0^{2\pi}H(x_{1,n}+R_n(\cos\theta,\sin\theta),y)\mathrm{d}\theta=H(x_{1,n},y).
\end{equation*}
Using Fubini's theorem and the boundedness of $H$ on $\mathcal{K}\times\mathcal{K}$, we obtain
\begin{equation*}
\begin{aligned}
\overline v_n(R_n)&=\int_{D_n}\frac{1}{2\pi}\int_0^{2\pi}G(x_{1,n}+R_n(\cos\theta,\sin\theta),y)\mathrm{d}\theta f_n(y)\mathrm{d}y\\
&=\frac{q_n}{2\pi}\log\frac{1}{R_n}+\int_{D_n}H(x_{1,n},y)f_n(y)\mathrm{d}y
=\frac{q_n}{2\pi}\log\frac{1}{R_n}+O(q_n).
\end{aligned}
\end{equation*}
Moreover,
\begin{equation*}
|x_{1,n}-a_n|+R_n\leq s_n^2+\frac{s_n}{256}<s_n
\end{equation*}
for sufficiently large $n$, so $\partial B_{R_n}(x_{1,n})\subset B_{s_n}(a_n)$. By \eqref{1005-8}, \eqref{1005-4} and \eqref{1005-22},
\begin{equation*}
0\leq\overline{h_{D_n}}(R_n)\leq CM_n,\quad \gamma_n^{p_n-1}q_n+\gamma_n^{p_n-1}M_n\leq C.
\end{equation*}
Consequently,
\begin{equation} \label{single-radius-green}
\begin{aligned}
\frac{p_n}{2}\gamma_n^{p_n-1}\overline u_n(R_n)&=\frac{p_n}{2}\gamma_n^{p_n-1}\left(\overline v_n(R_n)+\overline{h_{D_n}}(R_n)\right)
=\frac{p_n\gamma_n^{p_n-1}q_n}{4\pi}\log\frac{1}{R_n}+O(1).
\end{aligned}
\end{equation}

By \eqref{1005-26},
\begin{equation*}
\log\mu_n^2=L_n-2(p_n-1)\log\gamma_n-\gamma_n^{p_n}+\log\frac{8}{p_n^2h(x_{1,n})}.
\end{equation*}
Since $R_n=\frac{s_n}{256}$, \eqref{1005-27} gives $\frac{\mu_n}{R_n}\to0$. Thus \eqref{DT-radial-comparison} yields
\begin{equation*}
\begin{aligned}
\frac{p_n}{2}\gamma_n^{p_n-1}\overline u_n(R_n)
&=\frac{p_n}{2}\gamma_n^{p_n}-t_n(R_n)+O(1)\\
&=\frac{p_n}{2}\gamma_n^{p_n}+\log\mu_n^2+2\log\frac{1}{R_n}-\log\left(1+\frac{\mu_n^2}{R_n^2}\right)+O(1)\\
&=L_n-\frac{2-p_n}{2}\gamma_n^{p_n}-2(p_n-1)\log\gamma_n+2\log\frac{1}{R_n}\\
&\quad+\log\frac{8}{p_n^2h(x_{1,n})}-\log\left(1+\frac{\mu_n^2}{R_n^2}\right)+O(1)\\
&=-\log\Theta_n(\gamma_n)+2\log\frac{1}{R_n}+O(1).
\end{aligned}
\end{equation*}
Comparing with \eqref{single-radius-green} and using \eqref{gammaq-upper-M2}, we obtain
\begin{equation*}
\begin{aligned}
-\log\Theta_n(\gamma_n)
&=\left(\frac{p_n\gamma_n^{p_n-1}q_n}{4\pi}-2\right)\log\frac{1}{R_n}+O(1)\\
&\leq CM_n^{\ell_n}\log\frac{1}{R_n}+C =CM_n^{\ell_n}\left(\frac{\pi\kappa_0^{\frac{1}{p_n-1}}}{M_n^{\ell_n}}+\log\frac{256}{\sqrt{C_0}}\right)+C\leq C.
\end{aligned}
\end{equation*}
Therefore, we have
$\Theta_n(\gamma_n)\geq\mathrm{e}^{-C}>0$.
This proves \eqref{heavy-height-lower}.
\end{proof}

\begin{remark}
The upper bound in Remark \ref{rem-fixed-radius} follows from positivity of the circle mean at a fixed radius. For the lower bound, we also control the contribution from outside $D_n$ after multiplication by $\gamma_n^{p_n-1}$. By  \eqref{1005-4}, \eqref{1005-8} and \eqref{1005-22},
\begin{equation*}
0\leq\gamma_n^{p_n-1}\overline{h_{D_n}}(R_n)\leq C\gamma_n^{p_n-1}M_n\leq C\gamma_n^{p_n-1}q_n\leq C.
\end{equation*}
Since $R_n\to0$, the comparison also requires the quantitative estimate \eqref{gammaq-upper-M2}. The limit $Z_n(x_{1,n})\to4\pi$ alone does not bound $(Z_n(x_{1,n})-4\pi)\log\frac{1}{R_n}$ from above.
\end{remark}

We now complete the proof of Theorem \ref{thm-simply-connected} by comparing \eqref{1005-2} at the global maximum with \eqref{heavy-height-lower} at the maximum of the selected component.

\begin{proof}[\bf{Proof of Theorem \ref{thm-simply-connected}}]
Suppose that the conclusion fails. Then there exist positive solutions $u_n$ of \eqref{1}, with $p_n\in[1,2]$ and $\lambda_n>0$, such that $\beta_n\to\infty$. By Lemma \ref{lem-unbounded-sequence}, after passing to a subsequence,
\begin{equation*}
p_n\to2,\quad \lambda_n\to0,\quad M_n\to0,\quad \frac{L_n}{m_n^{p_n}}\to0,\quad M_nm_n^{p_n-1}\to\infty.
\end{equation*}
Fix $\kappa_0=\pi$. For sufficiently large $n$, choose a component $D_n$ of $\{u_n>\left(\frac{\kappa_0}{M_n}\right)^{\frac{1}{p_n-1}}\}$ with maximal source mass and set
\begin{equation*}
q_n=\int_{D_n}f_n\mathrm{d}x,\quad \gamma_n=\max_{\overline{D_n}}u_n\leq m_n.
\end{equation*}
Proposition \ref{prop-fixed-heavy} yields $q_n\geq c_*M_n$. By \eqref{1005-2},
\begin{equation*}
\Theta_n(m_n)\leq C\mathrm{e}^{-cM_nm_n^{p_n-1}}\to0.
\end{equation*}
Proposition \ref{prop-final-height} yields $\Theta_n(\gamma_n)\geq c>0$. Since $\Theta_n$ is nondecreasing on $(0,\infty)$ for $p_n\in[1,2]$,
\begin{equation*}
0<c\leq\Theta_n(\gamma_n)\leq\Theta_n(m_n)\to0,
\end{equation*}
a contradiction.

\vskip 0.1cm

To verify the stated dependence of $C_{\Omega,h}$, fix $\kappa_0=\pi$ and consider weights satisfying
\begin{equation*}
h\geq h_0>0\quad\text{on }\overline\Omega,\quad \|h\|_{C^2(\overline\Omega)}\leq H_0,
\end{equation*}
where $h_0,H_0>0$ are fixed. Then
\begin{equation*}
\|\nabla\log h\|_{L^\infty(\Omega)}\leq\frac{\|\nabla h\|_{L^\infty(\Omega)}}{\min_{\overline\Omega}h}\leq\frac{H_0}{h_0}.
\end{equation*}
Thus $\delta_*$, $c_*$ and the constants in \eqref{1005-2} and \eqref{heavy-height-lower} can be chosen depending only on $\Omega$, $h_0$ and $H_0$. A sequence of weights $h_n$ satisfying these bounds and corresponding solutions with $\beta_n\to\infty$ would yield the same contradiction. This proves the stated dependence of $C_{\Omega,h}$.
\end{proof}

\section{Convex domains and constrained critical points} \label{app-remaining}

In this section, we prove the energy limit on convex domains and the nonexistence of positive constrained critical points at sufficiently large energies.

\subsection{Energy limit on convex domains} \

\vskip 0.1cm

In this subsection, we take $h\equiv1$ and $p=2$ and prove Theorem \ref{thm-convex}. Set $R(x)=H(x,x)$. Since $H$ is symmetric, we have
\begin{equation*}
\nabla R(x)=2\nabla_xH(x,x).
\end{equation*}

\begin{proof}[\bf{Proof of Theorem \ref{thm-convex}}]
Suppose that the conclusion fails. Then there exist a constant $\varepsilon_0>0$ and a sequence of positive solutions $u_n$ of \eqref{1}, with parameters $\lambda_n\to0$, such that
\begin{equation*}
|E(u_n)-4\pi|\geq\varepsilon_0\quad\text{for every }n.
\end{equation*}
Theorem \ref{thm-simply-connected} and \eqref{lambda-m-critical} imply
\begin{equation*}
\sup_nE(u_n)<\infty,\quad m_n^2\geq\log\frac{\lambda_1}{2\lambda_n}\to\infty.
\end{equation*}
Apply Theorem 1.2 in \cite{DruetThizy2020} with parameter $2\lambda_n$ and weight equal to $1$. After passing to a subsequence, there exist an integer $N\geq1$, distinct points $x_1,\ldots,x_N\in\Omega$ and weights $\sigma_1,\ldots,\sigma_N>0$ such that
\begin{equation} \label{DT-energy-limit}
E(u_n)\to4\pi N,
\end{equation}
and
\begin{equation*}
\sigma_i\nabla_xH(x_i,x_i)+\sum_{j\neq i}\sigma_j\nabla_xG(x_i,x_j)=0\quad\text{for }i=1,\cdots,N.
\end{equation*}
Multiplying the $i$-th location equation by $-\sigma_i$, we obtain
\begin{equation*}
\frac{\sigma_i^2}{2}\nabla(-R)(x_i)-\sum_{j\neq i}\sigma_i\sigma_j\nabla_xG(x_i,x_j)=0\quad\text{for }i=1,\cdots,N.
\end{equation*}
By \cite{GrossiTakahashi2010}, the location equations have no solution with $N\geq2$. Thus $N=1$, and \eqref{DT-energy-limit} implies $E(u_n)\to4\pi$, contradicting $|E(u_n)-4\pi|\geq\varepsilon_0$.
\end{proof}

\subsection{Variational consequence} \label{app-variational}\

\vskip 0.1cm

In this subsection, we prove Corollary \ref{cor-bound} by showing that every positive critical point of $J|_{\mathcal{M}_\beta}$ is a positive solution of \eqref{1} with $p=2$ for some $\lambda>0$.

\vskip 0.1cm

By Lemma 2.1 in \cite{Struwe1988}, the functional $J$ is of class $C^1$ in $H_0^1(\Omega)$, with
\begin{equation*}
DJ(v)[\varphi]=2\int_\Omega hv\mathrm{e}^{v^2}\varphi\mathrm{d}x.
\end{equation*}
We also have
\begin{equation} \label{A-exp}
\text{$\mathrm{e}^{av^2}\in L^1(\Omega)$ for every $v\in H_0^1(\Omega)$ and $a>0$. }
\end{equation}
Indeed, by the density of $C_c^\infty(\Omega)$ in $H_0^1(\Omega)$, choose $\psi\in C_c^\infty(\Omega)$ such that
\begin{equation*}
2a\int_\Omega|\nabla(v-\psi)|^2\mathrm{d}x<4\pi.
\end{equation*}
Since $v^2\leq2(v-\psi)^2+2\psi^2$, the Moser-Trudinger inequality \cite{Trudinger1967,Moser1971} implies
\begin{equation*}
\int_\Omega\mathrm{e}^{av^2}\mathrm{d}x\leq\mathrm{e}^{2a\|\psi\|_{L^\infty(\Omega)}^2}\int_\Omega\mathrm{e}^{2a(v-\psi)^2}\mathrm{d}x<\infty.
\end{equation*}

\begin{proof}[\bf{Proof of Corollary \ref{cor-bound}}]
Let $u$ be a positive critical point of $J|_{\mathcal{M}_\beta}$. The Lagrange multiplier theorem yields $\lambda\in\mathbb{R}$ such that
\begin{equation*}
\int_\Omega\nabla u\cdot\nabla\varphi\mathrm{d}x=2\lambda\int_\Omega hu\mathrm{e}^{u^2}\varphi\mathrm{d}x \quad\text{for every }\varphi\in H_0^1(\Omega).
\end{equation*}
Taking $\varphi=u$ yields
\begin{equation*}
\lambda=\frac{\beta}{2\int_\Omega hu^2\mathrm{e}^{u^2}\mathrm{d}x}>0.
\end{equation*}

Fix $r>2$. Taking $v=u$ and $a=r+1$ in \eqref{A-exp} and using $t^r\mathrm{e}^{-t^2}\leq C_r$ for $t\geq0$, we obtain
\begin{equation*}
\int_\Omega|u\mathrm{e}^{u^2}|^r\mathrm{d}x=\int_\Omega|u|^r\mathrm{e}^{ru^2}\mathrm{d}x\leq C_r\int_\Omega\mathrm{e}^{(r+1)u^2}\mathrm{d}x<\infty.
\end{equation*}
Thus $hu\mathrm{e}^{u^2}\in L^r(\Omega)$. By the $L^r$ theory and the Sobolev embedding theorem,
\begin{equation*}
u\in W^{2,r}(\Omega)\subset C^{1,\alpha}(\overline\Omega)\quad\text{for }0<\alpha<1-\frac{2}{r}.
\end{equation*}
Hence $hu\mathrm{e}^{u^2}\in C^{0,\alpha}(\overline\Omega)$, and the Schauder theory implies $u\in C^{2,\alpha}(\overline\Omega)$. Since $u\geq0$ and $u\not\equiv0$, the strong maximum principle yields
\begin{equation*}
u>0 \quad \text{in } \Omega.
\end{equation*}
It follows by Theorem \ref{thm-simply-connected} that
\begin{equation*}
\beta=E(u)\leq C_{\Omega,h}.
\end{equation*}
Taking $\beta_*=\max\{C_{\Omega,h},4\pi\}+1$, we obtain $\beta\leq C_{\Omega,h}<\beta_*$, which proves Corollary \ref{cor-bound}.
\end{proof}

\appendix
\section{Boundary exclusion} \label{app-boundary}

In this section, we prove Lemma \ref{lem-interior} by the moving-plane and Kelvin-transform argument of \cite{deFigueiredoLionsNussbaum1982,Han1991}, as used in \cite{MalchiodiMartinazziThizy2024}.

\subsection{The reflected domain}\

\vskip 0.1cm

In this subsection, we use Kelvin inversion about an exterior tangent circle to construct a domain suitable for the moving-plane argument.

\vskip 0.1cm

Let $\nu$ be the outward unit normal to $\partial\Omega$. By Lemma 14.16 in \cite{GilbargTrudinger2001}, there exists $\rho_\Omega>0$ such that
\begin{equation*}
\overline{B_r(P+r\nu(P))}\cap\overline\Omega=\{P\}
\end{equation*}
for every $P\in\partial\Omega$ and $0<r<\rho_\Omega$. For every $P\in\partial\Omega$ and $0<s<\rho_\Omega$, the point $P-s\nu(P)$ belongs to $\Omega$ and has distance $s$ from $\partial\Omega$. Conversely, every $x\in\Omega$ with $\operatorname{dist}(x,\partial\Omega)<\rho_\Omega$ can be uniquely written as
\begin{equation} \label{local-representation}
x=P-s\nu(P),\quad P\in\partial\Omega,\quad s=\operatorname{dist}(x,\partial\Omega).
\end{equation}
Fix $P\in\partial\Omega$ and choose coordinates such that $P=0$ and $\nu(P)=e_2=(0,1)$. By the smoothness and compactness of $\partial\Omega$, there exist $b_0,C_0>0$, independent of $P$, and $\psi\in C^\infty([-b_0,b_0])$ such that
\begin{equation} \label{local}
\Omega\cap(-b_0,b_0)^2=\big\{x\in(-b_0,b_0)^2\mid x_2<\psi(x_1)\big\},
\end{equation}
where $\psi(0)=\psi'(0)=0$ and $\|\psi\|_{C^2([-b_0,b_0])}\leq C_0$.

\vskip 0.1cm

Choose
\begin{equation*}
r:=\min\left\{\frac{\rho_\Omega}{8},\frac{1}{2(1+\|\nabla\log h\|_{L^\infty(\Omega)})}\right\}.
\end{equation*}
Then $\overline{B_r(re_2)}\cap\overline\Omega=\{0\}$. Set
\begin{equation*}
K(y)=re_2+\frac{r^2(y-re_2)}{|y-re_2|^2},\quad \Sigma=K(\Omega),\quad v=u\circ K.
\end{equation*}
Therefore, we have $K\circ K=\mathrm{Id}$ and $u=v\circ K$. Since $|K(x)-re_2|=\frac{r^2}{|x-re_2|}$, we have $\Sigma\subset B_r(re_2)$ and
\begin{equation*}
\operatorname{dist}(re_2,\overline \Sigma)\geq\frac{r^2}{\operatorname{diam}\Omega+r}>0,\quad \overline \Sigma\cap\{y_2=0\}=\{0\}.
\end{equation*}

For $t>0$, write
\begin{equation*}
\Sigma_t=\Sigma\cap\{y_2<t\},\quad y^t=(y_1,2t-y_2).
\end{equation*}
For $y\in\Sigma_t$, we write $[y,y^t]:=\big\{(y_1,s)\mid y_2\leq s\leq2t-y_2\big\}$.

\begin{lemma} \label{lem-boundary-cap}
There exist $b,t_0>0$, depending only on $\Omega$ and $\|\nabla\log h\|_{L^\infty(\Omega)}$, and a smooth function $q$ such that
\begin{equation} \label{B-graph}
\begin{aligned}
\Sigma\cap(-b,b)^2&=\{y\in(-b,b)^2\mid y_2>q(y_1)\},\quad
q(0)=q'(0)=0,\quad |q|\leq\frac{b}{4}\quad\text{on }[-b,b].
\end{aligned}
\end{equation}
For every $0<t<t_0$,
\begin{equation} \label{B-cap}
[y,y^t]\subset \Sigma\cap\left(-\frac{b}{2},\frac{b}{2}\right)^2\quad\text{for every }y\in\Sigma_t.
\end{equation}
Moreover, $y^t\in \Sigma$ whenever $y\in\overline{\Sigma_t}$ and $y_2<t$.
\end{lemma}

\begin{proof}
For $|y|<\frac{r}{2}$, the definition of $K$ gives
\begin{equation*}
\|DK(y)\|\leq4,\quad \|D^2K(y)\|\leq\frac{C}{r}.
\end{equation*}
Choose $0<\delta<\min\left\{\frac{r}{4},\frac{b_0}{8}\right\}$. Since $K(0)=0$,
\begin{equation*}
|K(y)|\leq4|y|<4\sqrt{2}\delta<b_0
\quad\text{for }y\in(-\delta,\delta)^2.
\end{equation*}
Thus $K((-\delta,\delta)^2)\subset(-b_0,b_0)^2$. Write $K=(K_1,K_2)$ and set
\begin{equation*}
G(y)=\psi(K_1(y))-K_2(y).
\end{equation*}
Since $\psi(0)=\psi'(0)=0$, $K(0)=0$ and
\begin{equation*}
DK(0)=\begin{pmatrix}1&0\\0&-1\end{pmatrix},
\end{equation*}
we have
\begin{equation*}
G(0)=0,\quad \partial_{y_1}G(0)=0,\quad \partial_{y_2}G(0)=1.
\end{equation*}
By the uniform $C^2$ bound for $\psi$, after decreasing $\delta$ there exists $C>0$, independent of $P$, such that
\begin{equation*}
|D^2G(y)|\leq C,\quad \partial_{y_2}G(y)\geq\frac{1}{2},\quad |G(s,0)|\leq Cs^2
\end{equation*}
for $y\in(-\delta,\delta)^2$ and $|s|<\delta$. Choose $b\in(0,\delta)$ such that $8Cb<1$. For $|s|\leq b$,
\begin{equation*}
G\left(s,-\frac{b}{4}\right)\leq Cb^2-\frac{b}{8}<0<-Cb^2+\frac{b}{8}
\leq G\left(s,\frac{b}{4}\right).
\end{equation*}
Since $\partial_{y_2}G\geq\frac{1}{2}$, there exists a unique $q(s)\in(-\frac{b}{4},\frac{b}{4})$ satisfying $G(s,q(s))=0$. The implicit function theorem shows that $q$ is smooth and
\begin{equation*}
q'(s)=-\frac{\partial_{y_1}G(s,q(s))}{\partial_{y_2}G(s,q(s))},\quad q(0)=q'(0)=0.
\end{equation*}
For $y\in(-b,b)^2$, the local representation \eqref{local} and the strict monotonicity of $G$ in $y_2$ imply
\begin{equation*}
y\in\Sigma\iff K(y)\in\Omega\iff G(y)>0\iff y_2>q(y_1).
\end{equation*}
This proves \eqref{B-graph}.

\vskip 0.1cm

Choose $t_0>0$ such that
\begin{equation*}
2t_0<\frac{b}{2},\quad 2rt_0<\frac{b^2}{16}.
\end{equation*}
For $0<t<t_0$ and $y\in\Sigma_t$, the relation $\Sigma\subset B_r(re_2)$ gives
\begin{equation*}
0<y_2<t,\quad |y_1|^2<2ry_2-y_2^2<2rt_0<\frac{b^2}{16}.
\end{equation*}
Thus $|y_1|<\frac{b}{4}$. Moreover, by \eqref{B-graph}, we get
\begin{equation*}
q(y_1)<y_2<t<2t-y_2<2t_0<\frac{b}{2}.
\end{equation*}
Therefore, for every $(y_1,s)\in[y,y^t]$, we have
\begin{equation*}
|y_1|<\frac{b}{2},\quad q(y_1)<y_2\leq s\leq2t-y_2<\frac{b}{2}.
\end{equation*}
It follows from \eqref{B-graph} that
\begin{equation*}
[y,y^t]\subset \Sigma\cap\left(-\frac{b}{2},\frac{b}{2}\right)^2,
\end{equation*}
which proves \eqref{B-cap}.

\vskip 0.1cm
Finally, let $y\in\overline{\Sigma_t}$ with $y_2<t$. Since $\overline{\Sigma_t}\subset\overline{B_r(re_2)}\cap\{0\leq y_2\leq t\}$, we have
\begin{equation*}
|y_1|^2\leq2ry_2-y_2^2\leq2rt<\frac{b^2}{16}.
\end{equation*}
In particular, $|y_1|<\frac{b}{4}$. By \eqref{B-graph} and $y\in\overline \Sigma$, we obtain $q(y_1)\leq y_2$. Consequently,
\begin{equation*}
q(y_1)\leq y_2<2t-y_2<2t_0<\frac{b}{2}.
\end{equation*}
Thus $y^t\in \Sigma$ by \eqref{B-graph}.
\end{proof}

Reducing $t_0$ if necessary, we may assume that $2t_0<\frac{r}{2}$. Direct differentiation yields
\begin{equation*}
\begin{aligned}
&DK(y)=\frac{r^2}{|y-re_2|^2}\left(I_2-2\frac{(y-re_2)\otimes(y-re_2)}{|y-re_2|^2}\right),\\
&DK(y)DK(y)^{T}=\frac{r^4}{|y-re_2|^4}I_2,\quad \Delta K_i(y)=0\quad\text{for }i=1,2.
\end{aligned}
\end{equation*}
Set
$c(y)=\frac{r^4h(K(y))}{|y-re_2|^4}$.
Then the chain rule yields
\begin{equation*}
\begin{cases}
-\Delta v=c(y)F'(v)&\text{in }\Sigma,\\[1mm]
v=0&\text{on }\partial\Sigma.
\end{cases}
\end{equation*}
For $y\in\Sigma$ with $0<y_2<2t_0<\frac{r}{2}$, we have $|\partial_{y_2}K(y)|=\frac{r^2}{|y-re_2|^2}$. By the choice of $r$,
\begin{equation*}
\begin{aligned}
\frac{\partial_{y_2}c(y)}{c(y)}&=\frac{4(r-y_2)}{|y-re_2|^2}+\nabla\log h(K(y))\cdot\partial_{y_2}K(y)\\
&\geq\frac{4(r-y_2)-r^2\|\nabla\log h\|_{L^\infty(\Omega)}}{|y-re_2|^2}
\geq\frac{3r}{2|y-re_2|^2}>0.
\end{aligned}
\end{equation*}
Integrating along the segment in \eqref{B-cap}, we obtain
\begin{equation} \label{B-coeff}
c(y^t)-c(y)=\int_{y_2}^{2t-y_2}\partial_{y_2}c(y_1,s)\mathrm{d}s>0\quad\text{for }y\in\Sigma_t.
\end{equation}

\subsection{The reflected comparison}\

\vskip 0.1cm

In this subsection, we establish the reflection inequality and return to the original coordinates to prove Lemma \ref{lem-interior}. For $s>0$,
\begin{equation*}
F''(s)=p\lambda\mathrm{e}^{s^p}\left((p-1)s^{p-2}+ps^{2p-2}\right)>0.
\end{equation*}
Thus $F'$ is continuous and strictly increasing on $[0,\infty)$, with $F'(0)=\lambda$ if $p=1$ and $F'(0)=0$ if $p>1$.

\vskip 0.1cm
Let $\nu_\Sigma$ be the outward unit normal to $\partial\Sigma$. Since
\begin{equation*}
\begin{cases}
-\Delta v=c(y)F'(v)>0&\text{in }\Sigma,\\[1mm]
v=0&\text{on }\partial\Sigma,
\end{cases}
\end{equation*}
Hopf's lemma gives $\partial_{\nu_\Sigma}v<0$ on $\partial\Sigma$. On the graph in \eqref{B-graph}, differentiating $v(s,q(s))=0$ yields
\begin{equation*}
\partial_{y_1}v(s,q(s))+q'(s)\partial_{y_2}v(s,q(s))=0.
\end{equation*}
Since $\nu_\Sigma=\frac{(q',-1)}{\sqrt{1+(q')^2}}$, we obtain
\begin{equation} \label{B-normal}
\partial_{y_2}v(s,q(s))=-\frac{\partial_{\nu_\Sigma}v(s,q(s))}{\sqrt{1+(q'(s))^2}}>0.
\end{equation}

\begin{remark}
For $1<p<2$,
\begin{equation*}
F''(s)\sim p\lambda(p-1)s^{p-2}\to+\infty \quad\text{as }s\to0.
\end{equation*}
The Kelvin radius is chosen so that the transformed coefficient is strictly increasing in the direction of reflection. We use Hopf's lemma to start the comparison and the monotonicity of $F'$ to continue it, without estimating $F''$ near zero.
\end{remark}

\begin{proof}[\bf{Proof of Lemma \ref{lem-interior}}]
For $0<t<t_0$, set $w_t(y)=v(y^t)-v(y)$. By Lemma \ref{lem-boundary-cap},
\begin{equation*}
\begin{cases}
-\Delta w_t=c(y^t)F'(v(y^t))-c(y)F'(v(y))&\text{in }\Sigma_t,\\[1mm]
w_t=0&\text{on }\overline{\Sigma_t}\cap\{y_2=t\},\\[1mm]
w_t=v(y^t)>0&\text{on }\partial\Sigma\cap\{y_2<t\}.
\end{cases}
\end{equation*}
If $w_t\geq0$ in $\Sigma_t$, the monotonicity of $F'$ and \eqref{B-coeff} give
\begin{equation} \label{B-source}
\begin{aligned}
-\Delta w_t&=(c(y^t)-c(y))F'(v(y^t))+c(y)(F'(v(y^t))-F'(v(y)))\\
&\geq(c(y^t)-c(y))F'(v(y^t))>0\quad\text{in }\Sigma_t.
\end{aligned}
\end{equation}

By \eqref{B-normal}, $\partial_{y_2}v(0)>0$. Since $v\in C^1(\overline\Sigma)$, there exists $\eta>0$ such that $\partial_{y_2}v>0$ in $\overline\Sigma\cap B_\eta(0)$. For $y\in\Sigma_t$, the relation $\Sigma\subset B_r(re_2)$ gives
\begin{equation*}
|y_1|^2<2ry_2-y_2^2<2rt.
\end{equation*}
Hence every $z\in[y,y^t]$ satisfies $|z|^2<2rt+4t^2$. For all sufficiently small $t>0$, \eqref{B-cap} yields
\begin{equation} \label{B-seg}
w_t(y)=\int_{y_2}^{2t-y_2}\partial_{y_2}v(y_1,s)\mathrm{d}s>0\quad\text{for }y\in\Sigma_t.
\end{equation}

Set
\begin{equation*}
T=\sup\{s\in(0,t_0)\mid w_t\geq0\text{ in }\Sigma_t\text{ for every }0<t\leq s\}.
\end{equation*}
By \eqref{B-seg}, $T>0$. Suppose that $T<t_0$. For every $y\in\Sigma_T$, we have $y\in\Sigma_t$ whenever $y_2<t<T$. Letting $t\to T$ gives
\begin{equation*}
w_T(y)=\lim_{t\to T}w_t(y)\geq0.
\end{equation*}
If $w_T(y)=0$ at an interior point of $\Sigma_T$, then $-\Delta w_T(y)\leq0$, contradicting \eqref{B-source}. Thus $w_T>0$ in $\Sigma_T$. Moreover, for $y\in\partial\Sigma$ with $y_2<T$, Lemma \ref{lem-boundary-cap} gives $y^T\in\Sigma$. Hence
\begin{equation*}
w_T(y)=v(y^T)>0.
\end{equation*}

We next show that
\begin{equation} \label{B-top}
\partial_{y_2}v>0\quad\text{on }\overline\Sigma\cap\{y_2=T\}.
\end{equation}
Let $y\in\Sigma\cap\{y_2=T\}$ and choose $0<R<\frac{1}{2}\operatorname{dist}(y,\partial\Sigma)$. For every $z\in B_R(y-Re_2)$, we have
\begin{equation*}
z_2<T,\quad |z-y|\leq|z-(y-Re_2)|+R<2R<\operatorname{dist}(y,\partial\Sigma).
\end{equation*}
Hence $B_R(y-Re_2)\subset\Sigma_T$, and this ball is tangent to $\{y_2=T\}$ at $y$. Since
\begin{equation*}
\begin{cases}
-\Delta w_T>0&\text{in }\Sigma_T,\\[1mm]
w_T>0&\text{in }\Sigma_T,\\[1mm]
w_T=0&\text{on }\Sigma\cap\{y_2=T\},
\end{cases}
\end{equation*}
Hopf's lemma yields
\begin{equation*}
0<-\partial_{y_2}w_T(y)=2\partial_{y_2}v(y).
\end{equation*}
If $y\in\partial\Sigma\cap\{y_2=T\}$, choose $t\in(T,t_0)$. Then $y\in\overline{\Sigma_t}$. Approximating $y$ by points of $\Sigma_t$ in \eqref{B-cap}, we obtain
\begin{equation*}
y\in\left[-\frac{b}{2},\frac{b}{2}\right]^2\subset(-b,b)^2.
\end{equation*}
Thus $y$ belongs to the graph in \eqref{B-graph}, and \eqref{B-normal} gives $\partial_{y_2}v(y)>0$. This proves \eqref{B-top}.

\vskip 0.1cm

By the definition of $T$, for all sufficiently large $j$, there exist $t_j$ and $y_j$ satisfying
\begin{equation*}
T<t_j<T+\frac{1}{j}<t_0,\quad y_j\in\Sigma_{t_j},\quad w_{t_j}(y_j)<0.
\end{equation*}
After passing to a subsequence, we may assume that
\begin{equation*}
y_j\to y\in\overline\Sigma,\quad y_2\leq T,\quad y_j^{t_j}\to y^T.
\end{equation*}
If $y_2<T$, then $w_T(y)>0$. Indeed, this follows from \eqref{B-source} when $y\in\Sigma$, and from Lemma \ref{lem-boundary-cap} when $y\in\partial\Sigma$. Consequently,
\begin{equation*}
w_{t_j}(y_j)\to w_T(y)>0,
\end{equation*}
contrary to $w_{t_j}(y_j)<0$.
If $y_2=T$, then $y^T=y$. Since $t_j\to T$ and $y_j\to y$, we also have $y_j^{t_j}\to y$. Moreover,
\begin{equation*}
\sup_{z\in[y_j,y_j^{t_j}]}|z-y|
\leq\max\big\{|y_j-y|,|y_j^{t_j}-y|\big\}\to0.
\end{equation*}
By \eqref{B-top}, we have $\partial_{y_2}v(y)>0$. The segments $[y_j,y_j^{t_j}]$ lie in $\Sigma$ by \eqref{B-cap}. The continuity of $\partial_{y_2}v$ on $\overline\Sigma$ therefore implies that, for all sufficiently large $j$,
\begin{equation*}
\partial_{y_2}v(z)\geq\frac{1}{2}\partial_{y_2}v(y)>0\quad\text{for }z\in[y_j,y_j^{t_j}].
\end{equation*}
Therefore,
\begin{equation*}
w_{t_j}(y_j)=\int_{(y_j)_2}^{2t_j-(y_j)_2}\partial_{y_2}v((y_j)_1,s)\mathrm{d}s\geq(t_j-(y_j)_2)\partial_{y_2}v(y)>0,
\end{equation*}
which is again a contradiction. Hence $T=t_0$.

\vskip 0.1cm
For every $0<t<t_0$, we have $w_t\geq0$ in $\Sigma_t$. By \eqref{B-source}, an interior zero of $w_t$ is impossible. Thus
\begin{equation*}
\begin{cases}
-\Delta w_t>0&\text{in }\Sigma_t,\\[1mm]
w_t>0&\text{in }\Sigma_t,\\[1mm]
w_t=0&\text{on }\Sigma\cap\{y_2=t\}.
\end{cases}
\end{equation*}
For $y\in\Sigma\cap\{y_2=t\}$, Hopf's lemma yields
\begin{equation*}
0<-\partial_{y_2}w_t(y)=2\partial_{y_2}v(y).
\end{equation*}
Since $t\in(0,t_0)$ is arbitrary, we conclude that
\begin{equation*}
\partial_{y_2}v>0\quad\text{in }\Sigma\cap\{0<y_2<t_0\}.
\end{equation*}

Choose $0<\delta_*<\min\{\rho_\Omega,t_0\}$. For $0<s\leq\delta_*$, we have $-se_2\in\Omega$ by \eqref{local-representation} and
\begin{equation*}
K(-se_2)=\frac{rs}{r+s}e_2,\quad 0<\frac{rs}{r+s}<s<t_0.
\end{equation*}
With $u=v\circ K$, differentiating along the inward normal gives
\begin{equation*}
\frac{\mathrm{d}}{\mathrm{d}s}u(-se_2)
=\frac{r^2}{(r+s)^2}\partial_{y_2}v\left(\frac{rs}{r+s}e_2\right)>0.
\end{equation*}
Since $P\in\partial\Omega$ is arbitrary, returning to the original coordinates yields
\begin{equation*}
-\nabla u(P-s\nu(P))\cdot\nu(P)=\frac{\mathrm{d}}{\mathrm{d}s}u(P-s\nu(P))>0
\end{equation*}
for every $P\in\partial\Omega$ and $0<s\leq\delta_*$. By \eqref{local-representation}, $\nabla u(x)\neq0$ for every $x\in\Omega$ with $\operatorname{dist}(x,\partial\Omega)\leq\delta_*$.
\end{proof}

\medskip

\subsection*{Acknowledgements}\

\vskip 0.1cm

This work was supported by the National Key R\&D Program (No. 2023YFA1010002).
Hong-Ge Chen was supported by the National Natural Science Foundation of China (Grant Nos. 12671239 and 12571249).
Peng Luo and Shuangjie Peng were supported by the National Natural Science Foundation of China (Grant No. 12631005).

\vskip 0.2truein
\paragraph{\textbf{Declaration of AI-assisted tools.}}\

\vskip 0.1cm
This paper was written by the authors and was not generated by artificial intelligence. During its preparation, the authors used AI tools as auxiliary aids for exploratory discussions, preliminary proof drafting, and language editing. The authors take full responsibility for the mathematical validation, the final proof choices, and the final text.

\end{document}